\documentclass[11pt]{amsart}   
\usepackage[margin=1in]{geometry}
\usepackage{amsmath, amssymb, amsthm,hyperref} 
\hypersetup{hidelinks}
\usepackage{mathtools}
\usepackage{enumitem}
\usepackage{multicol}
\usepackage{graphicx}

\newcommand{\Rstar}{R_*}

\newtheorem{theorem}{Theorem}[section]
\newtheorem{lemma}[theorem]{Lemma}
\newtheorem{proposition}[theorem]{Proposition}
\newtheorem{corollary}[theorem]{Corollary}
\theoremstyle{definition}
\newtheorem{definition}[theorem]{Definition}
\newtheorem{example}[theorem]{Example}

\newtheorem{remark}[theorem]{Remark}

\newcommand{\qphi}[5]{{}_{#1}\phi_{#2}\!\left(#3;#4;#5\right)}

\newcommand{\R}{\mathbb{R}}
\newcommand{\C}{\mathbb{C}}

\newcommand{\Q}{\mathbb{Q}}

\numberwithin{equation}{section} 
\begin{document}

\title{On $q$-real and $q$-complex numbers}

\author{Pavel Etingof}

\address{Department of Mathematics, MIT, Cambridge, MA 02139, USA}

\begin{abstract} In \cite{MO1}, S. Morier-Genoud and V. Ovsienko introduced 
the notion of the $q$-rational number $[x]_q$, $x\in \Bbb Q$. This is a rational function of $q$ specializing to $x$ at $q=1$, obtained by $q$-deforming the continued fraction expansion of $x$. In \cite{MO2} the same authors introduced a more general notion of the $q$-real number $[x]_q$, $x\in \Bbb R$ -- a Laurent series in $q$ (with only nonnegative powers if $x\ge 0$) converging to the rational function $[x]_q$ when $x\in \Bbb Q$. There is by now quite a nice theory of such $q$-deformations, summarized in \cite{MO3}. In particular, in \cite{LMOV} it is proved that if $x\in \Bbb Q_{>1}$ then the series $[x]_q$ converges for $|q|<3-2\sqrt{2}\approx 0.17$ and conjectured that  for all $x\in \Bbb R_{>1}$ this series converges in some disk centered in the origin, with the expected common radius of convergence $R_*=\frac{3-\sqrt{5}}{2}\approx 0.38$, achieved when $x=\frac{1+\sqrt{5}}{2}$ is the golden ratio.  This was proved for rational $x$ in \cite{EGMS} using the theory of Kleinian groups. 

In this paper we (partially) prove this conjecture by showing that for all $x\in \Bbb R_{>1}$, the series $[x]_q$ converges in the disk $|q|<3-2\sqrt{2}$ to a nonvanishing holomorphic function. This is achieved by giving an expansion of $1/[x]_q$ into a $q$-adically convergent series of rational functions 
which converges absolutely and uniformly on compact sets in an explicit region $D$ containing this disk. 
We also show that this expansion converges to a positive analytic function on the interval $(-\frac{3-\sqrt{5}}{2},1)$, giving a definition of $[x]_q$ for $q$ from this interval.
 Moreover, we show that the result of \cite{EGMS} implies convergence of $[x]_q$ in a larger disk, $|q|<2-\sqrt{3}\approx 0.27$.  

At the same time, we show that there exist uncountably many 
$x$ for which $[x]_q$ has radius of convergence $\le R_*$ (conjecturedly, exactly $R_*$), contrary to what is conjectured in \cite{LMOV}. 
We also give examples of explicit computation
of $[x]_q$ for transcendental numbers $x$, for instance $x={\rm cotan}(1)$ and $x=e$.

We also prove sharp inequalities for the numerators and denominators of
$q$-rational numbers on the unit circle and determine the closure of the
set of their values when $q$ is not a root of unity.

We also show that coefficientwise reduction modulo every integer $m\ge2$
is injective on the Cantor line (the extended real line with doubled-up rationals and Cantor set topology). 
For $m=2$ it identifies the Cantor line with $\mathbb P^1(\mathbb F_2((q)))$.  We give an explicit inverse
algorithm, extend the rationality results to reduction modulo every prime,
characterize the quadratic series over $\mathbb F_2(q)$ with eventually
periodic decoded continued fraction, give criteria for eventual parity of the
integral coefficients, and compute the number corresponding to $1+q^n$.

Finally, we propose a definition of the $q$-complex number $[\tau]_q$, a meromorphic function of $\tau\in \Bbb C_+$. This definition coincides with that of \cite{O} when $\tau$ is in the $PSL_2(\Bbb Z)$-orbit of the 3-rd or 4-th root of $1$ in the upper half-plane, but is different otherwise and seems to have somewhat better properties. In particular, 
$[\tau]_q$ expresses via hypergeometric functions of modular functions of $\tau$. 
\end{abstract}

\maketitle

\centerline{\bf To Yuri Tschinkel on his 60th birthday with admiration}

\tableofcontents

\section{Introduction}

Let $c_1,...,c_N$ be a sequence of positive integers
such that $c_i\ge 2$. To this sequence one attaches a negative continued 
fraction 
$$
[[c_1,...,c_N]]:=c_1-\frac{1}{c_2-...\frac{1}{c_N}}\in \Bbb Q\cap (1,\infty),
$$
and any rational number $>1$ can be uniquely represented in this way. 

One may also consider infinite continued fractions. Namely, 
let $\bold c:=\lbrace c_j,j\ge 1\rbrace$ be a sequence of integers $\ge 2$. The 
negative continued fraction attached to $\bold c$ is 
\(
[[c_1,c_2,...]]:=c_1-\frac{1}{c_2-\frac{1}{c_3-...}}.
\)
It is classical that this continued fraction converges to a real number $x(\bold c)\ge 1$ in the sense that $x(\bold c)$ is the limit of the truncations  
$[[c_1,...,c_N]]$ 
as $N\to \infty$, and the assignment $\bold c\mapsto x(\bold c)$ is a bijection 
between the set of sequences $\bold c$ and the interval $[1,\infty)$. 
Under this bijection, rational numbers correspond to sequences 
that eventually equal $2$, and quadratic irrationals correspond to eventually periodic 
sequences. 

Recently S. Morier-Genoud and V. Ovsienko proposed a $q$-deformation of this story (\cite{MO1,MO2}). Namely, they defined the $q$-deformed negative continued fraction 
$$
[[c_1,...,c_N]]_q:=[c_1]_q-\frac{q^{c_1-1}}{[c_2]_q-...\frac{q^{c_{N-1}-1}}{[c_N]_q}}
$$
and the infinite version 
$$
[[c_1,c_2,...]]_q:=[c_1]_q-\frac{q^{c_1-1}}{[c_2]_q-\frac{q^{c_2-1}}{[c_3]_q-...}},
$$
where $[n]_q=\frac{1-q^n}{1-q}$. This infinite fraction converges in the $q$-adic sense to an element $[x(\bold c)]_q\in \Bbb Z[[q]]$ with constant term $1$ in the sense that $[x(\bold c)]_q$ is the limit of the truncations  
$[[c_1,...,c_N]]_q$ as $N\to \infty$. This way, we can define $[x]_q\in \Bbb Z[[q]]$ for all 
$x\in [1,\infty)$, namely $[x]_q=[x(\bold c)]_q$ when $x=x(\bold c)$. 
Moreover, the assignment $x\mapsto [x]_q$ can be extended from $[1,\infty)$ to $\Bbb R$  using the formula  
\(
[x-1]_q=q^{-1}([x]_q-1),
\)
which produces an element of $\Bbb Z((q))$ if $x<0$. In particular, 
\(
[x]_q=\frac{1-q^x}{1-q}
\)
for any $x\in \Bbb Z$. This definition is essentially determined by the requirement that $q$-deformation should be compatible with modular equivariance. 

The series $[x]_q$ have many interesting connections to combinatorics 
described in \cite{MO1,MO2,MO3}. They also turned out to arise 
in the theory of stability conditions for 2-Calabi-Yau categories, see \cite{BBL}, and in computing Jones polynomials of rational knots (\cite{MO1}).   

The series $[x]_q$ converges to a rational function of $q$ for $x\in \Bbb Q$ and to an explicit quadratic irrational function of $q$ when $x$ is a quadratic irrational (\cite{MO1,LM}). We also explicitly compute $[x]_q$ for some transcendental numbers $x$, such as ${\rm cotan}\frac{1}{s}$, $s\in \Bbb Z_{\ge 1}$, and more generally 
certain ratios of special values of Bessel functions. The answer is given by 
ratios of special values of $q$-Bessel functions. 
We also solve a problem posed in \cite[Subsection 5.1]{MO2} by giving an exact basic hypergeometric formula for the $q$-analog $[e]_q$ of the number $e$. 

In particular, in all these cases the series $[x]_q$ has a positive radius of convergence $R(x)>0$.\footnote{Since $[x]_q$ is a series with integer coefficients, 
it is clear that if $x\notin \Bbb Z$ then $R(x)\le 1$. The problem of classification of rational numbers $x$ with $R(x)=1$ is very
nontrivial and was studied in \cite{EVW}.}

In \cite{LMOV} it is proved that if $x\in \Bbb Q$ then $R(x)\ge 3-2\sqrt{2}$ and  conjectured that $R(x)>0$ for all $x\in \Bbb R$ and moreover, $R(x)\ge R_*:=\frac{3-\sqrt{5}}{2}$; this is attained for the golden ratio $x=\frac{1+\sqrt{5}}{2}$.  This was proved for rational $x$ in \cite{EGMS} using the theory of Kleinian groups. 

In Theorem \ref{ma1}, we extend the result of \cite{LMOV} to all $x\in \Bbb R$, giving a partial proof of this conjecture.  

Namely, let $D\subset \Bbb C$ be the open region with boundary given in polar coordinates by the equation 
$$
r=1+2\sin\tfrac{\theta}{2}-2\sqrt{\sin\tfrac{\theta}{2}(1+\sin\tfrac{\theta}{2})},\ 0\le \theta\le 2\pi:
$$

\begin{figure}[htbp]
\vspace{-10pt}
    \centering
  \includegraphics[width=0.4\columnwidth]{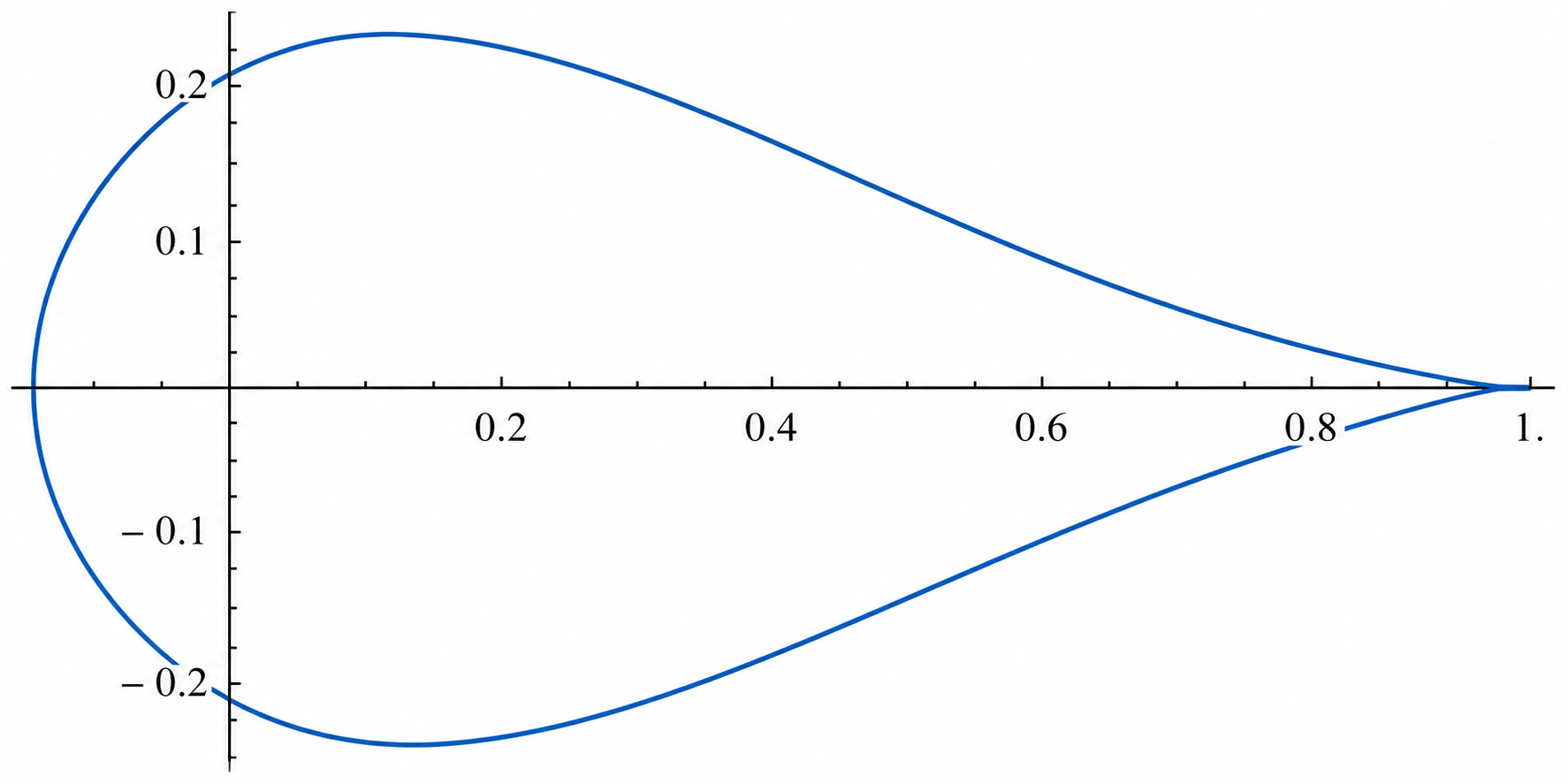}
  \label{fig:polarplot}
\end{figure}

In particular, the largest disk centered at the origin contained in $D$ has radius $3-2\sqrt{2}$. 

\begin{theorem}\label{ma1} (i) For any $x\ge 1$, the power series $[x]_q$ 
converges\footnote{For quadratic irrationals, this is proved in the PhD thesis of Ludivine Leclere (\cite{Le}), (see Theorems 4.4.4 and 4.4.5, p.63).} in the disk $|q|< 3-2\sqrt{2}$.

(ii) The holomorphic function defined by $[x]_q$ analytically continues 
to a non-vanishing holomorphic function on $D$.
\end{theorem} 

Theorem \ref{ma1} is proved by writing $1/[x]_q$ as a $q$-adically convergent series of rational functions which converges in $D$ absolutely and uniformly on compact sets. We also show 
that this series converges on the interval $(-\frac{3-\sqrt{5}}{2},1)$ to a positive analytic function, which allows us to define the function $[x]_q$ on this interval. We then proceed to investigate analytic properties of this function. 
In particular, we disprove the conjectural $q$-version of the 
Hurwitz irrational number theorem proposed in \cite{LMOV}. 

 Moreover, using the result of \cite{EGMS}, we extend 
Theorem \ref{ma1} to a larger disk. 

\begin{theorem}\label{ma1a} For any $x\ge 1$, the 
power series $[x]_q$ 
converges in the disk $|q|< 2-\sqrt{3}$ and defines a non-vanishing holomorphic function 
on this disk.
\end{theorem}  

We also study the analytic properties of $[x]_q$ as a function of $x$ when $0<q<1$.
It follows from \cite{MO1,MO2} that this function is increasing, so 
it gives rise to a Stieltjes measure $\mu_q=d[x]_q$ on $\Bbb R$. 
We show that $\mu_q$ is concentrated on a set of Lebesgue measure $0$, 
hence $\frac{d}{dx}[x]_q=0$ almost everywhere on $\Bbb R$. Moreover, 
if $q$ is not too close to $1$, we show that $\mu_q$ is purely atomic, concentrated at rational numbers, and we expect this to hold for all $q$. 

We next study the numerator and denominator polynomials of $q$-rational numbers on
the unit circle. For $x\in\Bbb Q$, write
$[x]_q=R_x(q)/S_x(q)$ in the standard normalization and put
\(
 T_x(q):=(1-q)R_x(q)-S_x(q).
\)
Let
\[
 C_-:=\left\{q:|q|=1,\ {\rm Re}\,q\le\frac12\right\},
 \qquad
 C_+:=\left\{q:|q|=1,\ {\rm Re}\,q>\frac12\right\}.
\]
A unitarity argument for the Jones polynomial of rational four-plats
shows that
\begin{equation}\label{inee1}
 |T_x(q)|+(1-q-q^{-1})|S_x(q)|\le |1-q|^2,
 \qquad q\in C_-,
\end{equation}
with equality for every integer $x$ and every $q\in C_-$,
and with equality at $q=e^{\pm \pi i/3}$ for every rational $x$; in particular, 
\begin{equation}\label{inee2}
|T_x(q)|\le |1-q|^2,\text{ and } |S_x(q)|\le \frac{|1-q|^2}{1-q-q^{-1}}\text{ if } q\ne  e^{\pm \pi i/3}.
\end{equation}
On the complementary arc we prove
\begin{equation}\label{inee3}
 |T_x(q)|\ge\sqrt{q+q^{-1}-1}\,|S_x(q)|,
 \qquad q\in C_+.
\end{equation} 
For $q\ne1$, \eqref{inee3} says that 
all finite $q$-rational values lie in the disk complement $E_q\subset \Bbb C\Bbb P^1$ defined by the inequality
\[
 \left|z-\frac1{1-q}\right|\ge 
 \frac{\sqrt{q+q^{-1}-1}}{|1-q|}
\]
(when $q$ tends to $1$ along the unit circle, this inequality degenerates into ${\rm Im}z\ge 0$ or ${\rm Im}z\le 0$, depending on whether 
$q$ approaches $1$ from above or below, and for $q=1$ the numbers $[x]_q=x$ indeed lie in the intersection of these two half-planes). If $q$ is not a root of unity and $X_q:=\{[x]_q:x\in\Bbb Q\}\subset\Bbb C\Bbb P^1$, then the closure 
$\overline{X_q}$ is $\Bbb C\Bbb P^1$ for $q\in C_-$, whereas for
$q\in C_+$ it is $E_q$. This shows that the inequality \eqref{inee3} is sharp; see Theorem
\ref{unitmainthm}.

Finally, motivated by \cite{O} and again guided by modular equivariance, we propose a notion of $q$-complex number $[\tau]_q$, where $\tau$ belongs to the upper half-plane $\Bbb C_+$, which is a meromorphic function of $\tau$. This notion coincides with that of \cite{O} when the lattice $\langle 1,\tau\rangle$ has rotational symmetries
(i.e., $\tau$ is in the $PSL_2(\Bbb Z)$-orbit of 
the third or fourth root of $1$), but is different elsewhere. 
However, this notion seems to have somewhat better properties. In particular $[\tau]_q$ can be computed explicitly: it is a ratio of hypergeometric functions
of a modular function of $\tau$. 

We also study coefficientwise reduction of $q$-real numbers.  Reduction
modulo every prime $p\ge 2$ is injective on the Cantor line, while
reduction modulo $2$ identifies it with
$\mathbb P^1(\mathbb F_2((q)))$.  We give an explicit decoder, extend the
rationality statements to every prime modulus, characterize the quadratic
series over $\mathbb F_2(q)$ whose decoded negative continued fraction is
eventually periodic, characterize polynomial images and eventual parity, and
compute the point corresponding to $1+q^n$.

The paper is organized as follows. Section 2 is dedicated to preliminaries; it is based on \cite{MO1,MO2,MO3} and contains no original results.
In Section 3 we give exact formulas for $q$-deformations of several transcendental numbers. Subsection 3.1 treats Bessel ratios, including ${\rm cotan}\frac{1}{s}$ for $s\in \Bbb Z_{\ge 1}$, while Subsection 3.2 gives an exact formula for $[e]_q$, solving a problem posed in \cite[Subsection 5.1]{MO2}. Section 4 is dedicated to the proofs of Theorems \ref{ma1},\ref{ma1a} and the study of the analytic properties of the function $[x]_q$ for real $x$ as a function of $q$. In Section 5 we study
$[x]_q$ as a function of $x$ when $0<q<1$. In Section 6 we study the Stieltjes measure $\mu_q=d[x]_q$. Section 7 is devoted to the behavior of the numerators and denominators of $q$-rational numbers on the unit circle. In Section 8 we study coefficientwise reduction of $q$-real numbers, with special emphasis on reduction modulo $2$. Finally, in Section 9 we develop the theory of $q$-complex numbers.

{\bf Acknowledgements.} I am very grateful to V. Ovsienko for introducing me to the topic and many useful discussions, and A. Veselov for helpful comments on preliminary versions of this text. This work was partially supported by the NSF grant DMS-2001318. 

I used ChatGPT 5.6 Pro for proofreading and editing the final version of the paper. In particular, based on a roadmap I provided, it wrote Subsection 3.2 and Sections 7 and 8.

\section{Preliminaries} \label{preli}

In this section we review the results of \cite{MO1,MO2}.

\subsection{Negative continued fractions}
Recall the classical theory of infinite negative continued fractions. Let
$\bold c:=\lbrace c_j,j\ge 1\rbrace$ be a sequence of integers $\ge 2$. 
The $N$-th convergent for this sequence 
is the $N$-truncated negative continued fraction attached to $\bold c$,
$$
x_N(\bold c)=x_N:=[[c_1,...,c_N]]=c_1-\frac{1}{c_2-...\frac{1}{c_N}}\in \Bbb Q\cap (1,\infty).
$$
For instance, $x_1(\bold c)=c_1,\ x_2(\bold c)=c_1-\frac{1}{c_2}$, etc. Thus
\(
x_N=\frac{a_N}{b_N} 
\)
where $a_N=a_N(\bold c),b_N=b_N(\bold c)$ are positive integers satisfying the recursion 
\begin{equation}\label{rec1}
u_{N}=c_Nu_{N-1}-u_{N-2}
\end{equation}
with $a_0=1,a_1=c_1$, $b_0=0,b_1=1$. For example, 
$$
a_2=c_1c_2-1,\ a_3=c_1c_2c_3-c_1-c_3,\ b_2=c_2, b_3=c_2c_3-1, {\rm etc.}
$$
and in general $b_N(c_1,c_2,c_3,...)=a_{N-1}(c_2,c_3,...)$.
Thus $a_N,b_N$ depend only on $c_1,...,c_N$, so e.g. we can write 
$a_N(\bold c)=a_N(c_1,...,c_N)$.

The numbers $a_N,b_N$ are positive. This can be seen by writing \eqref{rec1} as 
\(
u_{N}-u_{N-1}=(c_N-2)u_{N-1}+(u_{N-1}-u_{N-2}),
\)
implying by induction that if $u_1\ge 0,\ u_1\ge u_0$ then 
$u_N\ge 0,\ u_N\ge u_{N-1}$ for all $N$. The polynomials 
$a_N(\bold c)$ are called {\it Euler continuants}.  

We also have 
\(
\lim_{N\to \infty}x_N=x\in \Bbb R_+
\)
where  
$$
x(\bold c)=x:=[[c_1,c_2,...]]=c_1-\frac{1}{c_2-\frac{1}{c_3-...}}.
$$
We have $x\in [1,\infty)$, and any $x$ in this interval has a unique such representation. 

Note that 
$$
a_Nb_{N+1}-a_{N+1}b_N=a_N(c_{N+1}b_N-b_{N-1})-(c_{N+1}a_N-a_{N-1})b_N=
a_{N-1}b_N-a_Nb_{N-1}. 
$$
Thus $a_Nb_{N+1}-a_{N+1}b_N=a_0b_1-a_1b_0=1$, so 
$$
\frac{b_{N+1}}{a_{N+1}}-\frac{b_N}{a_N}=\frac{1}{a_Na_{N+1}}.
$$
It follows that
$$
\frac{1}{x_N}=\sum_{j=0}^{N-1}\frac{1}{a_ja_{j+1}},\
\frac{1}{x}=\sum_{j=0}^\infty \frac{1}{a_ja_{j+1}}.
$$

Finally, note that since $[[2,2,...]]=1$, if $c_N>2$ then we have 
\(
[[c_1,...,c_N,2,2,...]]=[[c_1,...,c_N-1]].
\)

\subsection{Negative $q$-continued fractions}
Now consider the $q$-analog of this theory (\cite{MO1,MO2,MO3}). Let $[n]_q:=\frac{1-q^n}{1-q}$. 
Set 
$$
[x(\bold c)]_{N,q}=[x]_{N,q}:=[c_1]_q-\frac{q^{c_1-1}}{[c_2]_q-...\frac{q^{c_{N-1}-1}}{[c_N]_q}}.
$$
For instance, $[x(\bold c)]_{1,q}=[c_1]_q$, $[x(\bold c)]_{2,q}=[c_1]_q-\frac{q^{c_1-1}}{[c_2]_q}$, etc. 
Thus 
\(
[x]_{N,q}=[x_N]_q=\frac{a_N(q)}{b_N(q)} 
\)
where $a_N(q)=a_N(\bold c|q),b_N(q)=b_N(\bold c|q)$ are polynomials with integer coefficients satisfying the recursion 
\begin{equation}\label{rec2}
u_{N}=[c_N]_qu_{N-1}-q^{c_{N-1}-1}u_{N-2}
\end{equation}
with $a_0=1,a_1=[c_1]_q$, $b_0=0,b_1=1$. 
For example, 
$$
a_2=[c_1]_q[c_2]_q-q^{c_1-1},\ a_3=[c_1]_q[c_2]_q[c_3]_q-q^{c_2-1}[c_1]_q-q^{c_1-1}[c_3]_q,\ b_2=[c_2]_q, b_3=[c_2]_q[c_3]_q-q^{c_2-1}, 
$$
and in general $b_N(c_1,c_2,c_3,...|q)=a_{N-1}(c_2,c_3,...|q)$.
Thus $a_N,b_N$ depend only on $c_1,...,c_N$, so e.g. we can write 
$a_N(\bold c|q)=a_N(c_1,...,c_N|q)$.

The polynomials $a_N,b_N$ have non-negative coefficients. This can be seen by writing \eqref{rec2} as 
$$
u_{N}-q^{c_N-1}u_{N-1}=q[c_{N}-2]_qu_{N-1}+(u_{N-1}-q^{c_{N-1}-1}u_{N-2}),
$$
implying by induction that if $u_1\ge 0$ and $u_1\ge q^{c_1-1}u_0$ then 
$u_N\ge 0$ and $u_N\ge q^{c_N-1}u_{N-1}$ coefficient-wise for all $N$. 
Also $a_N(\bold c|0)=b_N(\bold c|0)=1$ for $N\ge 1$. Finally, it follows by induction that the degree 
of $a_N(\bold c|q)$ is 
\(
C_N:=\sum_{j=1}^N(c_j-1)
\)
(so $C_N\ge N$), and that both the leading coefficient and constant term 
of $a_N(\bold c|q)$ equal $1$. The polynomials $a_N(\bold c|q)$ 
are called the {\it $q$-deformed Euler continuants}.

We also have \(\lim_{N\to \infty}[x]_{N,q}=[x]_q\in \Bbb Z[[q]]\) 
($q$-adic limit), where  
$$
[x(\bold c)]_q=[x]_q:=[[c_1,c_2,...]]_q=[c_1]_q-\frac{q^{c_1-1}}{[c_2]_q-\frac{q^{c_2-1}}{[c_3]_q-...}}.
$$
Indeed, note that 
$$
a_Nb_{N+1}-a_{N+1}b_N=a_N([c_{N+1}]_qb_N-q^{c_{N}-1}b_{N-1})-([c_{N+1}]_qa_N-q^{c_{N}-1}a_{N-1})b_N=
$$
$$
q^{c_{N}-1}(a_{N-1}b_N-a_Nb_{N-1}). 
$$
Thus 
$$
a_Nb_{N+1}-a_{N+1}b_N=q^{C_{N}}(a_0b_1-a_1b_0)=q^{C_{N}}.
$$ 
Hence 
$$
\frac{b_{N+1}}{a_{N+1}}-\frac{b_N}{a_N}=\frac{q^{C_{N}}}{a_Na_{N+1}}.
$$

Thus we get 

\begin{proposition} 
\begin{equation}\label{seri0}
\frac{1}{[x]_{N,q}}=\sum_{j=0}^{N-1}\frac{q^{C_{j}}}{a_j(q)a_{j+1}(q)}.
\end{equation}
So $[x]_{N,q}$ is $q$-adically convergent to $[x]_q\in \Bbb Z[[q]]$
such that
\begin{equation}\label{seri}
\frac{1}{[x]_q}=\sum_{j=0}^{\infty}\frac{q^{C_j}}{a_j(q)a_{j+1}(q)}.
\end{equation}
\end{proposition}

Note that similarly to the $q=1$ case, for $c_N>2$ we have 
\(
[[c_1,...,c_N,2,2,...]]_q=[[c_1,...,c_N-1]]_q.
\)

\subsection{Modular equivariance}
Observe that for any $x>1$, 
\begin{equation}\label{xp1q0}
[x+1]_q=q[x]_q+1,
\end{equation} 
i.e., 
\begin{equation}\label{xp1q}
[x+n]_q=q^n[x]_q+[n]_q
\end{equation}  
for $n\in \Bbb Z_{\ge 0}$. Using this formula, 
the assignment $x\to [x]_q$ can be extended from $(1,\infty)$ to the whole real line,
so that \eqref{xp1q} holds for all $n\in \Bbb Z$.
If $x<0$ then $[x]_q$ defined in this way is a Laurent series in $q$. 
The following lemma upgrades \eqref{xp1q} to modular equivariance of $[x]_q$. 

\begin{lemma} \label{modinv} For all nonzero $x\in \Bbb R$, we have 
$$
[1/x]_q[-x]_q=-q^{-1}.
$$
Thus the leading term of $[x]_q$ is $1$ if $x\ge 1$, 
$q^k$ if $\frac{1}{k+1}\le x<\frac{1}{k}$ 
and $-q^{-k}$ if $-k\le x<-k+1$ for integer $k\ge 1$. 
\end{lemma} 

\begin{proof} It suffices to treat the case $x>0$. Let $n$ be the smallest integer $\ge x$. We prove the lemma by induction in $n$. {If $n=1$, then $0<x\le1$.  For $0<x<1$, the defining
negative continued fraction and \eqref{xp1q} give, respectively,
\[
 [2-x]_q=[2]_q-\frac{q}{[1/x]_q},
 \qquad
 [2-x]_q=[2]_q+q^2[-x]_q.
\]
Comparing these identities yields $[1/x]_q[-x]_q=-q^{-1}$.
For $x=1$, the same conclusion follows from $[1]_q=1$ and
$[-1]_q=-q^{-1}$.  This proves the base of induction.} For the induction step, let $n>1$, and assume the statement has been proved for smaller numbers. Let $y:=x-1$. By \eqref{xp1q} and the induction assumption,
$$
[-x]_q=q^{-1}([-y]_q-1)=-q^{-1}\left(\frac{q^{-1}}{[\frac{1}{y}]_q}+1\right)
$$
and by \eqref{xp1q} and the base of induction applied to $\frac{y}{1+y}$
$$
[\tfrac{1}{x}]_q=[\tfrac{1}{1+y}]_q=[1-\tfrac{y}{1+y}]_q=1+q[-\tfrac{y}{1+y}]_q=1-\frac{1}{[1+\tfrac{1}{y}]_q}=
1-\frac{1}{q[\tfrac{1}{y}]_q+1}=\frac{[\tfrac{1}{y}]_q}{[\tfrac{1}{y}]_q+q^{-1}},
$$
which yields $[\tfrac{1}{x}]_q[-x]_q=-q^{-1}$, as desired. This proves the first statement, and the second one follows immediately. 
\end{proof} 

\begin{corollary}\label{qin} For $x\in \Bbb Q$ 
\begin{equation}\label{qinv}
[x]_{q^{-1}}=-q[-x]_q.
\end{equation}
\end{corollary} 

\begin{proof} By \eqref{xp1q} and Lemma \ref{modinv}, equation \eqref{qinv}
is stable under the action of $PSL_2(\Bbb Z)$ on $x$. Thus it suffices to check 
this equation for $x=0$, in which case both sides equal zero. 
\end{proof} 

Thus the theory of $q$-numbers for $|q|\ge 1$ is equivalent to one for $|q|\le 1$, 
so we will focus on this case. 

\section{Exact formulas for some $q$-deformed transcendental numbers}

\subsection{$q$-deformed Bessel ratios}\label{besselratios}

Consider the (renormalized) $q$-Bessel function
$$
J(c,q,z):=(-z;q)_\infty \cdot \qphi21{0,0}{c;q}{-z}={}_0\phi_1(c;q;-cz)=\sum_{n=0}^\infty \frac{(-1)^nq^{n(n-1)}c^nz^{n}}{(c;q)_n (q;q)_n}.
$$
If $|q|<1$ and $c\ne q^{-m}$ for $m\in \Bbb Z_{\ge 0}$, then 
this series converges for all $z\in \Bbb C$. 

\begin{lemma}\label{kslem} (\cite{Is}; see also \cite{KS}, formula (6.4)) We have 
$$
\frac{J(qy,q,x)}{J(y,q,x)}=\frac{1-y}{1-y-\frac{xy}{1-qy-\frac{qxy}{1-q^2y-...}}}
$$
\end{lemma} 

Setting $x=q^{-1}(1-q)^2$ and $y=q^2$, we obtain
$$
\frac{(1+q)J(q^2,q,q^{-1}(1-q)^2)}{J(q^3,q,q^{-1}(1-q)^2)}=\frac{1-q^2}{1-q}-\frac{q}{\frac{1-q^3}{1-q}-...}=[x(\bold c)]_q,
$$
where $\bold c=(2,3,4,...)$. 

As $q\to 1$, we have 
$$
J(q^{\nu+1},q,q^{-1}(1-q)^2)\to \sum_{n=0}^\infty \frac{(-1)^n}{(\nu+1)_n n!}=\Gamma(\nu+1)J_{\nu}(2),
$$
where $J_\nu$ is the Bessel function of the first kind:
\[
J_{\nu}(z)\;=\;\sum_{n=0}^{\infty}
      \frac{(-1)^n}{n!\,\Gamma(n+\nu+1)}
      \left(\frac{z}{2}\right)^{2n+\nu}.
\]
Thus 
\(
x(\bold c)=\frac{J_1(2)}{J_2(2)}\approx {1.63455},
\)
where 
\(
J_m(2)=\sum_{n=0}^\infty\frac{(-1)^n}{n!(n+m)!}.
\)
We thus obtain

\begin{proposition}
$$
\left[\frac{J_1(2)}{J_2(2)}\right]_q=\frac{J_1(2)_q}{J_2(2)_q},
$$
where 
$$
J_m(2)_q:=\sum_{n=0}^\infty \frac{(-1)^nq^{n(n+m-1)}}{[n]_q![n+m]_q!}
$$
with 
$$
[n]_q!:=\prod_{j=1}^n [j]_q=\frac{(q; q)_n}{(1-q)^n}.
$$ 
\end{proposition} 

In a similar way one can compute $[x(\bold c)]_q$ 
when $\bold c$ is any arithmetic progression. 
Namely, if the step of the progression is $r$, replace 
$q$ by $q^r$ in Lemma \ref{kslem}: 
$$
\frac{J(q^{r}y,q^r,x)}{J(y,q^r,x)}=\frac{1-y}{1-y-\frac{xy}{1-q^{r}y-\frac{q^rxy}{1-q^{2r}y-...}}}.
$$
Now if $c_1=s$ (i.e., $\bold c=(s,s+r,s+2r,...)$), set $y=q^s$, $x=q^{-1}(1-q)^2$. This yields
\begin{equation}\label{transex}
\frac{[s]_qJ(q^{s},q^r,q^{-1}(1-q)^2)}{J(q^{r+s},q^r,q^{-1}(1-q)^2)}=\frac{1-q^s}{1-q}-\frac{q^{s-1}}{\frac{1-q^{r+s}}{1-q}-\frac{q^{r+s-1}}{\frac{1-q^{2r+s}}{1-q}-...}}=[x(\bold c)]_q, 
\end{equation}
and by taking the limit $q\to 1$ we get that 
$$
x(\bold c)=\frac{J_{\frac{s}{r}-1}(\frac{2}{r})}{J_{\frac{s}{r}}(\frac{2}{r})}.
$$
Note that C. L. Siegel showed in 1929 that such numbers are transcendental (\cite{B}, Subsection 11.1).

\begin{example} Recall that 
\(
\frac{J_{-1/2}(z)}{J_{1/2}(z)}={\rm cotan}(z).
\)
Thus for $r=2s$, i.e., $\bold c=(s,3s,5s,...)$ we obtain an 
explicit formula for $[{\rm cotan}(\frac{1}{s})]_q$, $s\in \Bbb Z_{\ge 2}$:
\begin{equation}\label{cotanfor}
\left[{\rm cotan}\left(\frac{1}{s}\right)\right]_q=\frac{[s]_qJ(q^{s},q^{2s},q^{-1}(1-q)^2)}{J(q^{3s},q^{2s},q^{-1}(1-q)^2)}.
\end{equation}
It is easy to see that this formula is also valid for $s=1$. 

Let 
$$
{\rm Cos}_q(z):=\sum_{n\ge 0}\frac{(-1)^nq^{n(2n-1)}z^{2n}}{[2n]_{q}!},\
{\rm Sin}_q(z):=\sum_{n\ge 0}\frac{(-1)^nq^{n(2n+1)}z^{2n+1}}{[2n+1]_{q}!}
$$
be the (renormalized) $q$-trigonometric functions defined 
in \cite{KLS}, p.23 (which at $q=1$ obviously specialize to $\cos z,\sin z$)
and
$$
{\rm Cotan}_q(z):=\frac{{\rm Cos}_q(z)}{{\rm Sin}_q(z)}.
$$
Then formula \eqref{cotanfor} can be written as 
$$
\left[{\rm cotan}\left(\frac{1}{s}\right)\right]_q=q^{-1/2}{\rm Cotan}_{q^s}\left(\frac{q^{-1/2}}{[s]_q}\right),\ s\in \Bbb Z_{\ge 1}.
$$
For example, 
$$
[{\rm cotan}(1)]_q=q^{-1/2}{\rm Cotan}_{q}(q^{-1/2})=\frac{\sum_{n\ge 0}\frac{(-1)^nq^{2n(n-1)}}{[2n]_{q}!}}{\sum_{n\ge 0}\frac{(-1)^nq^{2n^2}}{[2n+1]_{q}!}}.
$$
\end{example}

\begin{remark} Note that both the numerator and denominator of the left hand side of 
\eqref{transex} are holomorphic in the open unit disk $|q|<1$. Thus the function $[x]_q$ 
for $x=\frac{J_{\frac{s}{r}-1}(\frac{2}{r})}{J_{\frac{s}{r}}(\frac{2}{r})}$ extends to a  
meromorphic function in this disk. However, the unit circle is its natural boundary, i.e., $[x]_q$ does not extend meromorphically to any larger open set. 

In fact, if $x\notin \Bbb Q$ then $[x]_q$ can never  extend meromorphically to any open set $U$ properly containing the open unit disk. Indeed, Proposition \ref{lim} below shows that $[x]_q\to x$ as $q\to1^-$. On the other hand, a theorem of P\'olya-Carlson (\cite{P},\cite{C}) states that if a power series $f(q)$ with integer coefficients has a positive radius of convergence and if its sum extends to a meromorphic function on $U$, then $f\in \Bbb Q(q)$, so if $f$ is regular at $q=1$ then $f(1)\in \Bbb Q$. So the statement follows by applying this theorem to $f(q)=[x]_q$.

However, for $x\notin \Bbb Q$, $[x]_q$ need not extend meromorphically to $|q|<1$. 
For example, the $q$-deformed golden ratio $[\varphi]_q$ is a quadratic irrational function 
with a branch point at $q=-\frac{3-\sqrt{5}}{2}$. In fact, the P\'olya-Carlson theorem implies that 
$[x]_q$ for any quadratic irrational $x$ must have at least one branch point  strictly inside the unit disk (\cite{EVW}, Proposition 4.1).  
This raises an interesting question: 

{\bf Question.} For which real numbers $x$ does $[x]_q$ meromorphically extend
to $|q|<1$?
\end{remark}

\subsection{$q$-deformation of $e$}\label{qdeforme}
The problem of finding an exact formula for the Morier--Genoud--Ovsienko
$q$-deformation $[e]_q$ was posed in \cite[Subsection 5.1]{MO2}.  We
now give such a formula.  It is obtained by contracting the triples
$(1,2n,1)$ in Euler's regular continued fraction for $e$, removing a
parity dependence by a scalar gauge, and solving the resulting
second-order $q$-difference equation by a confluent basic
hypergeometric function.

\subsubsection{Explicit formulas for $[e]_q$}
Put
\begin{equation}\label{eq:kappa}
 Q:=q^2,
 \qquad
 \kappa:=\frac{(1-q)^2}{(1+q)(1+q^3)}.
\end{equation}
We use the convention
\(
 (a;Q)_m:=\prod_{j=0}^{m-1}(1-aQ^j)
\)
and
\[
 {}_1\phi_1\!\left(\begin{matrix}a\\b\end{matrix};Q,t\right)
 :=\sum_{m=0}^{\infty}
 \frac{(a;Q)_m}{(b;Q)_m(Q;Q)_m}
 (-1)^mQ^{\binom m2}t^m.
\]
Thus $(0;Q)_m=1$.  Define
\begin{equation}\label{eq:Fdef}
 F_q(z):=
 {}_1\phi_1\!\left(\begin{matrix}\kappa q\\0\end{matrix};q^2,q^3z\right)
 =\sum_{m=0}^{\infty}
 (-1)^m q^{m(m+2)}z^m
 \frac{(\kappa q;q^2)_m}{(q^2;q^2)_m}.
\end{equation}

\begin{theorem}\label{thm:emain}
The series $F_q(z)$ is the unique element of
$1+z\Bbb Q(q)[[z]]$ satisfying
\begin{equation}\label{eq:main-difference}
 F_q(z)=(1-q^3z)F_q(q^2z)+\kappa q^4zF_q(q^4z).
\end{equation}
Moreover,
\begin{equation}\label{eq:emain-direct}
 [e]_q=
 \frac{(1+q^3)[3]_qF_q(1)+q^3(1-q)F_q(q^2)}
 {(1+q^3)F_q(1)+q^2(1-q)F_q(q^2)}.
\end{equation}
Equivalently, if
\begin{equation}\label{eq:Rdef}
 R_q:=\frac{1+q^3}{1-q}\frac{F_q(1)}{F_q(q^2)},
\end{equation}
then
\begin{equation}\label{eq:emain-R}
 [e]_q=\frac{[3]_qR_q+q^3}{R_q+q^2}.
\end{equation}
\end{theorem}

In terms of the renormalized $q$-Bessel function $J$ used in
Subsection \ref{besselratios}, the answer takes the following form.

\begin{corollary}\label{cor:eJformula}
With $\kappa$ as in \eqref{eq:kappa}, one has
\begin{equation}\label{eq:emain-J}
 [e]_q=
 \frac{
 [3]_q[6]_qJ(q^3,q^2,-\kappa q)
 +q^3J(q^5,q^2,-\kappa q)}
 {
 [6]_qJ(q^3,q^2,-\kappa q)
 +q^2J(q^5,q^2,-\kappa q)}.
\end{equation}
\end{corollary}

Numerical evaluation of the numerator and denominator in
\eqref{eq:emain-J} suggests that the zeros of the denominator of
smallest modulus form the conjugate pair $q_0,\overline q_0$, where 
\[
 q_0\approx -0.64149948+0.31755991i,\qquad
 |q_0|\approx 0.71579737,
\]
and that the numerator does not vanish at these points.  Thus the
numerical evidence suggests that the radius of convergence of $[e]_q$
is approximately $0.71579737$. 

\subsubsection{Contraction of Euler's continued fraction}
Euler's regular continued fraction is
\begin{equation}\label{eq:euler-cf}
 e=[2;1,2,1,1,4,1,1,6,1,1,8,1,\ldots],
\end{equation}
so, after the initial partial quotient $2$, it is built from the blocks
\(
 (1,2n,1),\qquad n=1,2,3,\ldots.
\)

We use the polynomial matrix form of the regular $q$-continued
fraction from \cite{MO1,MO2}.  At an odd position put
\[
 O(a):=
 \begin{pmatrix}
 [a]_q&q^a\\
 1&0
 \end{pmatrix},
\]
and at an even position put
\[
 E(a):=
 \begin{pmatrix}
 q[a]_q&1\\
 q^a&0
 \end{pmatrix}.
\]
The second matrix is $q^a$ times the usual matrix involving
$[a]_{q^{-1}}$; this scalar does not change the projective value of the
continued fraction.  Thus a finite regular $q$-continued fraction is
the ratio of the two components of
\(
 M_1(a_1)M_2(a_2)\cdots M_N(a_N)\binom10,
\)
where $M_j=O$ for $j$ odd and $M_j=E$ for $j$ even.

Because each block has length three, its starting parity alternates.
Set
\[
 \mathcal B_n:=
 \begin{cases}
 E(1)O(2n)E(1),&n\text{ odd},\\
 O(1)E(2n)O(1),&n\text{ even}.
 \end{cases}
\]
Introduce the two constant gauge matrices
\begin{equation}\label{eq:gauges}
 G_{\mathrm o}:=
 \begin{pmatrix}1&-q^2\\1&1\end{pmatrix},
 \qquad
 G_{\mathrm e}:=
 \begin{pmatrix}q&1\\1&-q\end{pmatrix}.
\end{equation}
A direct multiplication gives, for odd $n$,
\begin{equation}\label{eq:block-odd}
 G_{\mathrm o}^{-1}\mathcal B_nG_{\mathrm e}
 =
 \begin{pmatrix}
 (1+q^3)[2n+1]_q&q^{2n+2}\\
 -1&0
 \end{pmatrix},
\end{equation}
and, for even $n$,
\begin{equation}\label{eq:block-even}
 G_{\mathrm e}^{-1}\mathcal B_nG_{\mathrm o}
 =
 \begin{pmatrix}
 (1+q)[2n+1]_q&-q^{2n+2}\\
 1&0
 \end{pmatrix}.
\end{equation}

Let $Z_n$ be a projective tail vector beginning with the $n$th block,
and put
\[
 W_n:=
 \begin{cases}
 G_{\mathrm o}^{-1}Z_n,&n\text{ odd},\\
 G_{\mathrm e}^{-1}Z_n,&n\text{ even}.
 \end{cases}
\]
Equations \eqref{eq:block-odd}--\eqref{eq:block-even} imply that one can
write
\[
 W_n=
 \begin{cases}
 (r_n,-r_{n+1})^{\mathsf T},&n\text{ odd},\\
 (r_n,r_{n+1})^{\mathsf T},&n\text{ even},
 \end{cases}
\]
where the scalar sequence $r_n$ satisfies
\begin{equation}\label{eq:rrec}
 r_n=\gamma_n[2n+1]_q r_{n+1}+q^{2n+2}r_{n+2},
 \qquad
 \gamma_n=
 \begin{cases}
 1+q^3,&n\text{ odd},\\
 1+q,&n\text{ even}.
 \end{cases}
\end{equation}

The initial partial quotient is especially simple.  Since
$W_1=(r_1,-r_2)^{\mathsf T}$, one obtains
\begin{equation}\label{eq:initial-vector}
 O(2)G_{\mathrm o}\binom{r_1}{-r_2}
 =
 \binom{[3]_qr_1+q^3r_2}{r_1+q^2r_2}.
\end{equation}
Consequently,
\begin{equation}\label{eq:e-ratio-r}
 [e]_q=\frac{[3]_qr_1+q^3r_2}{r_1+q^2r_2}.
\end{equation}

\subsubsection{Removal of the parity and the $q$-difference equation}
The parity in \eqref{eq:rrec} is removed by a scalar gauge.  Put
\[
 \alpha_n:=
 \begin{cases}
 \dfrac{1+q^3}{1-q},&n\text{ odd},\\[6pt]
 \dfrac{1+q}{1-q},&n\text{ even}.
 \end{cases}
\]
Then
\[
 \gamma_n[2n+1]_q=\alpha_n(1-q^{2n+1}),
 \qquad
 \alpha_n\alpha_{n+1}=\kappa^{-1}.
\]
Choose $\lambda_1=1$, define
$\lambda_{n+1}:=\alpha_n\lambda_n$, and set
$f_n:=\lambda_nr_n$.  Equation \eqref{eq:rrec} becomes
\begin{equation}\label{eq:frec}
 f_n=(1-q^{2n+1})f_{n+1}+\kappa q^{2n+2}f_{n+2}.
\end{equation}
This is independent of the parity of $n$.

Let $f_n=F(q^{2n-2})$.  Then \eqref{eq:frec} becomes
\eqref{eq:main-difference}.  The tail selected by the infinite
continued fraction is the $q$-adically bounded solution for which
$f_n$ tends projectively to a nonzero limit as $n\to\infty$.  After an
overall normalization, this is the unique solution with $F(0)=1$.

To justify this selection, truncate after an even number $M$ of complete
blocks.  After the above gauge, the ratios
\(
 t_n^{(M)}:=f_n^{(M)}/f_{n+1}^{(M)}
\)
satisfy
\begin{equation}\label{eq:finite-tail-ratio}
 t_n^{(M)}=1-q^{2n+1}+\frac{\kappa q^{2n+2}}{t_{n+1}^{(M)}},
\end{equation}
with $t_{M+1}^{(M)}=(1-q)/(1+q^3)$; hence all $t_n^{(M)}$ are
$q$-adic units.  The ratios
\(
 t_n:=F_q(q^{2n-2})/F_q(q^{2n})
\)
satisfy the same recurrence and are also $q$-adic units.  Subtracting
and iterating gives
\[
 \operatorname{ord}_q\bigl(t_n^{(M)}-t_n\bigr)
 \ge \sum_{j=n}^{M}(2j+2).
\]
Thus, for fixed $n$, the finite-truncation ratios converge $q$-adically
to $t_n$.  Hence the continued fraction selects the regular solution
$F_q$; since $t_1=\frac{r_1}{\alpha_1r_2}$, an overall normalization gives
\eqref{eq:r1r2F}.

Thus
\begin{equation}\label{eq:r1r2F}
 r_1=F_q(1),
 \qquad
 r_2=\frac{1-q}{1+q^3}F_q(q^2).
\end{equation}
Substitution in \eqref{eq:e-ratio-r} proves
\eqref{eq:emain-direct} and \eqref{eq:emain-R}.

As a check, at $q=1$ the unnormalized recurrence \eqref{eq:rrec}
becomes
\(
 r_n=(4n+2)r_{n+1}+r_{n+2}.
\)
Hence
\[
 \frac{r_1}{r_2}=6+\cfrac1{10+\cfrac1{14+\cdots}},
 \qquad
 e=\frac{3(r_1/r_2)+1}{(r_1/r_2)+1},
\]
which is Euler's classical contraction of \eqref{eq:euler-cf}.

\subsubsection{Hypergeometric solution}
Set
\[
 t:=q^3z,
 \qquad
 a:=\kappa q,
 \qquad
 Q:=q^2,
 \qquad
 \Phi(t):=F_q(t/q^3).
\]
Then \eqref{eq:main-difference} takes the canonical form
\begin{equation}\label{eq:canonical}
 at\Phi(Q^2t)+(1-t)\Phi(Qt)-\Phi(t)=0.
\end{equation}
If $\Phi(t)=\sum_{m\ge0}c_mt^m$, comparison of coefficients gives
\[
 (1-Q^m)c_m=-Q^{m-1}\bigl(1-aQ^{m-1}\bigr)c_{m-1}
 \qquad(m\ge1).
\]
Equivalently, in the $z$-normalization of \eqref{eq:Fdef},
\begin{equation}\label{eq:coeff-rec}
 \frac{[z^m]F_q(z)}{[z^{m-1}]F_q(z)}
 =-q^{2m+1}\frac{1-\kappa q^{2m-1}}{1-q^{2m}}.
\end{equation}
Therefore
\(
 c_m=\frac{(a;Q)_m}{(Q;Q)_m}
 (-1)^mQ^{\binom m2},
\)
which proves \eqref{eq:Fdef}.  In particular, the ratio of successive
coefficients is rational in $Q^m$, so \eqref{eq:canonical} is a
confluent basic hypergeometric equation and its regular solution is a
${}_1\phi_1$.

Equation \eqref{eq:coeff-rec} also shows that the Taylor-series
solution at $t=0$ is unique up to scalar.

\subsubsection{Conversion to the $q$-Bessel function}
The standard transformation
\begin{equation}\label{eq:1phi1-transform}
 {}_1\phi_1\!\left(\begin{matrix}a\\0\end{matrix};Q,c\right)
 =(c;Q)_\infty J(c,Q,-a)
\end{equation}
relates the two normalizations.  Applying it with
$Q=q^2$, $a=\kappa q$, and $c=q^3z$ gives
\begin{equation}\label{eq:FJ}
 F_q(z)=(q^3z;q^2)_\infty J(q^3z,q^2,-\kappa q).
\end{equation}
Hence
\[
 \frac{F_q(1)}{F_q(q^2)}
 =(1-q^3)
 \frac{J(q^3,q^2,-\kappa q)}{J(q^5,q^2,-\kappa q)}.
\]
Since
\(
 \frac{1+q^3}{1-q}(1-q^3)=\frac{1-q^6}{1-q}=[6]_q,
\)
formula \eqref{eq:Rdef} becomes
\begin{equation}\label{eq:R-J}
 R_q=[6]_q
 \frac{J(q^3,q^2,-\kappa q)}{J(q^5,q^2,-\kappa q)}.
\end{equation}
Substitution in \eqref{eq:emain-R} proves Corollary
\ref{cor:eJformula}.

One may also express the answer through Jackson's second $q$-Bessel
function (see, e.g., \cite{KLS}).  {Suppose $0<q<1$.} If
\(
 x_q:=2i\sqrt{\kappa q},
\)
then the identity
\[
 J_\nu^{(2)}(x;Q)
 =\frac{(x/2)^\nu}{(Q;Q)_\infty}
 {}_1\phi_1\!\left(\begin{matrix}-x^2/4\\0\end{matrix};Q,Q^{\nu+1}\right)
\]
shows that $F_q(1)$ and $F_q(q^2)$ correspond respectively to the
orders $\nu=\frac12$ and $\nu=\frac32$, with base $Q=q^2$ and argument
$x_q$.  Thus
\begin{equation}\label{eq:R-Jackson}
 R_q=\frac{1+q^3}{1-q}\frac{x_q}{2}
 \frac{J_{1/2}^{(2)}(x_q;q^2)}{J_{3/2}^{(2)}(x_q;q^2)}.
\end{equation}
{For $0<|q|<1$, the right hand side of
\eqref{eq:R-Jackson} is interpreted by analytic continuation from
$0<q<1$, using compatible local branches with $\log(q^2)=2\log q$
and the same branch for the powers of $x_q/2$ in both Bessel functions.}
This is parallel to the half-integral-order $q$-Bessel formulas for the
$q$-cotangent above.

\subsubsection{Explicit $q$-series and analytic consequences}
Formula \eqref{eq:Fdef} gives
\begin{align}
 F_q(1)
 &=\sum_{m=0}^{\infty}
 (-1)^m q^{m(m+2)}
 \frac{(\kappa q;q^2)_m}{(q^2;q^2)_m},
 \label{eq:F1series}\\
 F_q(q^2)
 &=\sum_{m=0}^{\infty}
 (-1)^m q^{m(m+4)}
 \frac{(\kappa q;q^2)_m}{(q^2;q^2)_m}.
 \label{eq:F2series}
\end{align}
The $m$th summands have $q$-adic orders at least $m(m+2)$ and
$m(m+4)$, respectively, so only finitely many summands are needed
modulo any prescribed power of $q$.

For example, modulo $q^{31}$ the numerator and denominator of \eqref{eq:emain-direct} are 
\begin{align*}
 \mathcal N_q={}&1+q+q^2+q^3-3q^5+q^6-8q^7+4q^8-10q^9
 +7q^{10}-9q^{11}\\
 &+10q^{12}-11q^{13}+21q^{14}-30q^{15}+57q^{16}
 -88q^{17}+146q^{18}\\
 &-215q^{19}+320q^{20}-445q^{21}+609q^{22}-805q^{23}
 +1043q^{24}\\
 &-1325q^{25}+1666q^{26}-2084q^{27}+2605q^{28}
 -3288q^{29}+4189q^{30}+O(q^{31}),
\end{align*}
and
\begin{align*}
 \mathcal D_q={}&1+q^2-q^3+q^4-4q^5+4q^6-9q^7+11q^8-18q^9
 +25q^{10}-35q^{11}\\
 &+50q^{12}-69q^{13}+100q^{14}-141q^{15}+203q^{16}
 -285q^{17}+403q^{18}\\
 &-560q^{19}+779q^{20}-1074q^{21}+1473q^{22}-2006q^{23}
 +2715q^{24}\\
 &-3651q^{25}+4886q^{26}-6516q^{27}+8661q^{28}
 -11493q^{29}+15222q^{30}+O(q^{31}).
\end{align*}
Thus
\begin{align}\label{eq:e-series}
 [e]_q=\frac{\mathcal N_q}{\mathcal D_q}={}&1+q+q^3-q^5+2q^6-3q^7+3q^8-q^9-3q^{10}
 +9q^{11}-17q^{12}\notag\\
 &+25q^{13}-29q^{14}+23q^{15}+2q^{16}-54q^{17}
 +134q^{18}-232q^{19}\notag\\
 &+320q^{20}-347q^{21}+243q^{22}+71q^{23}-660q^{24}
 +1531q^{25}\notag\\
 &-2575q^{26}+3504q^{27}-3804q^{28}+2747q^{29}
 +488q^{30}+O(q^{31}).
\end{align}
These initial terms agree with the expansion computed directly in
\cite[Subsection 5.1]{MO2}.

For fixed $|q|<1$, the series \eqref{eq:Fdef} is entire in $z$
and holomorphic in $q$: the possible poles of $\kappa$ lie on the unit
circle, the factors $(q^2;q^2)_m$ do not vanish, and the quadratic
$q$-adic growth gives local uniform convergence.  Thus
\eqref{eq:emain-direct} extends $[e]_q$ meromorphically to $|q|<1$,
with possible poles only at zeros of its denominator.

\section{The function $[x]_q$ for $q\in D$}

\subsection{Proof of Theorem \ref{ma1}}
Let $q=re^{i\theta}\in \Bbb C\setminus\{0\}$, $r=|q|\le 1$, with
$0\le\theta<2\pi$. Suppose $q\in D$, which is equivalent to the condition
\begin{equation}\label{condi}
r+r^{-1}-2>4\sin \tfrac{\theta}{2},\ 0\le \theta<2\pi.
\end{equation}
Note that $r+r^{-1}-2=(r^{-1/2}-r^{1/2})^2\ge 0$. Thus \eqref{condi} is equivalent to 
\(
(r+r^{-1}-2)^2>16\sin^2\tfrac{\theta}{2},
\)
which we may write as
\(
(r^{-1}-r)^2-4(r+r^{-1})+8>8(1-\cos\theta),
\)
or 
\(
\frac{(r^{-1}-r)^2}{r+r^{-1}-2\cos \theta}>4. 
\)
Upon extracting square roots this inequality takes the form 
\(
\frac{r^{-1}-r}{\sqrt{r+r^{-1}-2\cos \theta}}>2. 
\)
Thus, if $q\in D\setminus\{0\}$, there is a unique solution $a=a(q)$
with $0<a<1$ of the equation
\begin{equation}\label{eqfora}
\frac{r^{-1}-r}{\sqrt{r+r^{-1}-2\cos \theta}}=a+\frac{1}{a}.
\end{equation}
We set $a(0):=0$.  This gives a continuous function on $D$; moreover,
\(
 r^{-1/2}a(q)\longrightarrow1\qquad(q\to0).
\)
Indeed, if the left hand side of \eqref{eqfora} is denoted by $L(q)$,
then
\[
 a(q)=\frac{2}{L(q)+\sqrt{L(q)^2-4}},
 \qquad
 r^{1/2}L(q)=\frac{1-r^2}{|1-q|}\longrightarrow1.
\]
Thus every occurrence of $r^{-1/2}a$ below is understood at $q=0$ to be 
its continuous value $1$.

Theorem \ref{ma1} will follow from 

\begin{theorem}\label{mai} The expansion of $\frac{1}{[x]_q}$ given by \eqref{seri} 
converges absolutely on the region $D$ and uniformly on its compact subsets
to a nonvanishing holomorphic function $\psi_x(q)$. In particular, $\psi_x(q)$ is analytic on $(-3+2\sqrt{2},1)$.

Moreover, for $q\in D$ and $N\ge1$,
\begin{equation}\label{boundd}
\left|\frac{1}{[x]_{N,q}}-\frac{1}{[x]_q}\right|
\le \frac{r^{-1/2}a}{1-a^2}
 r^{\sum_{i=1}^{N} (c_i-2)}a^{2N}
\end{equation}
and
\begin{equation}\label{boundd1}
|[x]_{N,q}-[x]_q|
\le \frac{r^{-1/2}a}{1-a^2}
 r^{\sum_{i=1}^{N} (c_i-2)}a^{2N-2},
\end{equation}
where $0\le a<1$ is defined above.  The right hand sides at $q=0$ are
interpreted by continuous extension.  With the convention in the next
remark, \eqref{boundd} is valid for $N=0$ as well.
\end{theorem} 

\begin{remark}
For $N=0$ it is natural to set $[x]_{N,q}:=\infty$, so that
$1/[x]_{N,q}:=0$. Then \eqref{boundd} gives
\(
 |[x]_q|\ge r^{1/2}a^{-1}(1-a^2)
\)
for $q\in D$ and $x>1$, where the right hand side has the continuous
value $1$ at $q=0$.
\end{remark}

\begin{proof} Let $\alpha:=r^{1/2}a^{-1}$, with $\alpha(0):=1$.

\begin{lemma}\label{keybound} 
$|a_{N+1}(q)|\ge \alpha |a_N(q)|.$
\end{lemma} 

\begin{proof}
For $q=0$ the claim is immediate, since $a_N(0)=1$ and $\alpha(0)=1$.
Thus assume $q\ne0$.
The proof is by induction in $N$. From \eqref{rec2} we get 
$$
|a_{N+1}|\ge \frac{1-|q|^{c_{N+1}}}{|1-q|}|a_N|-|q|^{c_N-1}|a_{N-1}|.
$$
Thus
\begin{equation}\label{eqq2}
|a_{N+1}|\ge \left(\frac{1-|q|^{c_{N+1}}}{|1-q|}-|q|^{c_N-1}\alpha^{-1}\right)|a_N|+|q|^{c_N-1}\alpha^{-1}(|a_N|-\alpha|a_{N-1}|).
\end{equation} 
Note that
\begin{equation}\label{eqq1}
\frac{1}{\sqrt{r^2+1-2r\cos \theta}}-\alpha= r\left(\frac{r}{\sqrt{r^2+1-2r\cos \theta}}+\alpha^{-1}\right).
\end{equation}
Set $s=r^{c_N-1}$, $t=r^{c_{N+1}-1}$, then $s,t\le r$. Thus \eqref{eqq1} implies that
$$
\frac{1}{\sqrt{r^2+1-2r\cos \theta}}-\alpha\ge \frac{tr}{\sqrt{r^2+1-2r\cos \theta}}+s\alpha^{-1},
$$
or, equivalently, 
$$
\frac{1-rt}{\sqrt{r^2+1-2r\cos \theta}}-s\alpha^{-1}\ge \alpha.
$$
Hence for any integers $n,m\ge 2$, 
\begin{equation}\label{eqq3}
\frac{1-|q|^{n}}{|1-q|}-|q|^{m-1}\alpha^{-1}\ge \alpha,
\end{equation}
which justifies the base of induction $N=0$. Also \eqref{eqq2} and \eqref{eqq3} imply that  
$$
|a_{N+1}|-\alpha |a_{N}|\ge r^{c_N-1}\alpha^{-1}(|a_N|-\alpha|a_{N-1}|),
$$
which justifies the induction step. 
\end{proof} 

Lemma \ref{keybound} implies that $|a_N(q)|\ge \alpha^N$. 
Hence
\begin{equation}\label{eqq4}
\left|\frac{q^{C_j}}{a_j(q)a_{j+1}(q)}\right|
\le r^{j+\sum_{i=1}^j(c_i-2)}\alpha^{-2j-1}
\le r^j\alpha^{-2j-1}
=r^{-1/2}a^{2j+1}=(r^{-1/2}a)a^{2j}.
\end{equation}
If $K\Subset D$ is compact, continuity of $a$ on $D$, including at
$q=0$, gives
\[
 \rho_K:=\max_{q\in K}a(q)<1,
 \qquad
 M_K:=\max_{q\in K}r^{-1/2}a(q)<\infty.
\]
Thus the $j$-th summand in \eqref{seri} is bounded on $K$ by
$M_K\rho_K^{2j}$.  The Weierstrass test therefore gives absolute
convergence on $D$ and uniform convergence on compact subsets.  Since
each summand is holomorphic on $D$, the sum is holomorphic there.

To prove \eqref{boundd}, note that 
$$
\left|\psi_x(q)-\frac{1}{[x]_{N,q}}\right|\le \sum_{j=N}^\infty \left|\frac{q^{C_j}}{a_j(q)a_{j+1}(q)}\right|,
$$
so by \eqref{eqq4}
\[
\left|\psi_x(q)-\frac{1}{[x]_{N,q}}\right|
\le \sum_{j=N}^\infty
 r^{\sum_{i=1}^j(c_i-2)}(r^{-1/2}a)a^{2j}
\le r^{\sum_{i=1}^{N}(c_i-2)}
\frac{(r^{-1/2}a)a^{2N}}{1-a^2}.
\]
Now, inequality \eqref{boundd1} and fact that $\psi_x(q)$ does not vanish in $D$ follow from \eqref{boundd} and 
Lemma \ref{modinv}. Indeed, let 
\[
y=[[c_2,c_3,\ldots]],\qquad
y_{N-1}=[[c_2,\ldots,c_N]]\quad(N\ge2).
\]
For $q\ne0$, formula \eqref{xp1q}, together with Lemma
\ref{modinv} applied to $-y_{N-1}$, gives
\[
[x]_{N,q}
=[c_1-1/y_{N-1}]_q
=[c_1]_q+q^{c_1}[-1/y_{N-1}]_q
=[c_1]_q-\frac{q^{c_1-1}}{[y]_{N-1,q}}.
\]
The identity
$[x]_{N,q}
=[c_1]_q-\frac{q^{c_1-1}}{[y]_{N-1,q}}$
also holds at $q=0$ by continuity, and, with
the convention $1/[y]_{0,q}:=0$, it holds for $N=1$ as well.
Define
\(
F_x(q):=[c_1]_q-q^{c_1-1}\psi_y(q).
\)
Then $F_x$ is holomorphic on $D$, and applying \eqref{boundd}
to $y$ with $N-1$ in place of $N$ (with the same empty-sum
convention when $N=1$) yields
\begin{align*}
|[x]_{N,q}-F_x(q)|
&=r^{c_1-1}
  \left|\frac{1}{[y]_{N-1,q}}-\psi_y(q)\right|\\
&\le
\frac{r^{-1/2}a}{1-a^2}\,
r^{c_1-1+\sum_{i=2}^{N}(c_i-2)}a^{2N-2}\\
&=
\frac{r^{-1/2}a}{1-a^2}\,
r^{1+\sum_{i=1}^{N}(c_i-2)}a^{2N-2}\\
&\le 
\frac{r^{-1/2}a}{1-a^2}\,
r^{\sum_{i=1}^{N}(c_i-2)}a^{2N-2}.
\end{align*}
since $r\le1$. This is \eqref{boundd1}. Since the same
convergents tend $q$-adically to the formal series $[x]_q$, the
holomorphic limit $F_x$ is its analytic realization, so we may
write $F_x=[x]_q$.

Finally, both $[x]_{N,q}\to F_x$ and
$1/[x]_{N,q}\to\psi_x$ locally uniformly on $D$. Passing to the
limit in the identity 
$[x]_{N,q}\cdot \frac{1}{[x]_{N,q}}=1$
gives $F_x(q)\psi_x(q)=1$. 
Hence neither $F_x$ nor $\psi_x$ vanishes on $D$, and
$F_x=1/\psi_x$ there.
  \end{proof} 
  
 \begin{corollary} The Taylor series for the function $\psi_x(q)$ 
 at $q=0$ is $1/[x]_q$, and this function does not vanish on $D$. 
 Thus the series $[x]_q$ converges for $|q|< 3-2\sqrt{2}$ to a non-vanishing holomorphic function.
  \end{corollary} 
  
\begin{proof} The closest point of $\partial D$ to the origin corresponds to $\theta=\pi$, in which case $r=3-2\sqrt{2}$. Thus, repeatedly differentiating series \eqref{seri} at $0$ and using uniform convergence of Theorem \ref{mai} and the Cauchy integral representation of derivatives, we get that the Taylor series of $\psi_x(q)$ is $1/[x]_q$. 
\end{proof}  
 
This completes the proof of Theorem \ref{ma1}.  
    
\subsection{Behavior of the function $x\mapsto [x]_q$ when $q\in D$}
\begin{corollary} Let $q\in D$. For any $x,y>1$ with $x_N=y_N$
 \begin{equation}\label{boundd2}
 \left|\frac{1}{[x]_q}-\frac{1}{[y]_q}\right|
 \le \frac{2r^{-1/2}a}{1-a^2}
 r^{\sum_{i=1}^{N} (c_i-2)}a^{2N}
\end{equation}
 and 
\begin{equation}\label{boundd3}
 |[x]_q-[y]_q|
 \le \frac{2r^{-1/2}a}{1-a^2}
 r^{\sum_{i=1}^{N} (c_i-2)}a^{2N-2}.
\end{equation}
\end{corollary}

\begin{proof} 
Note that $[x]_{N,q}=[x_N]_q=[y_N]_q=[y]_{N,q}$. 
 Thus \eqref{boundd2},\eqref{boundd3} follow from Theorem \ref{mai} and the triangle inequality. 
 \end{proof} 
 
\begin{proposition}\label{ratcon} (i) The function $x\mapsto [x]_q$ is continuous at irrational $x$ and right-continuous at rational $x$.

(ii) If $q\ne 0$ then the function $x\mapsto [x]_q$ is discontinuous from the left
at every rational point $x=[[c_1,...,c_N]]$. Moreover, 
\([x]_q^-:=\lim_{y\to x-}[y]_q=[[c_1,...,c_N,\infty]]_q,\) where 
$$
[[c_1,...,c_N,\infty]]_q:=\lim_{n\to \infty}[[c_1,...,c_N,n]]_q=
[c_1]_q-\frac{q^{c_1-1}}{[c_2]_q-...\frac{q^{c_{N-1}-1}}{[c_N]_q-q^{c_N-1}(1-q)}}.
$$
Thus the jump of this function at $x$ is 
$$
{\rm Jump}_x(q):=[x]_q-[x]_q^-=[[c_1,...,c_N]]_q-[[c_1,...,c_N,\infty]]_q,
$$
which is positive for $0<q<1$.

(iii) We have 
$$
\lim_{x\to +\infty}[x]_q=\frac{1}{1-q};\quad
\lim_{x\to -\infty}|[x]_q|=\infty,\ q\ne 0.
$$
\end{proposition} 

\begin{proof} (i) If $x(m)\to x$ for irrational $x$ or $x(m)\to x$ from the right
for rational $x$ then for each $i$ the number $c_i(x(m))$ stabilizes as $m\to \infty$. 
Thus (i) follows from Theorem \ref{mai}. 
 
(ii) It is easy to see that if $q\ne 0$ and $x\in \Bbb Q$ then 
$[x]_q^-\ne [x]_q$. Indeed,  
$$
[[c_1,c_2,...,c_N]]_q=[c_1]_q-\frac{q^{c_1-1}}{[[c_2,...,c_N]]_q},\
[[c_1,c_2,...,c_N,\infty]]_q=[c_1]_q-\frac{q^{c_1-1}}{[[c_2,...,c_N,\infty]]_q},
$$
so $[[c_1,...,c_N]]_q=[[c_1,...,c_N,\infty]]_q$ implies 
$[[c_2,...,c_N]]_q=[[c_2,...,c_N,\infty]]_q$, thus by induction 
$[c_N]_q=[c_N]_q-q^{c_N-1}(1-q)$, a contradiction.
The rest of the proof is analogous to the proof of (i). 

(iii) By Lemma \ref{modinv} and (i),(ii) 
$$
\lim_{x\to +\infty}[x]_q=\lim_{x\to 0-}[-\tfrac{1}{x}]_q=-\lim_{x\to 0-}\tfrac{q^{-1}}{[x]_q}=-\frac{q^{-1}}{[0]_q^-}=-\frac{q^{-1}}{1-q^{-1}}=\frac{1}{1-q},
$$
$$
\lim_{x\to -\infty}|[x]_q|=\lim_{x\to 0+}|[-\tfrac{1}{x}]_q|=\lim_{x\to 0+}\tfrac{|q|^{-1}}{|[x]_q|}=\infty.
$$
\end{proof} 

Proposition \ref{ratcon} implies that for $q\in D$, to define the function $x\mapsto [x]_q$ for irrational $x$, we may simply extend it by continuity from $x\in \Bbb Q$. 

\begin{remark} Theorem \ref{ma1}, Proposition \ref{ratcon} and Corollary \ref{qin} imply that for $q\in D\setminus 0$, we can define the function $[x]_{q^{-1}}$, $x\in \Bbb R$, which 
is holomorphic on $D\setminus\{0\}$ and meromorphic on $D$, namely $[x]_{q^{-1}}=-q[-x]_q$, and that as a function of $x$ 
it is continuous at irrational points and left-continuous at rational points. 
Also it has a right limit $[x]_{q^{-1}}^+=\lim_{y\to x+}[y]_{q^{-1}}$, namely 
$[x]_{q^{-1}}^+=-q[-x]_q^-$. 
\end{remark} 

\subsection{Proof of Theorem \ref{ma1a}} 

The proof of Theorem \ref{ma1a} is based on the result of \cite{EGMS} on the location of the roots 
of the $q$-continuant polynomials $a_N(\bold c|q)$. Namely, Morier-Genoud, Ovsienko and Veselov showed in \cite{MOV}, {Theorems 1 and 2, that, for $q\in\Bbb C^*$, the Burau representation of the braid group $B_3$ specialized at $t=-q$ is not faithful if and only if either $q=1$ or $q$ is a root of one of these polynomials}, and that the non-faithfulness locus is contained in the annulus 
\(
3-2\sqrt{2}\le |q|\le 3+2\sqrt{2}.
\)
They also conjectured that the annulus can be shrunk to 
\(
R_*\le |q|\le R_*^{-1},
\)
and this was proved in \cite{EGMS}, Corollary 6.5. This implies that the series $[x]_q$ for every rational $x$ converges 
in the disk $|q|<R_*$.

For rational $x$, the one-sided values in the next lemma are understood
as rational functions of $q$, obtained by continuing the expressions
already defined in Proposition \ref{ratcon}.

\begin{lemma}\label{nonzero}
For $0<|q|<R_*$ and $x\in\Bbb Q\setminus\{0\}$, the four values
$[x]_q$, $[x]_q^-$, $[x]_{q^{-1}}$ and $[x]_{q^{-1}}^+$
are finite and nonzero.
Moreover, for any finite sequence $c_1,\ldots,c_N\ge2$, all roots of
\[
 P_N(q):=a_N(q)-q^{c_N-1}(1-q)a_{N-1}(q)
\]
lie in the annulus $R_*\le |q|\le R_*^{-1}$.
\end{lemma}

\begin{proof}
The root bound for ordinary continuants implies that $[x]_q$ is finite
for every rational $x$ and $0<|q|<R_*$, first for $x>1$ and then by
translation.  For $x\ne0$, the identity
$[x]_q[-1/x]_q=-q^{-1}$ gives nonvanishing.

Fix $x=[[c_1,\ldots,c_N]]>1$ and put
\[
 Q_N(q):=b_N(q)-q^{c_N-1}(1-q)b_{N-1}(q).
\]
For integers $m\ge2$, the numerator and denominator continuants of
$[[c_1,\ldots,c_N,m]]_q$ are
\[
 A_m(q)=[m]_q a_N(q)-q^{c_N-1}a_{N-1}(q),\qquad
 B_m(q)=[m]_q b_N(q)-q^{c_N-1}b_{N-1}(q).
\]
Both are zero-free on $|q|<R_*$ by the ordinary-continuant root bound.
Since $[m]_q\to(1-q)^{-1}$ locally uniformly on $|q|<1$, we have
\[
 A_m(q)\longrightarrow\frac{P_N(q)}{1-q},\qquad
 B_m(q)\longrightarrow\frac{Q_N(q)}{1-q}
\]
locally uniformly on $|q|<R_*$.  Both limits have value $1$ at $q=0$.
Hurwitz's theorem therefore shows that both limits are zero-free there.
Consequently
\[
 H_x(q):=\frac{P_N(q)}{Q_N(q)}
\]
is finite and nonzero on this disk, and agrees with $[x]_q^-$ on $D$.

For an arbitrary rational $x$, choose an integer $n\ge0$ such that
$x+n>1$ and define
\[
 H_x(q):=q^{-n}\bigl(H_{x+n}(q)-[n]_q\bigr).
\]
The definitions agree on $D\setminus\{0\}$, by Proposition
\ref{ratcon} and translation equivariance, so they agree as rational
functions.  Thus every $H_x$ is finite on $0<|q|<R_*$.  The identities
\[
 H_{x+1}(q)=qH_x(q)+1,\qquad
 H_x(q)H_{-1/x}(q)=-q^{-1}\quad(x\ne0)
\]
hold first on $D\setminus\{0\}$, by taking the one-sided limits
there, and hence hold as rational identities.  This proves nonvanishing
of $H_x$ for $x\ne0$; also $H_0(q)=1-q^{-1}$.

The same modular-equivariance argument as in Corollary \ref{qin},
starting from $H_0(q)=1-q^{-1}$ (and
$H_\infty(q)=(1-q)^{-1}$), gives
\[
 H_x(q^{-1})=-qH_{-x}(q).
\]
Together with \eqref{qinv} and the right-limit formula preceding this
subsection, this proves the remaining assertions about the four values.
In particular, $H_x(q)$ is finite and nonzero for $|q|>R_*^{-1}$.
Finally, the continuant determinant identity gives
\[
 P_N(q)b_N(q)-a_N(q)Q_N(q)=-(1-q)q^{C_N}.
\]
Hence $P_N$ and $Q_N$ have no common zero outside $\{0,1\}$.
Since $P_N(0)=1$ and $P_N(1)=a_N(1)>0$, no zero of $P_N$ can be
cancelled in $H_x=P_N/Q_N$.  Its inner and outer nonvanishing therefore
place all roots of $P_N$ in the asserted annulus.
\end{proof}

Let $A_R$ be the annulus $R\le |z|\le R^{-1}$, where $R\le 1$.
Let $S(d,R)$ be the (compact) set of monic complex polynomials of degree $d$ and constant term $1$ with all roots contained in $A_R$. For $0<r<R$ and a polynomial $a\in S(d,R)$, let $M_{a,r}$ be the minimal value of $|a(z)|$ on the circle $|z|=r$ (or, equivalently, on the disk $|z|\le r$), and let $M(d,R,r):=\min_{a\in S(d,R)} M_{a,r}$.

\begin{lemma}\label{auxle} For even $d$
 $$
M(d,R,r)=(R-r)^{\frac{d}{2}}(R^{-1}-r)^{\frac{d}{2}},
$$
 and for odd $d$
$$
M(d,R,r)=(R-r)^{\frac{d-1}{2}}(R^{-1}-r)^{\frac{d-1}{2}}(1-r).
$$
\end{lemma}

\begin{proof} Write $a(q)=(q-z_1)...(q-z_d)$. This
identifies $S(d,R)$ with the set of tuples
$(z_1,...,z_d)\in {\rm Sym}^dA_R$ 
with $z_1...z_d=(-1)^d$. Let $a\in S(d,R)$ 
be a minimizer for $M_{a,r}$ and $q$ be a point with $|q|=r$ such that 
$|a(q)|=M(d,R,r)$. If $z_1,z_2$ lie in the interior of $A_R$ then by the open mapping theorem they can be tweaked so that $z_1z_2$ stays constant but the absolute value of $(q-z_1)(q-z_2)$ decreases. Such a deformation with 
the rest of $z_i$ fixed keeps $(z_1,...,z_d)$ in $S(d,R)$, which contradicts the assumption that $a$ is a minimizer. Thus all but one $z_j$ must be on the boundary of $A_R$. 

Moreover, for $R<1$ it follows that if no roots of $a(z)$ are in the interior of $A_R$ then $d$ is even and $d/2$ of them are on the inner boundary and $d/2$ on the outer boundary, and if one root is in the interior then it has absolute value $1$, $d$ is odd,  and $(d-1)/2$ roots are on the inner boundary and $(d-1)/2$ on the outer boundary. Indeed, let $k$ roots
lie on $|z|=R$ and $\ell$ roots on $|z|=R^{-1}$, counted with
multiplicity. Since the constant term of $a$ is $1$,
$|z_1\cdots z_d|=1$. If there is no interior root, then
$1=R^kR^{-\ell}=R^{k-\ell}$, so $k=\ell$; hence $d$ is even and $k=\ell=d/2$. If there is one
interior root $w$, then $1=|w|R^{k-\ell},\ \text{so}\ |w|=R^{\ell-k}$.
Because $\ell-k\in\mathbb Z$ and $R<|w|<R^{-1}$, the only possibility
is $\ell-k=0$. Thus $|w|=1$, $k=\ell$, and
$k=\ell=(d-1)/2$, so $d$ is odd. Finally,
\begin{equation}\label{rootineq}
|q-z|\ge\bigl||z|-r|
\end{equation}
 for every root $z$.

Thus we get 
\(
|a(q)|\ge (R-r)^{\frac{d}{2}}(R^{-1}-r)^{\frac{d}{2}}
\)
if $d$ is even and 
$$
|a(q)|\ge (R-r)^{\frac{d-1}{2}}(R^{-1}-r)^{\frac{d-1}{2}}(1-r)
$$
if $d$ is odd; if $R=1$, these inequalities 
both reduce to $|a(q)|\ge (1-r)^d$ and follow immediately 
from \eqref{rootineq}. Moreover, these bounds are sharp and attained for 
$a(z)=(z+R)^{\frac{d}{2}}(z+R^{-1})^{\frac{d}{2}}$ for even $d$ and 
$a(z)=(z+R)^{\frac{d-1}{2}}(z+R^{-1})^{\frac{d-1}{2}}(z+1)$
for odd $d$, at $q=-r$. 
\end{proof}

Now we are ready to prove Theorem \ref{ma1a}. Let $R:=R_*$ and $s:=(R-r)(R^{-1}-r)=1-3r+r^2$. Suppose $r<s$, which happens iff $1-4r+r^2>0$, i.e., iff $r<2-\sqrt{3}$. Let $|q|\le r$. 
By Corollary 6.5 of \cite{EGMS}, for all $N\ge 0$ we have $a_N(q)\ne 0$.
 
\begin{lemma}\label{auxle1} For any $0\le \beta<1$ we have 
$$
\left|\frac{q^{C_N}}{a_N(q)a_{N+1}(q)}\right|=o((r/s)^{\beta C_N}),\ N\to \infty, 
$$
uniformly in $q$.
\end{lemma} 

\begin{proof}
Suppose that the asserted uniform little-$o$ estimate fails.  Then there
are a constant $K>0$, integers $N_k\to\infty$, and points $q_k$ with
$|q_k|\le r$ such that
\begin{equation}\label{ee}
\left|\frac{q_k^{C_{N_k}}}
 {a_{N_k}(q_k)a_{N_k+1}(q_k)}\right|
\ge K(r/s)^{\beta C_{N_k}}.
\end{equation}
In the rest of the proof write $N=N_k$ and $q=q_k$.  Lemma \ref{auxle}
gives a constant
\(
 L:=\min\left\{1,\frac{1-r}{\sqrt{s}}\right\}>0
\)
such that, uniformly for $|q|\le r$,
\begin{equation}\label{e0}
 |a_N(q)|\ge Ls^{C_N/2}.
\end{equation}
Consequently, \eqref{ee} implies
\begin{equation}\label{e1}
 |a_{N+1}(q)|
 \le\frac{(r^{1-\beta}s^{\beta-1/2})^{C_N}}{KL}.
\end{equation}

Consider the polynomial
\[
 P_N(q)=a_N(q)-q^{c_N-1}(1-q)a_{N-1}(q)
\]
from Lemma \ref{nonzero}.  That lemma places all its roots in
$A_{R_*}$.  Its constant term is $1$, and its unique term of degree
$1+C_N$ comes from $q^{c_N}a_{N-1}(q)$ and has coefficient $1$.
Thus Lemma \ref{auxle} gives, with the same $L$,

\begin{equation}\label{e2}
 |P_N(q)|\ge Ls^{(1+C_N)/2}.
\end{equation}
On the other hand, the recurrence and \eqref{e1} give
\[
 |P_N(q)-q^{c_{N+1}}a_N(q)|
 \le \frac{1+r}{KL}
 \bigl(r^{1-\beta}s^{\beta-1/2}\bigr)^{C_N}.
\]
The ratio of the right hand side to the lower bound in \eqref{e2} is
$O((r/s)^{(1-\beta)C_N})$, hence tends to zero uniformly.  Therefore,
for all sufficiently large $k$,
\(
 |a_N(q)|\ge\tfrac12Ls^{(1+C_N)/2}r^{-c_{N+1}}.
\)
Using this estimate, $C_{N+1}\ge C_N+1$, and \eqref{e0} at $N+1$, we
obtain
\begin{align*}
\left|\frac{q^{C_N}}{a_N(q)a_{N+1}(q)}\right|
&\le \frac{2r^{C_{N+1}}}
 {Ls^{(1+C_N)/2}|a_{N+1}(q)|}\\
&\le \frac{2}{L^2}
 (r/s)^{(1+C_N+C_{N+1})/2}
 =O((r/s)^{C_N}).
\end{align*}
Since $r/s<1$ and $\beta<1$, this contradicts \eqref{ee}.  The
little-$o$ estimate is therefore uniform on $|q|\le r$.
\end{proof}

For every $r<2-\sqrt3$ we have $r<s$.  Taking, for example,
$\beta=1/2$ in Lemma \ref{auxle1}, and using $C_N\ge N$, shows that
\eqref{seri} converges absolutely and uniformly on $|q|\le r$.
Now let $K$ be a compact subset of the open disk
$|q|<2-\sqrt3$.  Choose $r$ with
$\max_{q\in K}|q|<r<2-\sqrt3$ and apply the preceding conclusion.
Thus the series converges uniformly on every such $K$.  The rest of
the proof of Theorem \ref{ma1a} is the same as for Theorem \ref{ma1};
in particular, nonvanishing follows from Lemma \ref{modinv}.

\subsection{Formulas for the jump of $[x]_q$}
Note that by \eqref{xp1q} 
\(
{\rm Jump}_{x+1}(q)=q{\rm Jump}_x(q).
\)
Also recall that 
$$
\prod_{i=0}^{N-1} [[c_{i+1},...,c_N]]_q=a_N(c_1,...,c_N|q).
$$

\begin{proposition}\label{jumpbo} If $x=[[c_1,...,c_N]]>1$ then for any $m\ge 2$
$$
{\rm Jump}_{m-\frac{1}{x}}(q)=\frac{q^{m-1}}{[x]_q[x]_q^-}{\rm Jump}_x(q).
$$
Thus for $1\le j\le N-1$
$$
{\rm Jump}_x(q)={\rm Jump}_{x(j)}(q)\prod_{i=1}^{j} \frac{q^{c_i-1}}{[x(i)]_q[x(i)]_q^-},
$$
where $x(i):=[[c_{i+1},...,c_N]]$. Hence
$$
{\rm Jump}_x(q)=q^{c_N-1}(1-q)\prod_{i=1}^{N-1} \frac{q^{c_i-1}}{[x(i)]_q[x(i)]_q^-}=
\frac{q^{C_N}}{a_{N-1}(c_2,...,c_N|q)a_{N}(c_2,...,c_N,\infty|q)},
$$
where 
$$
a_{N}(c_2,...,c_N,\infty|q):=\lim_{n\to \infty}a_{N}(c_2,...,c_N,n|q).
$$ 
\end{proposition} 

\begin{proof} This follows directly from the definitions.
\end{proof} 

\begin{example}\label{exajump} For $m\in \Bbb Z_{\ge 1},n\in \Bbb Z_{\ge 2}$ we have 
$$
{\rm Jump}_m(q)=q^{m-1}(1-q),\ {\rm Jump}_{m-\frac{1}{n}}(q)=
\frac{q^{m+n-2}(1-q)}{(1+q+...+q^{n-2}+q^{n-1})(1+q+...+q^{n-2}+q^n)},
$$
$$
{\rm Jump}_{m+\frac{1}{n}}(q)=\frac{q^{m+n-1}(1-q)}{(1+q+...+q^{n-2}+q^{n-1})(1+q^2+...+q^{n-1}+q^n)}.
$$
\end{example}

 \subsection{A counterexample to the $q$-deformed Hurwitz irrational number conjecture} \label{counte} 

It is clear that if $x$ is in the $PSL_2(\Bbb Z)$-orbit $O_\varphi$ of the golden ratio $\varphi=\frac{1+\sqrt{5}}{2}$ then the radius of convergence of the power series $[x]_q$ is $R_*:=\frac{3-\sqrt{5}}{2}$. 
In \cite{LMOV}, Conjecture 1.1 it is conjectured that 

(1) for every real $x$ the radius of convergence $R(x)$ of $[x]_q$
is at least $R_*$, and 

(2) if $x\notin O_\varphi$ then $R(x)$ is strictly bigger than $R_*$. 

Yet we construct below uncountably many $x\notin O_\varphi$ for which
$R(x)\le R_*$ (i.e., counterexamples to (2)).\footnote{A proof of (1) will be given in the 
forthcoming paper \cite{EO}.}

Define a real number $x$ by the negative continued fraction
\[
  x = \Bigl[\!\!\Bigl[\,\underbrace{2,3,\dots,3}_{n_1},\,
              \underbrace{2,3,\dots,3}_{n_2-n_1},\,
              \underbrace{2,3,\dots,3}_{n_3-n_2},\ldots\Bigr]\!\!\Bigr],
\]
where the sequence $\{n_i\}$ grows sufficiently fast. We define $n_{m}$ inductively, assuming $n_1,\dots,n_{m-1}$ have already been defined. Let
\[
  x(m)=\Bigl[\!\!\Bigl[\,\underbrace{2,3,\dots,3}_{n_1},\,
              \underbrace{2,3,\dots,3}_{n_2-n_1},\,\ldots,\,
              \underbrace{2,3,\dots,3}_{n_m-n_{m-1}},\,2,\overline{3}\Bigr]\!\!\Bigr],
\]
where $\overline 3$ denotes an infinite sequence of threes. Let $\beta_i(m)$ be the coefficients of $[x(m)]_q$ and $\beta_i$ those of $[x]_q$, where we set \(x(0)=[[2,\overline 3]]=\varphi\) and let \(\beta_i(0)\) denote the coefficients of \([\varphi]_q\).
Then
\(
\beta_i(m-1)=\beta_i,\ i\le n_{m}.
\)
Since $x(m-1)$ is a $PSL_2(\mathbb Z)$-image of $\varphi = \bigl[\!\!\bigl[2,\overline 3\bigr]\!\!\bigr]$, we have
\[
  \limsup_{n\to\infty} \lvert \beta_n(m-1)\rvert^{1/n} = \frac{1}{\Rstar}.
\]
Thus there exist infinitely many choices $n_{m} > n_{m-1}$ such that
\[
  \lvert \beta_{n_{m}}(m-1)\rvert^{1/n_{m}} \ge \frac{1}{\Rstar + \dfrac{1}{m}}.
\]
So we pick this $n_{m}$. 

Since  $\beta_{n_{m}}(m-1) = \beta_{n_{m}}$, we have 
\[
  \lvert \beta_{n_m}\rvert^{1/n_m} \ge \frac{1}{\Rstar + \dfrac{1}{m}}.
\]
for all $m$. Hence
\[
  \limsup_{n\to\infty} \lvert \beta_n\rvert^{1/n} \ge \frac{1}{\Rstar},
\]
which means that the radius of convergence of $[x]_q$ is $\le \Rstar$. Since at each step we have infinitely many choices, we obtain uncountably many counterexamples. 

\subsection{Behavior of $[x]_q$ when $q<0$} \label{realbeh}
 
 We would now like to extend the function $[x]_q$ to $-R_*<q\le -3+2\sqrt{2}$. 
 
\begin{theorem}\label{ma2} Let $x>1$. Then the series $\frac{1}{[x]_q}=\sum_{N\ge 0}\frac{q^{C_N}}{a_N(q)a_{N+1}(q)}$ converges absolutely on $(-R_*,1)$ and uniformly on compact sets to a continuous positive function of $q$. 
\end{theorem}  
 
\begin{proof} 
In view of Theorem \ref{ma1} it suffices to consider the case 
$q<0$. Let $q=-r$ for $0<r<R_*$. Then  
$r^{-1}-1+r>2$ and $r^{-1/2}-r^{1/2}>1$. Let $a>1$ solve the equation 
\(
 a+a^{-1}=r^{-1}-1+r,
 \)
 i.e., 
 $$
 a= \frac{1}{2}(r -1+r^{-1}+ \sqrt{(r - 3 +r^{-1})(r+1+r^{-1})}).
 $$
 Let 
 $$
 R_N:=\frac{a_{N+1}(q)}{a_{N}(q)},\ Q_N:=r^{-\frac{c_{N+1}-1}{2}}R_N.
 $$ 
We have 
\begin{equation}\label{recrn}
 R_{N}=[c_{N+1}]_q-\frac{q^{c_N-1}}{R_{N-1}},\ R_0=[c_1]_q.
\end{equation}
Hence 
\begin{equation}\label{recrn1}
 Q_N=\frac{r^{-c_{N+1}/2}-(-1)^{c_{N+1}}r^{c_{N+1}/2}}{r^{-1/2}+r^{1/2}}
 +(-1)^{c_N}\frac{r^{\frac{c_N-c_{N+1}}{2}}}{Q_{N-1}},\ Q_0=\frac{r^{-c_1/2}-(-1)^{c_1}r^{c_1/2}}{r^{1/2}+r^{-1/2}}.
\end{equation}

Let us find the conditions on $b>0$ such that we can prove by induction 
that $Q_N\ge a$ if $c_{N+1}$ is odd and $Q_N\ge b$ if $c_{N+1}$ is even. 

Note that the functions $x\mapsto x^{-1}\pm x$ are decreasing on $(0,1)$. Thus  \eqref{recrn1} implies that $Q_0\ge r^{-1}-1+r>a$ if $c_{1}$ is odd and 
$Q_0\ge r^{-1/2}-r^{1/2}$ when $c_{1}$ is even. 
So first of all we must have 
\(
b\le r^{-1/2}-r^{1/2}.
\)

Now consider 4 cases. 

{\bf Case 11:} $c_N$ and $c_{N+1}$ are odd. 
Consider two subcases. 

{\bf Subcase 11a:} $c_{N+1}=3$.  
In this case by the induction assumption \eqref{recrn} gives 
$$
Q_N= r^{-1}-1+r-\frac{r^{\frac{c_N-c_{N+1}}{2}}}{Q_{N-1}}\ge r^{-1}-1+r-\frac{1}{a}=a,
$$
as $c_N-c_{N+1}=c_N-3\ge 0$. 

{\bf Subcase 11b:} $c_{N+1}\ge 5$.  
In this case \eqref{recrn} 
gives 
$$
R_N= \frac{1+r^{c_{N+1}}}{1+r}-\frac{r^{c_N-1}}{R_{N-1}}\ge \frac{1}{1+r}-\frac{r^{\frac{c_N-1}{2}}}{Q_{N-1}}\ge \frac{1}{1+r}-\frac{r}{Q_{N-1}},
$$
again using that $c_N-3\ge 0$. So by the induction assumption
$$
Q_N\ge r^{-\frac{c_{N+1}-1}{2}}R_N\ge r^{-\frac{c_{N+1}-1}{2}}(\tfrac{1}{1+r}-\tfrac{r}{a})=
$$
$$
r^{-\frac{c_{N+1}-5}{2}}(a+(r^{-1}-1)(a-\tfrac{r^2}{1-r^2}))>r^{-\frac{c_{N+1}-5}{2}}a\ge a,
$$
as $a>1>\frac{r^2}{1-r^2}$.

{\bf Case 01:} $c_N$ is even, $c_{N+1}$ is odd. In this case \eqref{recrn1} gives 
$$
Q_N\ge \frac{r^{-c_{N+1}/2}-(-1)^{c_{N+1}}r^{c_{N+1}/2}}{r^{-1/2}+r^{1/2}}\ge r^{-1}-1+r>a.
$$
Moreover, if $c_N=2$ then using that $c_N-c_{N+1}=2-c_{N+1}\le -1$, \eqref{recrn1} gives
$$
Q_N\ge  \frac{r^{-c_{N+1}/2}-(-1)^{c_{N+1}}r^{c_{N+1}/2}}{r^{-1/2}+r^{1/2}}+\frac{r^{-\frac{c_{N+1}-2}{2}}}{Q_{N-1}} \ge r^{-1}-1+r+\frac{r^{-1/2}}{Q_{N-1}}.
$$

{\bf Case 10:} $c_N$ is odd, $c_{N+1}$ is even. 
In this case using the inequality $c_{N}-1\ge 2$ and the induction assumption, 
$$
R_N\ge  1-r-\frac{r^{\frac{c_N-1}{2}}}{a}\ge 1-r-\frac{r}{a}>r^{1/2}-r>r^{3/2},
$$
so 
\(
Q_N\ge r^{-1/2}R_N> 1-r^{1/2}>r.
\)
Moreover, if $c_{N+1}\ge 4$ then 
$$
Q_N\ge r^{-\frac{c_{N+1}-1}{2}}R_N>r^{-\frac{c_{N+1}-4}{2}}> 1.
$$
Thus $b$ must satisfy the inequality 
\(
b\le r^{-1/2}-r^{1/2}-\frac{r^{1/2}}{a}.
\)

{\bf Case 00:} $c_N$ and $c_{N+1}$ are even. In this case \eqref{recrn1} gives 
$$
Q_N\ge \frac{r^{-c_{N+1}/2}-(-1)^{c_{N+1}}r^{c_{N+1}/2}}{r^{-1/2}+r^{1/2}}\ge r^{-1/2}-r^{1/2}.
$$
Moreover, if $c_N=2$ then it gives 
$$
Q_N\ge \frac{r^{-c_{N+1}/2}-(-1)^{c_{N+1}}r^{c_{N+1}/2}}{r^{-1/2}+r^{1/2}}+\frac{r^{-\frac{c_{N+1}-2}{2}}}{Q_{N-1}}\ge r^{-1/2}-r^{1/2}+\frac{1}{Q_{N-1}}.
$$
So the condition for $b$ is again $b\le r^{-1/2}-r^{1/2}$. 

We conclude that the induction goes through if $b\le r^{-1/2}-r^{1/2}-\frac{r^{1/2}}{a}$. The condition for such $b$ to exist is that the right hand side is positive, i.e. that
$a>\frac{r}{1-r},$
which holds since $\frac{1-r}{r}+\frac{r}{1-r}=r^{-1}-1+\frac{r}{1-r}>\frac{1}{a}+a=r^{-1}-1+r$. 
Thus we get 

\begin{lemma}\label{leee1} We have $Q_N>0$ for all $N$, and 
if $c_{N+1}$ is odd then $Q_N\ge a$. 
\end{lemma} 

We also obtain 

\begin{lemma}\label{endpoint} For $N\ge 0$ we have 
\[
 Q_N\ge r^{-\frac{c_{N+1}-5}{2}}.
\]
\end{lemma} 

\begin{proof} 
For $N=0$ this follows directly from formula \eqref{recrn1} for $Q_0$. For $N\ge1$,
Case $11$ gives the claim immediately: if $c_{N+1}=3$, then
$Q_N\ge a>1>r$, while if $c_{N+1}\ge5$, Subcase $11b$ gives the stronger
bound $Q_N>a r^{-(c_{N+1}-5)/2}$.  In Case $01$,
\[
 Q_N\ge \frac{r^{-c_{N+1}/2}}{r^{-1/2}+r^{1/2}}
 =\frac{r^{1/2}}{1+r}r^{-c_{N+1}/2}
 \ge r^{5/2}r^{-c_{N+1}/2},
\]
since $r^2(1+r)<1$.  In Case $10$, the estimate $R_N>r^{3/2}$ gives
\[
 Q_N>r^{3/2}r^{-(c_{N+1}-1)/2}
 =r^2r^{-c_{N+1}/2}\ge r^{5/2}r^{-c_{N+1}/2}.
\]
Finally, in Case $00$,
\[
 Q_N\ge r^{-c_{N+1}/2}
 \frac{1-r^{c_{N+1}}}{r^{-1/2}+r^{1/2}}
 \ge r^{1/2}(1-r)r^{-c_{N+1}/2}
 \ge r^{5/2}r^{-c_{N+1}/2},
\]
because $1-r\ge r^2$ for $0<r<R_*$. 
\end{proof}

Now we are ready to prove Theorem \ref{ma2}. 
Using that $a_N=\prod_{j=0}^{N-1}R_j$ and
$R_j=r^{(c_{j+1}-1)/2}Q_j$, for $N\ge1$ we get
\[
 \left|\frac{q^{C_N}}{a_N(q)a_{N+1}(q)}\right|
 =r^{-\frac{c_{N+1}-1}{2}}Q_0^{-2}Q_N^{-1}
   \prod_{j=1}^{N-1}Q_j^{-2}
 =\frac{r^{\frac{c_1-c_{N+1}}{2}}}{|[c_1]_q|}
   \prod_{j=1}^{N}\frac1{Q_{j-1}Q_j}.
\]
Therefore Lemma \ref{endpoint} gives
\begin{equation}\label{intermineq}
 \left|\frac{q^{C_N}}{a_N(q)a_{N+1}(q)}\right|
 \le r^{-2}Q_0^{-2}\left(\prod_{j=1}^{N-1}Q_j\right)^{-2}.
\end{equation}
Thus to prove the theorem, it suffices to show that 
for some $b>1$ and constant $C>0$, 
\(
\prod_{j=1}^N \frac{Q_j}{b}\ge C
\)
for all $N$; then $\left|\frac{q^{C_N}}{a_N(q)a_{N+1}(q)}\right|=O(b^{-2N})$ as 
$N\to \infty$. Indeed, then \eqref{intermineq} will yield
\[
 \left|\frac{q^{C_N}}{a_N(q)a_{N+1}(q)}\right|
 \le r^{-2}Q_0^{-2}C^{-2}b^{-2(N-1)}=O(b^{-2N}).
\]
 
 Pick $b\in (1,a)$ sufficiently close to $1$. 
In particular, let $b<r^{-1/2}-r^{1/2}$, which can be arranged since $r^{-1/2}-r^{1/2}>1$.
Let $T_{ij}(N)\subset [1,N]$ for $i,j\in \lbrace 0,1\rbrace$
be the subset of $k$ such that $c_k,c_{k+1}$ are 
congruent mod $2$ to $i$ and $j$ respectively. 

If $j\in T_{11}(N)$ then by Case 11, $Q_j\ge a\ge b$. 
Thus it suffices to show that for some $C>0$, 
\begin{equation}\label{tripprod}
\prod_{j\in T_{01}(N)\cup T_{10}(N)\cup T_{00}(N)}\frac{Q_j}{b}\ge C.
\end{equation}  

For $j\in T_{10}(N)$ let $j'$ be the smallest element of $T_{01}(N)$ such that $j'>j$; 
this is defined for all $j$ except possibly the largest one, $j_*$, for which by Case 10, 
\(
Q_{j_*}\ge 1-r^{1/2}>r.
\)
By Case 10 for $c_{j+1}\ge 4$ and Case 01, we have 
$$
Q_jQ_{j'}\ge r^{-1}(r^{-1/2}-r^{1/2}-\tfrac{r^{1/2}}{a})(r^{-1}-1+r)\ge r^{-1}-1+r\ge 2\ge b^2, 
$$
except when $c_{j+1}=2$. In this case, if 
$j'>j+1$ then $c_{j+2}$ is even and by the second inequality in Case 00 we have 
$$
Q_jQ_{j+1}\ge Q_j(r^{-1/2}-r^{1/2}+\tfrac{1}{Q_j})= (r^{-1/2}-r^{1/2})Q_j+1\ge (r^{-1/2}-r^{1/2})(r^{-1/2}-r^{1/2}-\tfrac{r^{1/2}}{a})+1\ge 
b^2. 
$$
Finally, if $j'=j+1$ then by Case 10 for $c_{j+1}=2$ we have 
$$
Q_jQ_{j'}=Q_jQ_{j+1}\ge (r^{-1}-1+r)Q_j+r^{-1/2}\ge 1-r+r^2+r^{-1/2}\ge b^2.
$$
Thus we may delete from $T_{01}(N)\cup T_{10}(N)\cup T_{00}(N)$ all pairs $(j,j')$ with $j'=j+1$ and $(j,j+1)$ with $j'>j+1$. We will then be left with a subset $T\subset [1,N]$ such that $T\setminus (T_{00}\cup T_{01})$ is either empty or consists of $j_*$. But if $j\in T_{00}(N)$ then by Case 00, $Q_j\ge r^{-1/2}-r^{1/2}\ge b$, and if $j\in T_{01}$ then by Case 01, 
$Q_j\ge r^{-1}-1+r\ge 2\ge b$.

Thus \eqref{tripprod} holds with $C=r$.

To make the convergence uniform on compact subintervals, fix
$0<r_0<r_1<R_*$.  All strict inequalities used above are uniform for
$r\in[r_0,r_1]$: the functions $a(r)$,
$r^{-1/2}-r^{1/2}$, the three lower bounds for the paired products in
Cases $10$-$01$, and $(1-r^{1/2})/r$ are continuous and are strictly
larger than $1$ on this compact interval.  Hence one may choose a
single $b>1$ for which all the preceding individual and paired
estimates hold for every $r\in[r_0,r_1]$, and the possible unpaired
factor satisfies $Q_{j_*}/b\ge r\ge r_0$.  Thus
\(
 \prod_{j=1}^N\frac{Q_j}{b}\ge r_0
\)
uniformly in $r\in[r_0,r_1]$ and in the continued-fraction entries.
Moreover, Lemma \ref{endpoint} and $c_1\ge2$ give
$Q_0\ge r^{3/2}$, so
\(
 r^{-2}Q_0^{-2}\le r_0^{-5}.
\)
The bound \eqref{intermineq} is therefore dominated on the whole
compact interval by one convergent geometric series.  Compact sets
meeting $q=0$ are already covered by Theorem \ref{ma1}.  This proves
the stated uniform convergence on compact subsets of $(-R_*,1)$.

It remains to show that the sum $\psi_x(q)$ of series \eqref{seri} does not vanish 
for $q\in (-R_*,1)$. This, as before, follows from Lemma \ref{modinv}.

Theorem \ref{ma2} is proved. 
\end{proof} 

\begin{remark} Note that the endpoint $-R_*$ in Theorem \ref{ma2} is sharp, since $[1+\varphi]_q$ has a quadratic branch point singularity there, 
and the convergents $[1+\varphi]_{N,q}$ start oscillating 
for $q<-R_*$ (the equation $z=[3]_q-\frac{q^2}{z}$ for 
the supposed limit of these convergents has no real solutions in this range). 
\end{remark} 

Let us now study the properties of $[x]_q$ for $q=-r$, where $0<r<R_*$. It is easy to see that \eqref{xp1q} holds for all $x\ge 1$, and as before we use this recursion to extend $[x]_q$ to all real numbers $x$. 
 
For $n\in \Bbb Z$ let 
$$
I_n:=\inf_{n\le x<n+1}[x]_q,\ S_n:=\sup_{n\le x<n+1}[x]_q.
$$ 

\begin{proposition} 
These bounds are attained. Namely, if $n$ is odd, then
\[
 I_n=[n+2-\varphi]_q,\qquad S_n=[n-1+\varphi]_q,
\]
whereas if $n$ is even, then
\[
 I_n=[n-1+\varphi]_q,\qquad S_n=[n+2-\varphi]_q.
\]
\end{proposition} 

\begin{proof} 
By \eqref{xp1q} it suffices to prove the statement for $n=1$. We have
\(
I_{n+1}=1-rS_n,\ S_{n+1}=1-rI_n.
\)
Thus for odd $n\in \Bbb Z$
$$
I_n=\tfrac{1}{1+r}+r^{n-1}(I_1-\tfrac{1}{1+r}),\ S_n= \tfrac{1}{1+r}+r^{n-1}(S_1-\tfrac{1}{1+r})
$$
and for even $n\in \Bbb Z$ 
$$
I_n=\tfrac{1}{1+r}-r^{n-1}(S_1-\tfrac{1}{1+r}),\ S_n= \tfrac{1}{1+r}-r^{n-1}(I_1-\tfrac{1}{1+r}).
$$
Let 
$$
I:=\inf_{j\ge 1}(I_j)=\inf_{x\ge 1}[x]_q,\ S:=\sup_{j\ge 1}S_j=\sup_{x\ge 1}[x]_q.
$$
Since $1\le x\le 2$, we have
\(
[x]_q=1-r+\frac{r}{[(2-x)^{-1}]_q},
\)
hence
\(
I_1=1-r+\frac{r}{S},\ S_1=1-r+\frac{r}{I}.
\)
It follows that 
for odd $n\in \Bbb Z$
$$
I_n=\tfrac{1}{1+r}+r^{n}(\tfrac{1}{S}-\tfrac{r}{1+r}),\ S_n= \tfrac{1}{1+r}+r^{n}(\tfrac{1}{I}-\tfrac{r}{1+r})
$$
and for even $n\in \Bbb Z$ 
$$
I_n=\tfrac{1}{1+r}-r^{n}(\tfrac{1}{I}-\tfrac{r}{1+r}),\ S_n= \tfrac{1}{1+r}-r^{n}(\tfrac{1}{S}-\tfrac{r}{1+r}).
$$
So $I_n,S_n\to \frac{1}{1+r}$ as $n\to +\infty$. Thus $I\le \frac{1}{1+r}\le S$. 
Hence $\frac{1}{I}-\frac{r}{1+r}>0$. It follows that $I\le I_2=1-r+r^2-\frac{r^2}{I}\le 1-r+r^2-r^2(1+r)=1-r-r^3$. Note that $I_1\ge 1-r$, so $I<I_1$. 
Thus $I=I_2$. We first determine which root of the resulting quadratic occurs. Put
\[
 L:=ra,\qquad U:=1-r+a^{-1}=r^{-1}-a,
\]
and, for $c\ge2$, set
\[
 F_c(z):=[c]_{-r}-\frac{(-r)^{c-1}}{z},\qquad T(z):=1-rz.
\]
Then $F_{c+1}=T\circ F_c$. Put $b:=(1-r)/r$. Since $b>1$ and
\[
 b+b^{-1}=r^{-1}-1+\frac{r}{1-r}>a+a^{-1},
\]
we have $a<b$, and hence $L<1-r<U$. Moreover,
\[
 F_2(L)=U,\qquad F_2(U)=1-\frac{r^2}{a}\in(L,U),
 \qquad T(U)=L,\qquad T(L)=1-r^2a\in(L,U).
\]
Here $F_2(U)>L$ follows from
$F_2(U)-L=1-r(a+r/a)>r-r^2>0$, while
$T(L)<U$ follows from
$U-T(L)=a^{-1}-r+r^2a=(1-ra+r^2a^2)/a>0$; the other two
inequalities follow from the monotonicity of $F_2$ and $T$.
Thus both $F_2$ and $T$, and hence every $F_c$, preserve $[L,U]$.
Since $[2]_{-r}=1-r\in[L,U]$, the same is true of every
$[c]_{-r}=T^{c-2}(1-r)$. Backward induction through a finite continued
fraction therefore shows that every finite $q$-convergent belongs to
$[L,U]$. Passing to the limit gives $[x]_{-r}\in[L,U]$ for every
$x\ge1$, so $I\ge L$.

The equation
\[
 I=1-r+r^2-\frac{r^2}{I}
\]
has roots $ra=L$ and $r/a<L$. Hence
\[
 I=I_2=ra.
\]
Moreover, $1+\varphi=[[3,3,\ldots]]$, and its $q$-value is a fixed point
of $F_3$. The other fixed point is $r/a<L$, so
\[
 [1+\varphi]_{-r}=ra=I_2.
\]
It follows that
\[
 S_1=\frac{1-I_2}{r}=r^{-1}-a=1-r+a^{-1}=[\varphi]_{-r}.
\]
As above, $S_2\le1-r+r^2<S_1$, while
$S_n\to(1+r)^{-1}$ and $S_1>(1+r)^{-1}$; hence $S=S_1$. Finally,
\[
 I_1=1-r+\frac rS
 =1-r+\frac r{1-r+a^{-1}}
 =1-\frac{r^2}{a}=[3-\varphi]_{-r}.
\]
\end{proof} 

\begin{corollary}\label{bounded} There exist $K,L>0$ such that 
$$
|[x]_{N,q}|\le K+Lr^x,\ |[x]_q|\le K+Lr^x.
$$
\end{corollary} 

\begin{remark}
The preceding estimates do not settle convergence of \eqref{seri} at
$q=-R_*$ for arbitrary $x$.  There is, however, a simple endpoint
example.  If $x=1+\varphi=[[3,3,\ldots]]$ and
$y_N=[x]_{N,-R_*}$, then
\(
 y_N=2R_*-\frac{R_*^2}{y_{N-1}},\qquad y_1=2R_*,
\)
so
\(
 y_N=\frac{N+1}{N}R_*.
\)
Consequently the partial sums of \eqref{seri} are
\[
 \frac1{y_N}=\frac{N}{(N+1)R_*}
 =\sum_{j=1}^N\frac{1}{j(j+1)R_*}.
\]
Thus the endpoint series converges in this example, with a tail of
order $1/N$, and its reciprocal limit is
$[1+\varphi]_{-R_*}=R_*$.
\end{remark}

\subsection{Analyticity of $[x]_q$ on the interval $(-R_*,1)$}

We give a contraction argument showing that series \eqref{seri} converges absolutely and uniformly on compact sets in an $x$-independent neighborhood of $(-R_*,1)$ in the complex plane, implying analyticity of $[x]_q$ on this interval. Since the only point of $\partial D$ with absolute value $3-2\sqrt{2}$ 
is $-3+2\sqrt{2}$, this implies that the uniform radius $R_\bullet=\inf_{x\in \Bbb R}R(x)$ of convergence of the Laurent series $[x]_q$ is strictly bigger than $3-2\sqrt{2}$. Actually, with a bit more work our bounds can be quantified to produce an explicit neighborhood with this property and explicit lower bound for $R_\bullet$, but we won't do so here, since this would make the formulas more cumbersome while the result clearly would not be optimal. 

\begin{theorem}\label{analy} There exists a neighborhood $U$ of the interval $(-R_*,1)$ such that for any $x>1$, the series $\frac{1}{[x]_q}=\sum_{N\ge 0}\frac{q^{C_N}}{a_N(q)a_{N+1}(q)}$ is absolutely convergent on $U$ uniformly on compact sets. Thus it converges to a holomorphic function 
on $U$. So the function $q\mapsto [x]_q$ on $(-R_*,1)$ is real analytic. 
\end{theorem} 

\begin{proof}
It suffices to construct a neighborhood of each point $q_0=-r$,
$0<r<R_*$.  We use the two maps
\[
 T_q(z):=1+qz,\qquad F_q(z):=1+q-\frac qz.
\]
The map defining a continued-fraction digit is
\begin{equation}\label{analy:digitmaps}
 L_{c,q}(z):=[c]_q-\frac{q^{c-1}}z
 =T_q^{c-2}F_q(z),\qquad c\ge2.
\end{equation}
Thus only two maps need to be controlled, even when the digits are
unbounded.

Choose $a>1$ satisfying
\[
 a+a^{-1}=r^{-1}-1+r,
\]
and set $A:=ra$, $B:=r^{-1}-a$, and $I:=[A,B]$.
Since $r<R_*$, the right hand side is greater than $2$, so $a$
exists.  Moreover, $a<(1-r)/r$, as follows by comparing
$t+t^{-1}$ at these two numbers, both greater than $1$.  Hence
\[
 0<A<1-r<1<B=1-r+a^{-1}.
\]
At $q=q_0$ both maps are decreasing on $I$, and
\[
 T_{q_0}(B)=A,\quad T_{q_0}(A)<1<B,\qquad
 F_{q_0}(A)=B,\quad F_{q_0}(B)>1-r>A.
\]
Consequently each map preserves $I$.  Direct differentiation gives
\begin{equation}\label{analy:pairbounds}
\begin{split}
 \sup_I |(T_{q_0}T_{q_0})'|&=r^2,\\
 \sup_I |(T_{q_0}F_{q_0})'|,
 \ \sup_I |(F_{q_0}T_{q_0})'|&\le a^{-2},\\
 \sup_I |(F_{q_0}F_{q_0})'|
 &\le\bigl(1+(1-r)a\bigr)^{-2}.
\end{split}
\end{equation}
All these bounds are strictly less than $1$.

We spell out why these contractions persist for complex $q$.
Choose $\lambda<1$ larger than the three bounds in
\eqref{analy:pairbounds}.  For sufficiently small $\eta>0$ and
$|q-q_0|$, the convex neighborhood
\[
 V:=\{z\in\mathbb C:\operatorname{dist}(z,I)<\eta\}
\]
has closure disjoint from $0$, both single maps and all four pair
maps are holomorphic on a neighborhood of $\overline V$, and the
absolute value of the derivative of every pair map on $V$ is at
most $\lambda$.  Fix also a uniform bound $M\ge1$ for the absolute
values of the derivatives of the two single maps there.

For a sufficiently small $\delta>0$ let
\[
 J:=\{z\in\mathbb C:\operatorname{dist}(z,I)<\delta\}.
\]
After decreasing the neighborhood of $q_0$, we have
\[
 T_q(J)\cup F_q(J)\subset V,\qquad
 H_q(J)\subset J
 \quad(H_q=T_qT_q,T_qF_q,F_qT_q,F_qF_q).
\]
Indeed, first take $\delta<\eta/(2M)$, so the single-map inclusions
follow from their derivative bound and continuity in $q$.
For the pair maps, if $z\in J$ and $t\in I$ minimizes $|z-t|$,
then
\[
 \operatorname{dist}(H_q(z),I)
 \le \lambda|z-t|+
       \sup_{u\in I}|H_q(u)-H_{q_0}(u)|.
\]
The last supremum can simultaneously be made smaller than
$(1-\lambda)\delta$ for all four maps.
It follows that
\[
 K_q:=J\cup T_q(J)\cup F_q(J)\subset V
\]
is preserved by both $T_q$ and $F_q$.  Since $V$ is convex, the
derivative bounds imply the following estimate for every word
$W_q$ of length $\ell$ in the two maps:
\begin{equation}\label{analy:wordbound}
 |W_q(z)-W_q(w)|\le
 M\lambda^{\lfloor\ell/2\rfloor}|z-w|,
 \qquad z,w\in K_q.
\end{equation}
To prove it, group the word into adjacent pairs and, if necessary,
one single map.  All intermediate values lie in $K_q$, and the
segment joining any two of them lies in $V$.

Now $1\in I\subset K_q$, and
$[c]_q=T_q^{c-1}(1)$.  Every finite continued fraction
$z_N(q):=[x]_{N,q}$ is therefore the value at $1$ of a word of
length $C_N$ in these two maps.  In particular, $z_N(q)\in K_q$,
and the finite fractions are holomorphic and nonzero throughout
the chosen neighborhood of $q_0$.  The usual continuant formulas
then give $z_N=a_N/b_N$ with $a_N,b_N\ne0$ there; simultaneous
zeros are also excluded by the determinant identity and $q\ne0$.

For $N\ge1$, the two consecutive fractions have the common prefix
\[
 G_{N,q}:=L_{c_1,q}\cdots L_{c_{N-1},q}T_q^{c_N-2},
 \qquad \operatorname{length}(G_{N,q})=C_N-1,
\]
and
\[
 z_N=G_{N,q}(T_q(1)),\qquad
 z_{N+1}=G_{N,q}(F_q([c_{N+1}]_q)).
\]
Both arguments on the right lie in $K_q$.  Set
$m:=\inf_{z\in V}|z|>0$ and $D_V:=\operatorname{diam}(V)$.
By \eqref{analy:wordbound} and the determinant identity,
\begin{equation}\label{analy:seriesbound}
 \left|\frac{q^{C_N}}{a_N(q)a_{N+1}(q)}\right|
 =\left|\frac1{z_{N+1}}-\frac1{z_N}\right|
 \le\frac{MD_V}{m^2}
       \lambda^{\lfloor(C_N-1)/2\rfloor}.
\end{equation}
The term with $N=0$ is bounded by $m^{-1}$.  Since $C_N\ge N$,
this proves absolute and uniform convergence of \eqref{seri} on
a neighborhood of $q_0$, uniformly over all digit sequences.
The same estimates show that $z_N$ converges uniformly; its limit
has absolute value at least $m$, and \eqref{seri0} identifies the
sum with its reciprocal.

Finally take the union of $D$ and the neighborhoods just obtained
for every $q_0\in(-R_*,0)$.  This is an open neighborhood $U$ of
$(-R_*,1)$, independent of $x$.  Every compact subset of $U$ has a
finite cover by regions with the preceding convergence bounds or
those of Theorem \ref{mai}.  Thus \eqref{seri} converges absolutely
and uniformly on compact subsets of $U$, and its reciprocal is
holomorphic and nonzero there.  On the real interval this is the
function already constructed in Theorem \ref{ma2}, proving the
assertion.
\end{proof}

\begin{remark} Theorem \ref{analy} gives additional evidence for the conjecture 
of \cite{LMOV} that the power series $[x]_q$ always converges for $|q|<R_*$, and together with Theorems \ref{ma1},\ref{ma1a} raises the question: which simply-connected domains $0\in \mathcal D\subset \Bbb C$ admit analytic continuation of $[x]_q$ 
for all $x\in \Bbb R$? Is there the maximal one? 

Corollary 6.5 of \cite{EGMS} suggests the following candidate.  Choose
a simply connected domain $\Omega$ containing $D$ and the disk
$|q|<R_*$ but avoiding the two branch points $e^{\pm\pi i/3}$, and
on $\Omega$ choose the branch of $\sqrt{1-q+q^2}$ which has value
$1$ at $q=0$.  Let $\mathcal D$ be the connected component containing
$0$ of
\[
 \left\{q\in\Omega: |q+1+\sqrt{1-q+q^2}|>3\sqrt{|q|}\right\}.
\]
This is the intended open drop-like region; by the cited corollary it
is zero-free for all the $q$-continuants $a_N(q)$.  It contains $D$
and the disk $|q|<R_*$, and
$\mathcal D\cap\mathbb R=(-R_*,1)$.
\end{remark} 

\section{Behavior of $[x]_q$ as a function of $x$ for $0<q<1$}

\subsection{Monotonicity of $x\mapsto [x]_q$}
In this subsection we discuss monotonicity of the function $x\mapsto [x]_q$ 
when $0<q<1$ is fixed. These results are not really original, 
 as they follow easily from the known properties 
 of $q$-rational and $q$-real numbers, see \cite{MO2,BBL}. 

\begin{proposition}\label{inc} The function $x\mapsto [x]_q$ is strictly increasing 
for $x\in \Bbb R$, and 
$$
\lim_{x\to -\infty}[x]_q=-\infty.
$$ 
\end{proposition} 

\begin{proof} By \eqref{xp1q} it suffices to prove the Proposition for $x>1$.
For $0<q<1$, \eqref{seri0} gives
\[
 \frac1{[x]_{N+1,q}}-\frac1{[x]_{N,q}}
 =\frac{q^{C_N}}{a_N(q)a_{N+1}(q)}>0,
\]
because the continuants have positive coefficients and constant term
$1$.  Hence the positive sequence $[x]_{N,q}=[x_N]_q$ is decreasing;
by Theorem \ref{mai} it tends to $[x]_q$.

Now, if $x<y$ then for large $N$ we have $x_N<y_N$, 
so by \cite{MO1},  
\(
[x]_{N,q}=[x_N]_q<[y_N]_q=[y]_{N,q}.
\)
Now take rational numbers $z_1,z_2$ such that
\(
x<z_1<z_2<y.
\)
Since $x_N\to x$ and $y_N\to y$, for all sufficiently large $N$ we
have
\(
x_N<z_1<z_2<y_N.
\)
All four numbers in this inequality are rational. Hence, by the
strict monotonicity of $q$-rationals proved in \cite{MO1},
\(
[x_N]_q<[z_1]_q<[z_2]_q<[y_N]_q.
\)
Letting $N\to\infty$ and using
\[
[x_N]_q=[x]_{N,q}\longrightarrow[x]_q,
\qquad
[y_N]_q=[y]_{N,q}\longrightarrow[y]_q,
\]
we obtain
\(
[x]_q\le [z_1]_q<[z_2]_q\le[y]_q.
\)
In particular, $[x]_q<[y]_q$, as desired.

The last statement follows from monotonicity and Proposition \ref{ratcon}.
\end{proof} 

\begin{proposition}\label{lim} For all $x\in \Bbb R$, one has $\lim_{q\to 1^-}[x]_q=x$.
\end{proposition} 

\begin{proof} If $x\in \Bbb Q$, the result is obvious. 
So let $x$ be irrational and $a,b\in \Bbb Q$ be 
such that $a<x<b$. Then by Proposition \ref{inc}, for any $0<q<1$ 
we have $[a]_q<[x]_q<[b]_q$. Thus $\liminf_{q\to 1}[x]_q\ge a$ and 
$\limsup_{q\to 1}[x]_q\le b$. Since this holds for any $a<x<b$, 
the proposition follows.  
\end{proof} 

\subsection{Monotonicity of $[x]_q$ with respect to $q$}

\begin{proposition} For any $x>1$:

(i) we have 
$$
\frac{d}{dq}[x]_q\ge 0,
$$ 
and if $0<q_1<q_2<1$ then 
$$
[x]_{q_1}<[x]_{q_2};
$$

(ii) $[x]_q<x$ for $q<1$. 
\end{proposition} 

\begin{proof} (i) It is enough to show that $[x]_q'\ge 0$ (as  
$[x]_q$ is a nonconstant analytic function, so cannot be constant on an interval).
We first prove the statement for $x\in \Bbb Q$. 
If $x\in \Bbb Z$, the statement is easy, so we assume $x\notin \Bbb Z$. 
Let 
\(
c_1(x)=n+1\ge 2.
\)
We have 
\(
[x]_q=\frac{1-q^{n+1}}{1-q}-\frac{q^{n}}{[y]_q},
\)
where $y:=(n+1-x)^{-1}$. Differentiating this equality, we get 
$$
[x]_q'=\frac{1-(n+1)q^{n}+nq^{n+1}}{(1-q)^2}-\frac{nq^{n-1}}{[y]_q}+\frac{q^{n}}{[y]_q^2}[y]_q'.
$$
Assume that $[y]_q'\ge 0$. Then, since $[y]_q\ge 1$, we have
$$
(1-q)^2[x]_q'\ge 1-(n+1)q^{n}+nq^{n+1}-n(1-q)^2q^{n-1}=1-q^n(n(q^{-1}-1)+1).
$$
It is easy to check by differentiation that the function $z\mapsto q^z(z(q^{-1}-1)+1)$ is decreasing 
on $[1,\infty)$. Hence its value for every integer $n\geq1$ is bounded above by its value at $n=1$. Thus if $n=1$ then
\(
(1-q)^2[x]_q'\ge 1-q((q^{-1}-1)+1)=0.
\)
So it follows by induction in the length of the negative continued fraction 
representation of $x$ that $[x]_q'\ge 0$. 

Now consider the general case. By Proposition \ref{ratcon}, the rational case implies that 
if $q_1<q_2$ then $[x]_{q_1}\le [x]_{q_2}$. This implies that $[x]_q'\ge 0$. 

(ii) follows from (i) and Proposition \ref{lim}.
\end{proof} 

\subsection{The Lebesgue derivative of $[x]_q$}
By Lebesgue's theorem, the function $x\to [x]_q$ must be differentiable almost everywhere with respect to the Lebesgue measure, and its Stieltjes differential is a positive measure
$\mu_q$ on $\Bbb R$ such that 
for $x<y$
\(
[y]_q-[x]_q=\mu_q((x,y]). 
\)
The measure $\mu_q$ has a representation 
$$
\mu_q=\mu_q^*+\sum_{x\in \Bbb Q}{\rm Jump}_x(q)\delta_x,
$$
where $\mu_q^*$ is an atomless, regular measure on $\Bbb R$.
We have $\mu_q^*=d[x]_q^*$, where 
\(
[x]_q^*=[x]_q+\sum_{y\in \Bbb Q,y>x}{\rm Jump}_y(q)
\)
is a continuous, non-decreasing function. 

\begin{proposition} $\frac{d}{dx}[x]_q=\frac{d}{dx}[x]_q^*=0$ almost everywhere. 
\end{proposition} 

\begin{proof} 
For a typical real number $x$ (i.e., outside a set of Lebesgue measure $0$), one has (\cite{It}, Theorem 3.1)
\(
\frac{1}{N}\log|x-x_N|\to 0, N\to \infty,
\)
so $|x-x_N|$ decays 
subexponentially. This together with Theorem \ref{mai} implies that 
\(
\frac{[x]_q-[x_N]_q}{x-x_N}\to 0, N\to \infty. 
\)
Thus if the derivative of $[x]_q$ exists at the point $x$, it must equal zero, as desired. 

Since $[x]_q-[x]_q^*$ is a countable linear combination of shifted Heaviside functions
with positive coefficients, it has zero derivative almost everywhere, so the statement follows for $[x]_q^*$ as well.  
\end{proof}

\section{The total jump of $[x]_q$ on $(1,\infty)$}

This section is dedicated to showing, under various assumptions, 
that the sum of the jumps of $[x]_q$ at rational points 
in $(1,\infty)$,
\begin{equation}\label{jumpser}
{\rm Jump}(q):=\sum_{x\in \Bbb Q\cap (1,\infty)} {\rm Jump}_x(q)
\end{equation}
 equals the difference between its limits at $+\infty$ and at $1$ along this interval, i.e., $\frac{1}{1-q}-1=\frac{q}{1-q}$. This implies a similar statement on any open interval of the real line
and means that the measure $\mu_q$ is purely atomic, i.e., $\mu_q^*=0$.  

\subsection{The formal total jump}
We first study ${\rm Jump}(q)$ as a formal series in $q$.  

\begin{proposition}\label{jumpform} (i) For $x=[[c_1,...,c_N]]>1$,  
${\rm Jump}_x(q)=q^{C_N}(1+O(q))\in \Bbb Z[[q]]$.  
Thus series \eqref{jumpser} is $q$-adically convergent. 

(ii) We have  
$$
{\rm Jump}(q)=[\infty]_q-[1]_q=\frac{q}{1-q}=q+q^2+...
$$
\end{proposition} 

\begin{proof} (i) follows easily from the definition, so we may focus on (ii).
It suffices to prove (ii) modulo $q^{n+1}$ for every $n$. 
But (i) implies that modulo $q^{n+1}$, the jump ${\rm Jump}_x(q)$ 
is only nonzero if $x=[[c_1,...,c_N]]$ with $C_N\le n$, 
in particular $N\le n$. 
There are finitely many such numbers $y_1(n)<...<y_m(n)$, 
and $[x]_q$ is constant modulo $q^{n+1}$ for $y_i(n)\le x<y_{i+1}(n)$, $0\le i\le m$, where 
$y_0(n):=1$ and $y_{m+1}(n):=\infty$. Thus modulo $q^{n+1}$ we have 
$[y_m(n)]_q=[\infty]_q=\frac{1}{1-q}$ and 
\(
{\rm Jump}_{y_i(n)}(q)=[y_i(n)]_q-[y_{i-1}(n)]_q,
\)
so 
$$
{\rm Jump}(q)=\sum_{i=1}^m{\rm Jump}_{y_i(n)}(q)=[y_m(n)]_q-[y_0(n)]_q=\frac{q}{1-q}.
$$
\end{proof} 

\subsection{The numerical total jump} 

Now let us study the numerical series
\(
{\rm Jump}(q)=[1]_q^*-[1]_q=[1]_q^*-1.
\)
For each fixed $0\le q<1$ its terms are nonnegative. More precisely,
if $F=\{x_1<\cdots<x_m\}$ is a finite subset of
$\Bbb Q\cap(1,\infty)$, then monotonicity of $x\mapsto[x]_q$ gives
$[x_1]_q^-\ge [1]_q$ and $[x_i]_q^-\ge [x_{i-1}]_q$ for $i\ge2$.
Consequently,
\begin{align*}
\sum_{x\in F}{\rm Jump}_x(q)
&=\sum_{i=1}^m\bigl([x_i]_q-[x_i]_q^-\bigr)\\
&\le [x_m]_q-[1]_q\\
&\le \lim_{x\to+\infty}[x]_q-[1]_q
 =\frac{1}{1-q}-1=\frac{q}{1-q}.
\end{align*}
Taking the supremum over all finite $F$ shows that \eqref{jumpser}
converges for every fixed $q\in[0,1)$ and that
\(
0\le {\rm Jump}(q)\le \frac{q}{1-q}.
\)

We would now like to study convergence of series \eqref{jumpser} for complex $q$. 
Let $D'\subset D$ be the region defined by the inequality 
$a^2<1-r:$
\begin{figure}[htbp]
\vspace{-10pt}
    \centering
  \includegraphics[width=1.1\columnwidth]{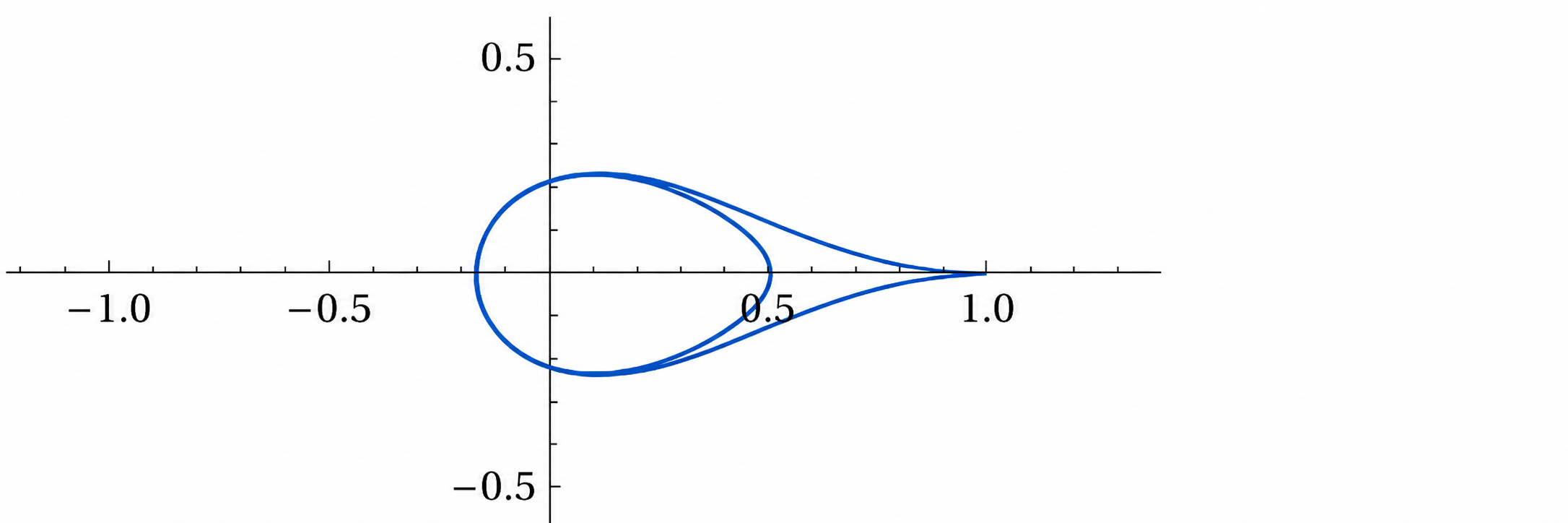}
  \vspace{-10pt}
\end{figure}

The picture is deceptive here -- the boundaries of $D'$ and $D$ don't touch, they are 
just very close on the negative real axis. Namely, $\partial D'$ intersects 
the real axis at $1/2$ and $-\rho_0$, where 
\(
\rho_0\approx 0.1705<3-2\sqrt{2}\approx 0.1716. 
\)
(recall that $\partial D$ intersects the negative real axis at $-3+2\sqrt{2}$). 
The largest disk centered at the origin contained in $D'$ has radius $\rho_0$. 

\begin{corollary}\label{unifo} The series \eqref{jumpser} converges 
absolutely and uniformly on compact sets in $D'$ to the function $\frac{q}{1-q}$. 
\end{corollary} 

\begin{proof}  We will need the following lemma. 
\begin{lemma}\label{bouu} For $q\in D$ and $x=[[c_1,c_2,...,c_N]]$ 
\[
|{\rm Jump}_x(q)|\le r^{\sum_{i=1}^N (c_i-2)}a^{2N}.
\]
\end{lemma} 

\begin{proof} Put
\[
A_{N-1}:=a_{N-1}(c_2,\ldots,c_N\mid q),
\qquad
A_N^\infty:=a_N(c_2,\ldots,c_N,\infty\mid q),
\]
and recall that $\alpha=r^{1/2}a^{-1}$.  By Lemma \ref{keybound},
\(
|A_{N-1}|\ge \alpha^{N-1}.
\)
Moreover, for every integer $n\ge2$, applying the same lemma to the
sequence
\(
(c_2,\ldots,c_N,n)
\)
gives
\(
|a_N(c_2,\ldots,c_N,n\mid q)|\ge\alpha^N.
\)
Letting $n\to\infty$ therefore yields
\(
|A_N^\infty|\ge\alpha^N.
\)

We also have $\alpha\le1$.  Indeed, by \eqref{eqfora},
\(
a+\frac1a
\le r^{-1/2}+r^{1/2}.
\)
The function $u\mapsto u+u^{-1}$ is decreasing on $(0,1]$, so this
inequality implies $a\ge r^{1/2}$.  Hence
\(
 \alpha=r^{1/2}a^{-1}\le1.
\)

Proposition \ref{jumpbo} now gives
\begin{align*}
|{\rm Jump}_x(q)|
&=
\frac{r^{C_N}}{|A_{N-1}A_N^\infty|}\\
&\le
r^{C_N}\alpha^{-2N+1}\\
&\le
r^{C_N}\alpha^{-2N}\\
&=
r^{\sum_{i=1}^N(c_i-2)}a^{2N}.
\end{align*}
where in the last equality we used
\(
C_N=N+\sum_{i=1}^N(c_i-2).
\)
This proves the asserted estimate.  The argument also covers $N=1$
with the convention $a_0=1$.
\end{proof} 

But if $q\in D'$ then 
\[
\sum_{N\ge 1}\sum_{c_1,...,c_N\ge 2}
r^{\sum_{i=1}^N (c_i-2)}a^{2N}
=
\sum_{N\ge 1}\left(\frac{a^2}{1-r}\right)^N
=\frac{a^2}{1-r-a^2}.
\]
So the Corollary follows from Lemma \ref{bouu}, Proposition \ref{jumpform} and the Weierstrass test. 
\end{proof} 

Corollary \ref{unifo} implies

\begin{corollary}\label{atomic} If $0<q<1/2$ then the measure $\mu_q$ is purely atomic, i.e., 
$\mu_q^*=0$. In other words, for every $x\in \Bbb R_{>0}$, 
$$
[x]_q=\sum_{y\in \Bbb Q\cap (0,x]}{\rm Jump}_y(q),
$$
and $[x]_q^*=\frac{1}{1-q}$ for all $x\in \Bbb R$. 
\end{corollary} 

\begin{remark}\label{range}  We expect that Corollary \ref{unifo} holds on a larger region containing the interval $(0,1)$. In particular, we expect that  
series \eqref{jumpser} converges to $\frac{q}{1-q}$ and hence Corollary 
\ref{atomic} holds for any $q\in (0,1)$. In the following subsection 
we provide a method of extending the range of $q\in (0,1)$ 
in which we can prove this. 
\end{remark} 

\subsection{Improving jump bounds} 

For $\ell\ge 1$ let ${\rm Jump}(q)_\ell$ be the sum of ${\rm Jump}_x(q)$ 
over all rational $x>1$ whose negative continued fraction has 
length $\ell(x)=\ell$; clearly, ${\rm Jump}(q)_\ell<\infty$ since 
$\sum_{\ell\ge 1}{\rm Jump}(q)_\ell={\rm Jump}(q)\le \frac{q}{1-q}$. 
For example, 
\(
{\rm Jump}(q)_1=\sum_{m=2}^\infty q^{m-1}(1-q)=q.
\)
Let ${\rm Jump}(q,z)$ be the generating function of ${\rm Jump}(q)_\ell$:
$$
{\rm Jump}(q,z)=\sum_{x\in \Bbb Q\cap (1,\infty)}{\rm Jump}_x(q)z^{\ell(x)}=  \sum_{\ell\ge 1}{\rm Jump}(q)_\ell z^\ell.
$$

Let $q_*$ be the largest number in $(0,1]$ such that for every
$0\le q<q_*$ there exists $z>1$ for which
${\rm Jump}(q,z)<\infty$. Equivalently, the decomposition by continued
fraction length converges exponentially fast, in the sense that
\[
\sum_{\ell\ge1}{\rm Jump}(q)_\ell z^\ell<\infty
\qquad\text{for some }z>1.
\]
\begin{proposition}\label{extrange}
For any $0<r<q_*$, there is $\delta>0$ 
such that \eqref{jumpser} converges absolutely 
and uniformly on compact sets in the disk $|q-r|<\delta$; 
in other words, this holds in some neighborhood of the interval 
$(0,q_*)$. Thus \eqref{jumpser} converges in $(0,q_*)$ 
to the function $\frac{q}{1-q}$ and Corollary 
\ref{atomic} holds for $q\in (0,q_*)$. 
\end{proposition} 

\begin{proof}
Let $R_N:=R_N(r)$ and $R_N^*:=R_N(q)$. 
By \eqref{recrn}, 
$$
R_N-R_N^*=\frac{1-r^{c_{N+1}}}{1-r}-\frac{1-q^{c_{N+1}}}{1-q}-\frac{r^{c_{N}-1}}{R_{N-1}}+\frac{q^{c_{N}-1}}{R_{N-1}^*}.
$$
Let $\Delta_N:=\frac{R_N-R_N^*}{R_N}$, then 
$$
\Delta_N=\frac{1}{R_N}\left(\frac{1-r^{c_{N+1}}}{1-r}-\frac{1-q^{c_{N+1}}}{1-q}\right)-\frac{r^{c_{N}-1}}{R_{N}R_{N-1}}+\frac{q^{c_{N}-1}}{R_NR_{N-1}^*}.
$$
We would like to find $\varepsilon>0$ such that for sufficiently small $|q-r|$, the inequality $|\Delta_{N-1}|<\varepsilon$ implies that $|\Delta_N|<\varepsilon$. 
Then we will get 
\(
|R_N^*|\ge R_N(1-\varepsilon). 
\)
Recall that
\(
 R_N\ge \alpha=r^{1/2}a^{-1}.
\)
But for a positive real parameter $r$, equation \eqref{eqfora} gives
$a=r^{1/2}$ and hence $\alpha=1$. Thus $R_N\ge 1$. Pick
$K<1-r$. Then for small enough
$\varepsilon$, if $|q-r|$ is sufficiently small compared to
$\varepsilon$ then
$$
\left|\frac{1-r^n}{1-r}-\frac{1-q^n}{1-q}\right|<\frac{K\varepsilon}{2},\ |q^{n-1}-r^{n-1}|<\frac{K\varepsilon}{2}
$$ 
for all $n\ge 2$.
Then we get 
\[
|\Delta_N|\le \frac{K\varepsilon}{2R_N}+\frac{K\varepsilon}{2R_{N}R_{N-1}(1-\varepsilon)}+\frac{r^{c_{N}-1}|\Delta_{N-1}|}{R_NR_{N-1}(1-\varepsilon)}\le \varepsilon\frac{K+r}{1-\varepsilon}.
\]
Since $K+r<1$, for sufficiently small $\varepsilon$ the right hand side is $<\varepsilon$, as desired.

Now fix $z>1$ such that ${\rm Jump}(r,z)<\infty$, which is possible since
$r<q_*$. For a string $\bold d:=(d_1,d_2,\ldots)$, set
\[
R_k(q,\bold d):=
\frac{a_{k+1}(d_1,\ldots,d_{k+1}\mid q)}
     {a_k(d_1,\ldots,d_k\mid q)}.
\]
Then
\[
R_0(q,\bold d)=[d_1]_q,\qquad
R_k(q,\bold d)=[d_{k+1}]_q-
\frac{q^{d_k-1}}{R_{k-1}(q,\bold d)}
\quad(k\ge1).
\]
The preceding ratio argument therefore gives, uniformly in all the
entries,
\(
|R_k(q,\bold d)|\ge(1-\varepsilon)R_k(r,\bold d)
\)
whenever $|q-r|$ is sufficiently small. More precisely, since 
\[
a_m(d_1,\ldots,d_m\mid q)
=
\prod_{k=0}^{m-1}R_k(q,\bold d),
\]
for every
sufficiently small $\varepsilon>0$ there exists $\delta_0>0$ such that,
if $|q-r|<\delta_0$, then
\begin{equation}\label{contes1}
|a_m(d_1,\ldots,d_m\mid q)|
\ge
(1-\varepsilon)^m
a_m(d_1,\ldots,d_m\mid r)
\end{equation}
for every $m\ge1$ and $(d_1,...,d_m)$. 
Applying the same estimate to $(d_1,\ldots,d_{m-1},n)$ and then letting
$n\to\infty$ also gives
\begin{equation}\label{contes2}
|a_m(d_1,\ldots,d_{m-1},\infty\mid q)|
\ge
(1-\varepsilon)^m
a_m(d_1,\ldots,d_{m-1},\infty\mid r).
\end{equation}

Choose $\sigma>1$ with $r\sigma<1$, and then choose an integer $L$ so
large that
\(
\sigma^{-L}\frac{r\sigma}{1-r\sigma}<1.
\)
Put
\(
\lambda:=(1-\varepsilon)^{-2},
\)
so that
\(
(1-\varepsilon)^{-2N+1}\le\lambda^N.
\)
By taking $\varepsilon>0$ sufficiently small and choosing $\rho>r$
sufficiently close to $r$, we may arrange that
\[
\rho\sigma<1,\qquad
\lambda\left(\frac{\rho}{r}\right)^L<z,
\qquad
\lambda\sigma^{-L}
\frac{\rho\sigma}{1-\rho\sigma}<1.
\]
After decreasing $\delta_0$ if necessary, fix $\delta>0$ so that
$|q-r|<\delta$ implies $|q|\le\rho$ and the two continuant estimates
\eqref{contes1},\eqref{contes2} hold.

Let
\[
x=[[c_1,\ldots,c_N]],
\qquad
C_N=\sum_{i=1}^N(c_i-1).
\]
The formula of Proposition \ref{jumpbo}, whose two
denominator continuants have lengths $N-1$ and $N$, gives
\[
|{\rm Jump}_x(q)|
\le
{\rm Jump}_x(r)
\left(\frac{|q|}{r}\right)^{C_N}
(1-\varepsilon)^{-2N+1}.
\]
If $C_N\le LN$, it follows that
\[
|{\rm Jump}_x(q)|
\le
{\rm Jump}_x(r)
\left[
\lambda\left(\frac{\rho}{r}\right)^L
\right]^N
\le
{\rm Jump}_x(r)z^N.
\]
Consequently, the sum of all terms with $C_N\le LN$ is bounded,
and uniformly so for $|q-r|<\delta$ as
${\rm Jump}(r,z)<\infty$. 

For the remaining strings, for which $C_N>LN$, note that the relevant
finite continuants are at least $1$ at $r$, and the same is true for
the terminal-$\infty$ continuant by passage to the limit from finite
terminal entries. Thus the continuant comparison gives
\[
|{\rm Jump}_x(q)|
\le
\lambda^N\rho^{C_N}.
\]
Writing $k_i=c_i-1\ge1$, we obtain
\begin{align*}
\sum_{\substack{c_1,\ldots,c_N\ge2\\ C_N>LN}}
\rho^{C_N}
&\le
\sigma^{-LN}
\sum_{k_1,\ldots,k_N\ge1}
(\rho\sigma)^{k_1+\cdots+k_N} \\
&=
\left(
\sigma^{-L}
\frac{\rho\sigma}{1-\rho\sigma}
\right)^N.
\end{align*}
Hence the sum of all terms with $C_N>LN$ is bounded by the convergent
geometric series
\[
\sum_{N\ge1}
\left(
\lambda\sigma^{-L}
\frac{\rho\sigma}{1-\rho\sigma}
\right)^N.
\]

The Weierstrass test therefore proves absolute and uniform convergence
of \eqref{jumpser} in the disk $|q-r|<\delta$. Since $r\in(0,q_*)$ was
arbitrary, the union of these disks is a connected open neighborhood
of $(0,q_*)$. On the nonempty subinterval where Corollary
\ref{unifo} applies, the sum is $\frac{q}{1-q}$. The identity theorem
therefore shows that the same equality holds throughout $(0,q_*)$.
This proves the proposition.
\end{proof} 

It remains to determine $q_*$. We expect that $q_*=1$, so that the statements of Remark \ref{range} hold for all $0<q<1$. However, so far we only know that $q_*\ge 1/2$. Let us give a better lower bound for $q_*$. 

For $s\ge 1$, let ${\rm Jump}_*(q,z)_s$ be the sum ${\rm Jump}_x(q)z^{\ell(x)}$ over all $x=[[c_1,...,c_N]]$ such that $c_1\ge 3$ and the number of $i$ with $c_i\ge 3$ equals $s$. Then 
\begin{equation}\label{starformula}
{\rm Jump}(q,z)=q^{-1}\sum_{s=1}^\infty {\rm Jump}_*(q,z)_s.
\end{equation}
By Proposition \ref{jumpbo}, 
${\rm Jump}_x(q)\le q^{C_N}(1-q)$, which implies that 
${\rm Jump}_*(q,z)_s<\infty$ for all $s$ as long as $qz<1$. 

Next, note that if in Proposition \ref{jumpbo} one has $1\le i\le j\le N-1$ and $c_{j+1}\ge 3$ then 
\(
x(i)>[[c_{i+1},...,c_j,c_{j+1}-1]],
\)
so
\begin{equation}\label{xiine}
[x(i)]_q>[x(i)]_q^-\ge [[c_{i+1},...,c_j,c_{j+1}-1]]_q.
\end{equation} 

\begin{corollary}\label{betterform} If $x=[[c_1,...,c_N]]$ and $1\le j\le N-1$ with $c_{j+1}\ge 3$ then 
$$
{\rm Jump}_x(q)\le {\rm Jump}_{x(j)}(q)\prod_{i=1}^{j} \frac{q^{c_i-1}}{[[c_{i+1},...,c_j,c_{j+1}-1]]_q^2}={\rm Jump}_{x(j)}(q)\frac{q^{C_j}}{a_{j}(c_{2},...,c_j,c_{j+1}-1|q)^2}.
$$
\end{corollary} 

\begin{proof} 
Monotonicity of $x\mapsto[x]_q$ for $0<q<1$ gives the two inequalities in
\eqref{xiine}.  Substituting them into the first product formula of Proposition
\ref{jumpbo} yields
\[
 {\rm Jump}_x(q)\le {\rm Jump}_{x(j)}(q)
 \prod_{i=1}^{j}\frac{q^{c_i-1}}
 {[[c_{i+1},\ldots,c_j,c_{j+1}-1]]_q^2}.
\]
Finally, the product-of-tails identity for the length-$j$ string
$(c_2,\ldots,c_j,c_{j+1}-1)$ gives
\[
 \prod_{i=1}^{j}[[c_{i+1},\ldots,c_j,c_{j+1}-1]]_q
 =a_j(c_2,\ldots,c_j,c_{j+1}-1\mid q),
\]
which proves the Corollary.
\end{proof} 

Let 
$$
\phi(q,z):=\sum_{n\ge 2}\frac{q^{n-1}z^{n-1}}{[n]_q^2},\ \phi(q):=\phi(q,1).
$$
Since for any $r<1$,  for $n\ge 1,|q|\le r$ we have $|[n]_q|\ge 1-r$,
the series $\phi(q,z)$ converges absolutely in the region $|q|z<1$ and uniformly on its compact subsets, so defines a holomorphic function in this region.
Moreover, 
\(
\phi(q)=\sum_{n=2}^\infty \frac{1}{[n]_q[n]_{q^{-1}}}
\)
is a continuous, increasing function on $[0,1]$ with 
$\phi(0)=0$ and $\phi(1)=\sum_{n=2}^\infty \frac{1}{n^2}=\frac{\pi^2}{6}-1$. 
It follows that the function $\phi(q)(1+q)$ increases 
on $[0,1]$ from $0$ to $\frac{\pi^2}{3}-2\approx 1.29$.
Thus there exists a unique $\beta\in (0,1)$ 
solving the equation 
\(
\phi(q)(1+q)=1.
\)
Numerical computation shows that $\beta\approx 0.816$. 

\begin{proposition}\label{extrange1} (i) For $s\ge 1$, 
$$
{\rm Jump}_*(q,z)_{s+1}\le \phi(q)\phi(q,z)(1+q)^2{\rm Jump}_*(q,z)_s. 
$$
(ii) If $q<\beta$ then ${\rm Jump}(q,z)<\infty$ for some $z>1$. 

(iii) $q_*\ge \beta$. 
\end{proposition} 

\begin{proof}
(i) Let $m_0,m,p\ge1$, let $\boldsymbol c=(c_1,\ldots,c_N)$ be a
possibly empty string containing exactly $s-1$ entries $c_i\ge3$, and set
\[
 x=[[m_0+2,2^{(p-1)},m+2,c_1,\ldots,c_N]],
 \qquad
 y=[[m+2,c_1,\ldots,c_N]].
\]
Then $\ell(x)=\ell(y)+p$, and Corollary \ref{betterform}, applied with
$j=p$, gives
\begin{equation}\label{firstinteractionbound}
 {\rm Jump}_x(q)z^{\ell(x)}
 \le
 {\rm Jump}_y(q)z^{\ell(y)}
 \frac{q^{m_0+p}z^p}{B_{m,p}(q)^2},
 \qquad
 B_{m,p}(q):=a_p(2^{(p-1)},m+1\mid q).
\end{equation}
Here $2^{(r)}$ denotes a string of $r$ copies of $2$. Since
$a_r(2^{(r)}\mid q)=[r+1]_q$, the recurrence \eqref{rec2} gives
\[
 B_{m,p}(q)
 =[p]_q[m+1]_q-q[p-1]_q
 =1+q[m]_q[p]_q.
\]
Moreover,
\[
 [2]_qB_{m,p}(q)-[m+1]_q[p+1]_q
 =q\bigl([m]_q-1\bigr)\bigl([p]_q-1\bigr)\ge0,
\]
and therefore
\begin{equation}\label{Bmpbound}
 \frac1{B_{m,p}(q)^2}
 \le
 \frac{[2]_q^2}{[m+1]_q^2[p+1]_q^2}.
\end{equation}

It remains to keep track of the summation over $m_0$. For each such
string $\boldsymbol c$, put
\[
 u_{\boldsymbol c}:=[[2,c_1,\ldots,c_N]].
\]
Then $y=m+u_{\boldsymbol c}$, so translation equivariance gives
\[
 {\rm Jump}_y(q)=q^m{\rm Jump}_{u_{\boldsymbol c}}(q),
 \qquad
 \ell(y)=\ell(u_{\boldsymbol c}).
\]
Consequently,
\begin{align*}
 {\rm Jump}_*(q,z)_s
 &=\sum_{\boldsymbol c}\sum_{m\ge1}
 q^m{\rm Jump}_{u_{\boldsymbol c}}(q)z^{\ell(u_{\boldsymbol c})}\\
 &=\frac{q}{1-q}\sum_{\boldsymbol c}
 {\rm Jump}_{u_{\boldsymbol c}}(q)z^{\ell(u_{\boldsymbol c})},
\end{align*}
where the sums over $\boldsymbol c$ range over all finite strings with
exactly $s-1$ entries at least $3$. Summing \eqref{firstinteractionbound}
over $m_0,m,p\ge1$ and over these strings therefore yields
\(
 {\rm Jump}_*(q,z)_{s+1}
 \le H(q,z){\rm Jump}_*(q,z)_s,
\)
where the geometric factor $\sum_{m_0\ge1}q^{m_0}=\frac{q}{1-q}$ cancels
against the identical factor in the preceding expression for
${\rm Jump}_*(q,z)_s$, and
\(
 H(q,z):=\sum_{m,p\ge1}
 \frac{q^{m+p}z^p}{B_{m,p}(q)^2}.
\)
Using \eqref{Bmpbound}, we obtain
\begin{align*}
 H(q,z)
 &\le [2]_q^2
 \left(\sum_{m\ge1}\frac{q^m}{[m+1]_q^2}\right)
 \left(\sum_{p\ge1}\frac{q^pz^p}{[p+1]_q^2}\right)\\
 &=(1+q)^2\phi(q)\phi(q,z),
\end{align*}
which proves (i).

(ii) If $q<\beta$, choose $z>1$ sufficiently close to $1$ that
$qz<1$ and
\(
 (1+q)^2\phi(q)\phi(q,z)<1.
\)
Part (i), together with the finiteness of ${\rm Jump}_*(q,z)_1$, then
shows that 
\(
\sum_{s\ge1}{\rm Jump}_*(q,z)_s<\infty.
\)
Equation
\eqref{starformula} gives ${\rm Jump}(q,z)<\infty$.

(iii) follows from (ii) and Proposition \ref{extrange}.
\end{proof}

The bound of Proposition \ref{extrange1} was obtained 
by analyzing a single string $(2^{(p-1)},m+2)$, $m\ge 1$, preceded by an element $\ge 3$
in the negative continued fraction of $x$. 
We can further improve this bound by accounting for ``interactions'' between 
such strings. Let us start with taking into account the  
interactions of neighbors (distance $1$). 
Consider the series
$$
h_1(q,z)=\sum_{m=1}^\infty\sum_{p=1}^\infty  
\frac{q^{m+p}z^{p}}{a_{p+1}(m+2,2^{(p)}|q)^2},\ h_1(q):=h_1(q,1).
$$
We have 
\begin{equation}\label{ap}
a_{p+1}(m+2,2^{(p)}|q)=[m+1]_q[p+1]_q+q^{m+p+1},
\end{equation}
so 
$$
h_1(q,z)=\sum_{m,p\ge 1}\frac{q^{m+p}z^{p}}{([m+1]_q[p+1]_q+q^{m+p+1})^2}.
$$
For $q\in [0,1)$ the series $h_1(q,z)$ is dominated 
by the series $\phi(q)\phi(q,z)$, so it converges absolutely and uniformly on compact subsets in the region $qz<1$, and defines a continuous function in this region. Moreover, $h_1(q)$ extends to a continuous 
function on $[0,1]$ with $h_1(0)=0$ and 
$$
h_1(1)=\sum_{m,p\ge 1}\frac{1}{((m+1)(p+1)+1)^2}\approx 0.34.
$$
Let $\beta_1\in (0,1)$ be the smallest solution of the equation 
\begin{equation}\label{h=1}
h_1(q)(1+q)^2=1,
\end{equation} 
which exists since $4h_1(1)>1$. 

In fact, the left-hand side of \eqref{h=1} is strictly increasing on
$[0,1]$, so this solution is unique.  To see this, fix $m,p\ge1$, put
$d=m+p+1$,
\(
 P(q):=[m+1]_q[p+1]_q,\qquad A(q):=P(q)+q^d.
\)
The polynomial $P$ has positive coefficients and is palindromic of
degree $d-1$, hence
\(
 \frac{qP'(q)}{P(q)}\le\frac{d-1}{2}\qquad(0<q\le1).
\)
Also $P(q)\ge(m+1)(p+1)q^{d-1}$, and therefore
\[
 \frac{qA'(q)}{A(q)}
 \le \frac{d-1}{2}+\frac{d+1}{2}\frac{q^d}{P(q)+q^d}
 \le \frac{d-1}{2}+\frac{q}{1+q};
\]
the last inequality follows from
\(
 2\bigl((m+1)(p+1)+q\bigr)-(d+1)(1+q)
 =2mp+(m+p)(1-q)>0.
\)
Thus each summand
\[
 \frac{q^{m+p}(1+q)^2}
 {\bigl([m+1]_q[p+1]_q+q^{m+p+1}\bigr)^2}
\]
is strictly increasing.  Hence $(1+q)^2h_1(q)$ is strictly increasing,
and numerical computation gives
$\beta_1\approx{0.8631673307}$.

\begin{proposition}\label{extrange2} (i) For $s,k\ge 1$, 
$$
{\rm Jump}_*(q,z)_{s+k}\le \phi(q)h_1(q,z)^{k-1}\phi(q,z)(1+q)^{2k}{\rm Jump}_*(q,z)_s. 
$$
(ii) If $q<\beta_1$ then ${\rm Jump}(q,z)<\infty$ for some $z>1$. 

(iii) $q_*\ge \beta_1$. 
\end{proposition} 

\begin{proof}
\begingroup
We first establish a continuant inequality.  Fix $0<q<1$ and write
$\mathcal K_q(c_1,\ldots,c_r):=a_r(c_1,\ldots,c_r\mid q)$.
Set
\[
 u=(1,1),\qquad v=\begin{pmatrix}1\\q\end{pmatrix},\qquad
 T(m,p)=\begin{pmatrix}
 1&[p]_q\\
 q[m]_q&q[m]_q[p]_q+q^{m+p}
 \end{pmatrix}\quad(m,p\ge1).
\]
For $m_1,\ldots,m_k,p_1,\ldots,p_k\ge1$, put
\[
 \boldsymbol d=
 (2^{(p_1-1)},m_1+2,2^{(p_2-1)},\ldots,
  m_{k-1}+2,2^{(p_k-1)},m_k+1).
\]
The recurrence \eqref{rec2} gives
\begin{equation}\label{blockfactorprefix}
 \mathcal K_q(\boldsymbol d)
 =(1,[p_1]_q)
   \left(\prod_{i=2}^kT(m_{i-1},p_i)\right)
   \begin{pmatrix}1\\q[m_k]_q\end{pmatrix},
\end{equation}
with the matrix product ordered by increasing indices.  When $k=1$,
$\boldsymbol d=(2^{(p_1-1)},m_1+1)$ and the product is empty.
The same recurrence gives, for $r\ge1$,
\begin{equation}\label{blockfactorinterior}
 \mathcal K_q(m_1+2,2^{(p_1-1)},\ldots,m_r+2,2^{(p_r)})
 =u\left(\prod_{i=1}^rT(m_i,p_i)\right)v.
\end{equation}

Since $[m]_q,[p]_q\ge1$, right multiplication by $T(m,p)$ preserves
the cone of nonnegative rows $X=(X_1,X_2)$ with $X_2\ge X_1$;
left multiplication preserves the cone of nonnegative columns
$Y=(Y_1,Y_2)^{\mathsf T}$ with $Y_2\ge qY_1$.
For such a row and column,
\begin{equation}\label{blockfactorprojection}
 XY-\frac{(Xv)(uY)}{1+q}
 =\frac{(X_2-X_1)(Y_2-qY_1)}{1+q}\ge0.
\end{equation}
The endpoint row and column in \eqref{blockfactorprefix} lie in these
cones.  Applying \eqref{blockfactorprojection} at each of the $k$
cuts between the endpoint vectors and the intervening matrices, and
using \eqref{blockfactorinterior}, gives
\begin{equation}\label{blockfactorbound}
 \mathcal K_q(\boldsymbol d)
 \ge\frac{[p_1+1]_q[m_k+1]_q}{(1+q)^k}
     \prod_{i=2}^k\mathcal K_q(m_{i-1}+2,2^{(p_i)}).
\end{equation}
\endgroup

 Let $x=[[m_0+2,2^{(p_1-1)},m_1+2,...,2^{(p_k-1)},m_k+2,c_1,...,c_N]]$, $y=[[m_k+2,c_1,...,c_N]]$, where $m_i,p_i\ge 1$. {Applying Corollary \ref{betterform} with
$j=p_1+\cdots+p_k$, followed by \eqref{blockfactorbound}, gives}
$$
{\rm Jump}_x(q)z^{\ell(x)}\le {\rm Jump}_{y}(q)z^{\ell(y)}\frac{q^{\sum_{i=1}^k(m_{i-1}+p_i)}z^{\sum_{i=1}^k p_i}[2]_q^{2k}}{[p_1+1]_q^2\left(\prod_{i=2}^{k}a_{p_i+1}(m_{i-1}+2,2^{(p_i)}|q)^2\right)[m_k+1]_q^2}
$$
Sum this inequality over $m_0,\ldots,m_k\ge1$, over
$p_1,\ldots,p_k\ge1$, and over all finite strings
$c_1,\ldots,c_N$ containing exactly $s-1$ entries at least $3$.
The two endpoint sums combine exactly as in the proof of Proposition
\ref{extrange1}: the factor
$\sum_{m_0\ge1}q^{m_0}=\frac{q}{1-q}$ cancels against the translation-orbit
factor from the sum over $m_k$, while the remaining weighted endpoint
sum contributes $\phi(q)$. The sum over $p_1$ contributes
$\phi(q,z)$, and each pair $(m_{i-1},p_i)$ with $2\le i\le k$
contributes $h_1(q,z)$ by \eqref{ap}. Hence
\[
{\rm Jump}_*(q,z)_{s+k}
\le \phi(q)h_1(q,z)^{k-1}\phi(q,z)(1+q)^{2k}
{\rm Jump}_*(q,z)_{s},
\]
as claimed. 

(ii) follows from (i) and \eqref{starformula}. 

(iii) follows from (ii) and Proposition \ref{extrange}.
\end{proof} 

This bound can be further improved by taking into account string 
interactions at larger distances, up to distance $n$. 
To this end, consider the series
$$
h_n(q,z):=\sum_{m_1,...,m_n=1}^\infty\sum_{p_1,...,p_n=1}^\infty  
\frac{q^{\sum_{i=1}^n (m_i+p_i)}z^{\sum_{i=1}^n p_i}}{a_{p_1+...+p_n+1}(m_1+2,2^{(p_1-1)},m_2+2,2^{(p_2-1)},...,m_n+2,2^{(p_n)}|q)^2}.
$$
To make the case $n=2$ explicit, set
\(
 F(m,p\mid q):=[m]_q[p]_q+q^{m+p-1}.
\)
A direct induction from \eqref{rec2} gives
\(
 a_p(m+1,2^{(p-1)}\mid q)=F(m,p\mid q).
\)
We use the continuant concatenation identity
\begin{equation}\label{concon}
 a_{r+s}(\boldsymbol u,\boldsymbol v\mid q)
 =a_r(\boldsymbol u\mid q)a_s(\boldsymbol v\mid q)
 -q^{u_r-1}a_{r-1}(u_1,\ldots,u_{r-1}\mid q)
  a_{s-1}(v_2,\ldots,v_s\mid q).
\end{equation} 
Set
\[
 \boldsymbol u=(m_1+1,2^{(p_1-1)}),
 \qquad
 \boldsymbol v=(m_2+2,2^{(p_2-1)}).
\]
For these two blocks, \eqref{concon} gives
\[
 a_{p_2}(m_2+2,2^{(p_2-1)}\mid q)
 =[p_2]_q+qF(m_2,p_2\mid q)
\]
and
\[
 q^{u_{p_1}-1}a_{p_1-1}(u_1,\ldots,u_{p_1-1}\mid q)
 =F(m_1,p_1\mid q)-[m_1]_q.
\]
Consequently,
\begin{align*}
&D(m_1,p_1,m_2,p_2\mid q)\\
&\quad:=a_{p_1+p_2}
 (m_1+1,2^{(p_1-1)},m_2+2,2^{(p_2-1)}\mid q)\\
&\quad=F(m_1,p_1\mid q)
 \bigl([p_2]_q+qF(m_2,p_2\mid q)\bigr)
 -\bigl(F(m_1,p_1\mid q)-[m_1]_q\bigr)[p_2]_q\\
&\quad=[m_1]_q[p_2]_q
 +qF(m_1,p_1\mid q)F(m_2,p_2\mid q).
\end{align*}
Equivalently, \scriptsize
\begin{equation}\label{equivd}
D(m_1,p_1,m_2,p_2\mid q)
=q[m_1]_q[p_1]_q[m_2]_q[p_2]_q
 +q^{m_1+p_1}[m_2]_q[p_2]_q
+[m_1]_q[p_2]_q
 +q^{m_2+p_2}[m_1]_q[p_1]_q
 +q^{m_1+p_1+m_2+p_2-1}.
\end{equation} \normalsize
Thus 
\[
 h_2(q,z)=
 \sum_{m_1,p_1,m_2,p_2\ge1}
 \frac{q^{m_1+p_1+m_2+p_2}z^{p_1+p_2}}
 {D(m_1+1,p_1,m_2,p_2+1\mid q)^2}.
\]

In general, as explained in \cite{MO1}, 
$a_{p_1+...+p_n}(m_1+1,2^{(p_1-1)},m_2+2,2^{(p_2-1)},...,m_n+2,2^{(p_n-1)})$
is the numerator of the positive continued fraction 
$$
[m_1,p_1,m_2,p_2,...,m_n,p_n]_q=[[m_1+1,2^{(p_1-1)},m_2+2,2^{(p_2-1)},...,m_n+2,2^{(p_n-1)}]]_q.
$$
For $q\in [0,1)$ the series $h_n(q,z)$ is dominated 
by $\phi(q)(\phi(q)+1)^{n-1}(\phi(q,z)+1)^{n-1}\phi(q,z)$, so it converges absolutely and uniformly on compact subsets in the region $qz<1$, and defines a continuous function in this region. 
Let $h_n(q):=h_n(q,1)$ and define
\[
 \beta_n:=\sup\Bigl\{b\in[0,1]:
 h_n(w)(1+w)^2<1\ \text{for every }0\le w<b\Bigr\}.
\]
Thus $q<\beta_n$ implies $h_n(q)(1+q)^2<1$, 
which is the precise condition used below. Since $h_n$ is continuous and $h_n(0)=0$, if
$\beta_n<1$ then
\(
 h_n(\beta_n)(1+\beta_n)^2=1,
\)
so $\beta_n$ is the first solution of this equation. These constants can be computed
numerically. E.g., \eqref{equivd} implies that  
$\beta_2\approx0.9103$. 

\begin{proposition}\label{extrange3} (i) For $s,k\ge 1$, 
$$
{\rm Jump}_*(q,z)_{s+1+n(k-1)}\le \phi(q)h_n(q,z)^{k-1}\phi(q,z)(1+q)^{2k}{\rm Jump}_*(q,z)_s. 
$$
(ii) If $q<\beta_n$ then ${\rm Jump}(q,z)<\infty$ for some $z>1$. 

(iii) $q_*\ge \beta_n$. 
\end{proposition} 

\begin{proof}
\begingroup
(i) Use the notation of the proof of Proposition \ref{extrange2} and
put $L:=1+n(k-1)$.  For $m_0,\ldots,m_L,p_1,\ldots,p_L\ge1$ and
a possibly empty string $c_1,\ldots,c_N$ containing exactly $s-1$
entries at least $3$, set
\[
 \begin{aligned}
 x&=[[m_0+2,2^{(p_1-1)},m_1+2,\ldots,
        2^{(p_L-1)},m_L+2,c_1,\ldots,c_N]],\\
 y&=[[m_L+2,c_1,\ldots,c_N]].
 \end{aligned}
\]
Every string counted by ${\rm Jump}_*(q,z)_{s+L}$ has a unique such
decomposition.  Let
\[
 \boldsymbol d=(2^{(p_1-1)},m_1+2,\ldots,
                 m_{L-1}+2,2^{(p_L-1)},m_L+1)
\]
and group the interior matrices in \eqref{blockfactorprefix} into
$k-1$ consecutive products of length $n$:
\[
 G_j:=\prod_{i=(j-1)n+2}^{jn+1}T(m_{i-1},p_i),
 \qquad 1\le j\le k-1.
\]
For $k=1$ these products are absent and
$\boldsymbol d=(2^{(p_1-1)},m_1+1)$.
Applying \eqref{blockfactorprojection} at the $k$ cuts between the
endpoint vectors and these products yields
\begin{equation}\label{blockfactorgrouped}
 \mathcal K_q(\boldsymbol d)
 \ge\frac{[p_1+1]_q[m_L+1]_q}{(1+q)^k}
      \prod_{j=1}^{k-1}(uG_jv).
\end{equation}
By \eqref{blockfactorinterior}, each $uG_jv$ is the continuant in the
denominator defining $h_n$, with its $n$ pairs of indices taken from
the corresponding group.  Corollary \ref{betterform}, applied with
$j=\sum_{i=1}^L p_i=:P$, and \eqref{blockfactorgrouped} therefore give
\[
 \begin{aligned}
 {\rm Jump}_x(q)z^{\ell(x)}
 &\le {\rm Jump}_y(q)z^{\ell(y)}\,
 \frac{q^{\sum_{i=1}^L(m_{i-1}+p_i)}z^{P}(1+q)^{2k}}
 {[p_1+1]_q^2[m_L+1]_q^2
   \prod_{j=1}^{k-1}(uG_jv)^2}.
 \end{aligned}
\]
Sum over all the displayed indices and tails.  As in the proof of
Proposition \ref{extrange2}, the sum over $m_0$ cancels the
translation-orbit factor in ${\rm Jump}_*(q,z)_s$.  The weighted sum
over $m_L$ contributes $\phi(q)$, the sum over $p_1$ contributes
$\phi(q,z)$, and each of the $k-1$ groups contributes $h_n(q,z)$.
Consequently,
\[
 {\rm Jump}_*(q,z)_{s+L}
 \le\phi(q)h_n(q,z)^{k-1}\phi(q,z)(1+q)^{2k}
       {\rm Jump}_*(q,z)_s,
\]
which is (i).

(ii) If $0<q<\beta_n$, continuity allows us to choose $z>1$ such that
$qz<1$ and $A:=(1+q)^2h_n(q,z)<1$.  Part (i) gives
\[
 {\rm Jump}_*(q,z)_{s+1+n(k-1)}
 \le (1+q)^2\phi(q)\phi(q,z)A^{k-1}
       {\rm Jump}_*(q,z)_s.
\]
For $s=1,\ldots,n$ the terms on the right are finite, since $qz<1$,
and these $n$ progressions cover all indices at least $2$.
Summing the geometric bounds, and adding the finite term with
index $1$, gives $\sum_{j\ge1}{\rm Jump}_*(q,z)_j<\infty$.
Equation \eqref{starformula} now implies ${\rm Jump}(q,z)<\infty$.
The case $q=0$ is immediate.

(iii) follows from (ii) and the definition of $q_*$.
\endgroup
\end{proof} 

It seems plausible that $\beta_j\to 1$ as $j\to\infty$, which 
would imply the desired equality $q_*=1$. 

\section{Behavior of numerators and denominators of $q$-rationals when $|q|=1$}\label{unitcircle}

\subsection{The main result}\label{unitmain}

For a rational number $x=r/s\ge 1$ written in lowest terms, write
\[
 [x]_q=\frac{R_x(q)}{S_x(q)},
 \ R_x(q),S_x(q)\in\Bbb Z[q]
\]
in the standard normalization, so that
$R_x(1)=r, S_x(1)=s$ (\cite{MO1}); namely, in the notation of Section \ref{preli}, 
\(
R_x(q)=a_N(c_1,...,c_N|q),\ S_x(q)=b_N(c_1,...,c_N|q),
\)
where $x=[[c_1,...,c_N]]$.
For an arbitrary rational number $x$, choose an integer $n$ such that
$x+n\ge 1$ and extend $R_x,S_x$ by the translation identities
\begin{equation}\label{unittranslation}
 S_{x+n}(q)=S_x(q),\qquad
 R_{x+n}(q)=q^nR_x(q)+[n]_qS_x(q)
\end{equation}
(this definition is independent of the choice of $n$). 
Then for all $x\in \Bbb Q$ the polynomial $S_x$ and the Laurent polynomial $R_x$ are relatively prime. 
Note that for rational $x\ge 1$ and integer $n\ge 2$ 
we have 
\(
[n-\tfrac{1}{x}]_q=[n]_q-\tfrac{q^{n-1}}{[x]_q},
\)
hence 
\begin{equation}\label{RSeq}
R_x(q)=S_{n-\frac{1}{x}}(q).
\end{equation}
Also for $x\in \Bbb Q$ put
\(
 T_x(q):=(1-q)R_x(q)-S_x(q).
\)
This combination has an especially nice translation property:
\(
T_{x+n}(q)=q^n T_x(q).
\)
Let
\[
 C_-:=\left\{q\in\Bbb C:|q|=1,\ {\rm Re}\,q\le\frac12\right\},
 \qquad
 C_+:=\left\{q\in\Bbb C:|q|=1,\ {\rm Re}\,q>\frac12\right\}.
\]
 We regard
$[x]_q=[R_x(q):S_x(q)]$ as a point of $\Bbb C\Bbb P^1$, and set
\(
 X_q:=\{[x]_q:x\in\Bbb Q\}\subset\Bbb C\Bbb P^1.
\)

\begin{theorem}\label{unitmainthm}
Let $x\in\Bbb Q$ and $|q|=1$.

\begin{enumerate}[label=\textup{(\roman*)}]
\item If $q\in C_-$, then
\begin{equation}\label{unitmainineq}
 |T_x(q)|+(1-q^{-1}-q)|S_x(q)|\le |1-q|^2.
\end{equation}
Equality holds for every integer $x$, so the bound is sharp
for each $q\in C_-$.
At the endpoints $q=e^{\pm \pi i/3}$, 
\eqref{unitmainineq} turns into the equality
$|T_x(q)|=1$. 

\item If $q\in C_+$, then
\begin{equation}\label{unitoutsideineq}
 |T_x(q)|\ge\sqrt{q+q^{-1}-1}\,|S_x(q)|,
\end{equation}
i.e. if $S_x(q)\ne 0$ then 
\begin{equation}\label{unitoutsidevalue}
 \left|(1-q)[x]_q-1\right|\ge\sqrt{q+q^{-1}-1}.
\end{equation}

\item Suppose in addition that $q$ is not a root of unity. Then
\begin{equation}\label{unitclosure}
 \overline{X_q}=
 \begin{cases}
 \Bbb C\Bbb P^1, q\in C_-,\\
 \lbrace z\in\Bbb C\Bbb P^1: |(1-q)z-1|\ge\sqrt{q+q^{-1}-1}\rbrace,\ q\in C_+.
 \end{cases}
\end{equation}
 In particular,
the boundary circle in the second line of \eqref{unitclosure} lies in
$\overline{X_q}$, so the constants in \eqref{unitoutsideineq},
\eqref{unitoutsidevalue} are sharp.
\end{enumerate}
\end{theorem}

\begin{corollary}\label{corthe} (of Theorem \ref{unitmainthm}(i)) For $x\in \Bbb Q$ and $q\in C_-$, $q\ne e^{\pm \pi i/3}$ we have 
$$
|T_x(q)|\le |1-q|^2,\ |R_x(q)|\le \frac{|1-q|}{\min(1,1-q^{-1}-q)},\ |S_x(q)|\le \frac{|1-q|}{\min(1,1-q^{-1}-q)}.
$$
\end{corollary} 

\begin{proof} The first inequality follows directly from Theorem \ref{unitmainthm}(i). 
Now, we have $R_x(q)=\frac{T_x(q)+S_x(q)}{1-q}$. Let $M:=\min(1,1-q^{-1}-q)$. 
By Theorem \ref{unitmainthm}(i) 
$$
M(|T_x(q)|+|S_x(q)|)\le |T_x(q)|+(1-q^{-1}-q)|S_x(q)|\le |1-q|^2,
$$
so $|R_x(q)|\le \frac{|1-q|}{M}.$ 
Hence by \eqref{RSeq} $|S_x(q)|\le \frac{|1-q|}{M}$. 
\end{proof} 

The rest of this section is devoted to the proof of Theorem \ref{unitmainthm} and to some related
endpoint and growth phenomena.

\subsection{The Jones polynomial bound}\label{unitjones}

Set
\(
 J_x(q):=qR_x(q)+(1-q)S_x(q).
\)
For positive rational $x$, this is the normalized Jones polynomial of
the corresponding rational four-plat link, up to the usual choices of mirror
and a monomial factor; on the unit circle the latter has absolute value
one (\cite{MO1}, Appendix A; cf. also \cite{Jones,Kauffman}).

\begin{proposition}\label{unitarybound}
Let $x>0$ be rational. If $|q|=1$ and $q\in C_-$, then
\[
 |J_x(q)|\le |1-q|.
\]
\end{proposition}

\begin{proof}
The rational four-plat computation of $J_x$ takes place in the rank-two
Temperley--Lieb state space
$\mathcal H_q={\rm Hom}({\bf 1},X^{\otimes4})$, 
where $X$ is the 2-dimensional representation of the quantum group $SL(2)_{(-q)^{1/2}}$, where $(-q)^{1/2}$ 
is chosen to have positive real part. The quantum dimension of $X$ equals 
$$
d=(-q)^{1/2}+(-q)^{-1/2}=|(-q)^{1/2}+(-q)^{-1/2}|=|1-q|.
$$
In the basis $u_0,u_1\in \mathcal H_q$ of the two non-crossing pairings of four points,
the unitary Temperley--Lieb Hermitian form has Gram matrix
\[
 G_q=\begin{pmatrix}
 d^2& d\\
 d& d^2
 \end{pmatrix}.
\]
Hence the form is positive definite for $d>1$ and positive semidefinite for $d\ge 1$,
i.e., precisely on the arc $C_-$. Equivalently, this is Squier's
Hermitian form on the reduced
Burau representation $\rho_q$ of $B_3$ with parameter $t=-q$, and it is positive semidefinite precisely when
$|1+t|=|1-q|\ge1$; see \cite{Squier}.

Let $v_q$ be the two-cup plat vector. Closing $v_q$ with its dual
produces two loops, so
\(
 \|v_q\|^2=d^2.
\)
For a rational four-plat braid $\beta_x$ one has
\(
 |J_x(q)|=d^{-1}|\langle v_q,\rho_q(\beta_x)v_q\rangle|.
\)
The factor $d^{-1}$ makes the value of a single unknotted
circle equal to $1$.
On the interior of $C_-$ the Burau representation is unitary with respect
to the above positive definite form. Therefore Cauchy--Schwarz gives
\(
 |J_x(q)|\le d^{-1}\|v_q\|^2=d.
\)
The endpoint cases follow by continuity.
\end{proof}

\subsection{Proof of Theorem \ref{unitmainthm}(i)}\label{unitboundproof}

For any rational $x$, the translation identities \eqref{unittranslation}
give
\begin{equation}\label{unitJtranslate}
 (1-q)J_{x+n}(q)
 =q^{n+1}T_x(q)+(1-q+q^2)S_x(q).
\end{equation}

First suppose that $x>0$. Proposition \ref{unitarybound} applied to
$x+n$ gives
\(
 \left|q^{n+1}T_x(q)+(1-q+q^2)S_x(q)\right|\le |1-q|^2
\)
for every integer $n\ge0$. If $q$ is not a root of unity, its positive
powers are dense in the unit circle. Hence
\(
 \left|zT_x(q)+(1-q+q^2)S_x(q)\right|\le |1-q|^2
\)
for every $|z|=1$. Maximizing the left-hand side over $z$ yields
\(
 |T_x(q)|+|1-q+q^2|\,|S_x(q)|\le |1-q|^2.
\)
If $q$ is a root of unity, the same conclusion follows by approximating
$q$ within $C_-$ by points which are not roots of unity and using
continuity. Since
$|1-q+q^2|=1-q^{-1}-q$
on $C_-$, this proves \eqref{unitmainineq} for positive rational $x$.

For arbitrary $x\in\Bbb Q$, choose $n$ such that $y=x+n>0$. Since
$T_y(q)=q^nT_x(q),\ S_y(q)=S_x(q)$, the inequality for $y$ immediately
gives the inequality for $x$.

For $x=n\in\Bbb Z$, we have $S_n(q)=1$ and $T_n(q)=-q^n$.
Hence the left hand side of \eqref{unitmainineq} equals
$1+(1-q-q^{-1})=|1-q|^2$, proving sharpness.

It remains to prove the equality at the endpoints. 
Let $\zeta:=e^{\pi i/3}$.
Then $1-\zeta+\zeta^2=0,\ |1-\zeta|=1$.
Thus the second term on the left-hand side of
\eqref{unitmainineq} vanishes at $\zeta$.

\begin{proposition}\label{unitendpointequality}
For every rational $x$,
\[
 |T_x(\zeta)|=|T_x(\zeta^{-1})|=1.
\]
\end{proposition}

\begin{proof}
It is enough to prove the assertion for positive rational numbers and
then use translation. 

We treat $q=\zeta$; the other endpoint follows by complex conjugation.
We use the standard Laurent matrix form of the Morier--Genoud--Ovsienko
continued fraction. Put
\[
 M_Q(a):=\begin{pmatrix}
 [a]_Q&Q^a\\
 1&0
 \end{pmatrix}.
\]
Then a Laurent representative of $(R_x,S_x)$ is obtained from the product
\[
 M_{Q_1}(a_1)M_{Q_2}(a_2)\cdots M_{Q_k}(a_k)\binom10,
\]
where the parameters $Q_j$ alternate between $q$ and $q^{-1}$; see
\cite{MO1}. 

For $Q=\zeta$ or $Q=\zeta^{-1}$ one has $1-Q=Q^{-1}$. Set
\(
 \ell_Q:=(1-Q,-1).
\)
A direct calculation gives
\(
 \ell_QM_Q(a)=-Q^{a-1}\ell_{Q^{-1}},
\)
since
\[
 \ell_QM_Q(a)
 =\bigl((1-Q)[a]_Q-1,(1-Q)Q^a\bigr)
 =(-Q^a,Q^{a-1}).
\]
Iterating this identity through the continued-fraction product
and evaluating on $e_1=(1,0)^{\mathsf T}$ gives
\[
 |(1-\zeta)R_x(\zeta)-S_x(\zeta)|=1,
\]
since every scalar factor and $|1-\zeta^{\pm1}|$ have modulus one.
Translation gives the result for every rational $x$.
\end{proof}

This proves Theorem \ref{unitmainthm}(i).

\subsection{Behavior of $R_x,S_x,T_x$ on $C_+$}\label{unitcontrast}

Corollary \ref{corthe} implies uniform boundedness of $R_x,S_x$ with respect to $x$ in the interior
of $C_-$. But at the endpoints of $C_-$ this breaks since $1-q^{-1}-q=0$. 

In fact, on $C_+$ the behavior of absolute values of $R_x,S_x,T_x$ is very different from $C_-$. Namely, fix
$q=e^{i\theta},\ 0<|\theta|<\frac\pi3.$
Then $|R_x(q)|$ and $|S_x(q)|$ are unbounded as
$x\in(1,2)\cap\Bbb Q$ varies (and hence the same holds for $T_x(q)$ by 
Theorem \ref{unitmainthm}(ii)). This is already visible on the Fibonacci
convergents
\[
 x_m=[1;\underbrace{1,\ldots,1}_{2m-1}]
 =\frac{F_{2m+1}}{F_{2m}}\in(1,2),\ m\ge 2.
\]
A standard matrix formula for $q$-rationals (see \cite{MO1}) gives the following Laurent representative:
\[
 \binom{\widetilde R_m(q)}{\widetilde S_m(q)}
 =\mathcal A(q)^m\binom10,
 \qquad
 \mathcal A(q)=
 \begin{pmatrix}
 1+q&q^{-1}\\
 1&q^{-1}
 \end{pmatrix}.
\]
 Now
\[
 \det\mathcal A(q)=1,
 \qquad
 {\rm tr}\,\mathcal A(q)=1+q+q^{-1}=1+2\cos\theta.
\]
For $0<|\theta|<\pi/3$, this trace is real and greater than $2$.
Therefore $\mathcal A(q)$ is hyperbolic, with eigenvalues
$\lambda,\lambda^{-1}$, where $\lambda>1$. The second component satisfies
\[
 \widetilde S_{m+1}
 =(1+q+q^{-1})\widetilde S_m-\widetilde S_{m-1},
 \qquad
 \widetilde S_0=0,\quad \widetilde S_1=1,
\]
and hence
\[
 \widetilde S_m(q)=
 \frac{\lambda^m-\lambda^{-m}}{\lambda-\lambda^{-1}}.
\]
Thus $|S_{x_m}(q)|\to\infty$ exponentially fast. Since
\(
 \widetilde R_m=\widetilde S_{m+1}-q^{-1}\widetilde S_m
\)
and $\lambda\ne q^{-1}$, the same is true of the numerator:
$|R_{x_m}(q)|\to\infty$ exponentially fast. 

At the endpoints the matrix $\mathcal A$ is 
parabolic.  For $q=\zeta$ one has
$\operatorname{tr}\mathcal A(\zeta)=2$ and $\det\mathcal A(\zeta)=1$, so
\[
 \mathcal A(\zeta)^m=I+m(\mathcal A(\zeta)-I),\qquad
 \binom{\widetilde R_m(\zeta)}{\widetilde S_m(\zeta)}
 =\binom{1+m\zeta}{m}.
\]
Thus $|\widetilde S_m(\zeta)|=m$ and
$|\widetilde R_m(\zeta)|=\sqrt{m^2+m+1}$, whereas
\(
 (1-\zeta)\widetilde R_m(\zeta)-\widetilde S_m(\zeta)=\zeta^{-1}.
\)
Hence the coordinates grow linearly while the distinguished linear
combination remains bounded; the statement at $\zeta^{-1}$ follows by
complex conjugation.

\subsection{Proof of Theorem \ref{unitmainthm}(ii),(iii)}\label{unitclosureproof}

At $q=1$, one has
\(
 T_x(1)=-S_x(1),\qquad q+q^{-1}-1=1,
\)
so both \eqref{unitoutsideineq} and \eqref{unitoutsidevalue} hold with
equality. For the remainder of the proof, fix $q\in\Bbb C$ with
$|q|=1$ and $q\ne1$. For $|u|=1$, define the M\"obius map
\begin{equation}\label{unitFu}
 F_u(z):=\frac{1-qu}{1-q}-\frac{u}{z},
 \qquad z\in\Bbb C\Bbb P^1,
\end{equation}
with the usual projective conventions at $0$ and $\infty$. Let
$\mathcal R_q$ be the smallest subset of $\Bbb C\Bbb P^1$ which contains $1$
and is invariant under all the maps $F_u$, $|u|=1$.
Equivalently,
\[
 \mathcal R_q=
 \left\{F_{u_1}\circ\cdots\circ F_{u_k}(1):
 k\ge0,\ |u_1|=\cdots=|u_k|=1\right\}.
\]
Put
\(
 a:=\frac1{1-q},\qquad A:=|a|=\frac1{|1-q|}.
\)
Since $|q|=1$, one has
\(
 \overline a=-qa,\qquad a+\overline a=1,
\)
and hence
\begin{equation}\label{unitFucentered}
 F_u(z)=a+u\left(\overline a-\frac1z\right).
\end{equation}
For finite $z\ne0$ this implies the basic identity
\begin{equation}\label{unitreachinvariant}
 |F_u(z)-a|^2-(A^2-1)
 =\frac{|z-a|^2-(A^2-1)}{|z|^2}.
\end{equation}

We will need

\begin{lemma}\label{unitreachableset} We have
\[
 \mathcal R_q=
 \begin{cases}
 \Bbb C\Bbb P^1,&{\rm Re}\,q<\dfrac12,\\[5pt]
 \Bbb C\Bbb P^1\setminus\{a\},&{\rm Re}\,q=\dfrac12,\\[5pt]
 \left\{z\in\Bbb C\Bbb P^1:|z-a|>\sqrt{A^2-1}\right\},
 &{\rm Re}\,q>\dfrac12.
 \end{cases}
\]
\end{lemma}

\begin{proof}
For fixed $z$, varying $u$ fills the circle centered at $a$ with
radius $|\overline a-1/z|$.  Let $\mathcal R_n$ be the set reached after
exactly $n$ steps.
{
If $A>1$, identity \eqref{unitreachinvariant} and $|1-a|=A$
show that every reachable finite point satisfies
$|z-a|>\sqrt{A^2-1}$. At $z=0,\infty$, the same invariant-exterior
assertion follows directly from \eqref{unitFucentered}.  Thus every radius $\rho=|z-a|$ arising
from a reachable circle satisfies $\rho^2-(A^2-1)\ge0$.
For $A\le1$ this inequality is automatic.
For these radii, \eqref{unitreachinvariant} and $|z|\le A+\rho$
show that the smallest possible output radius is}

\[
 m(\rho)=\left|A-\frac1{A+\rho}\right|.
\]
Consequently, whenever
$\mathcal R_n=\{z:|z-a|\ge s_n\}$ (with $\infty$ included) and
$A-1/(A+s_n)>0$, connectedness gives
\begin{equation}\label{unitradiusrecurrence}
 \mathcal R_{n+1}=\{z:|z-a|\ge s_{n+1}\},\qquad
 s_{n+1}=A-\frac1{A+s_n}.
\end{equation}
Indeed, the radius function is continuous on this connected exterior
set, has the displayed minimum, and tends to $\infty$ at $z=0$.
Directly from $\mathcal R_1=\{z:|z-a|=A\}$ one obtains
\[
 \mathcal R_2=\left\{z:|z-a|\ge
 \left|A-\frac1{2A}\right|\right\}.
\]

If $A>1$, put $A=\cosh\eta$.
{As noted above, the invariant circle
$|z-a|=\sqrt{A^2-1}$ and its interior are excluded.}
With the auxiliary value
$s_1=A$, recurrence \eqref{unitradiusrecurrence} has the solution
\[
 s_n=\sinh\eta\,\coth(n\eta),
\]
so $s_n\searrow\sqrt{A^2-1}$ and the reachable set is precisely the
strict exterior.  If $A=1$, the same recurrence gives $s_n=1/n$;
thus every point except $a$ is reached.

Now let $A<1$.  Since $|q|=1$, one has $A\ge1/2$.  If
$A\le1/\sqrt2$, the set $\mathcal R_2$ contains both $0$ and
$1/\overline a$ (and is already the whole sphere when
$A=1/\sqrt2$); hence the next iterate has every radius and
$\mathcal R_3=\Bbb C\Bbb P^1$.  If $1/\sqrt2<A<1$, write
$A=\cos\phi$, $0<\phi<\pi/4$.  Until the whole sphere is reached,
\eqref{unitradiusrecurrence} gives
\[
 s_n=\sin\phi\,\cot(n\phi).
\]
For $N=\lceil\pi/(2\phi)\rceil$, the set $\mathcal R_{N-1}$ contains
both $0$ and $1/\overline a$, because
$s_{N-1}\le A$ and
$s_{N-1}\le\sin\phi\tan\phi=|1/\overline a-a|$.
Thus $\mathcal R_N=\Bbb C\Bbb P^1$.  Since $A>1$, $A=1$, and $A<1$
are respectively equivalent to
$\operatorname{Re}q>1/2$, $=1/2$, and $<1/2$, the result follows.
\end{proof}

We now apply Lemma \ref{unitreachableset} to $q$-rationals. For $c\ge2$, 
$F_{q^{c-1}}(z)=[c]_q-\frac{q^{c-1}}z$. 
If $x=[[c_1,\ldots,c_N]]>1$, its canonical eventually-$2$ expansion is
\(
 x=[[c_1,\ldots,c_N+1,2,2,\ldots]],
\)
and $[[2,2,\ldots]]_q=1$. Therefore
\begin{equation}\label{unitqratcomposition}
 [x]_q=
 F_{q^{c_1-1}}\circ\cdots\circ
 F_{q^{c_{N-1}-1}}\circ F_{q^{c_N}}(1).
\end{equation}
Conversely, every finite composition of maps $F_{q^n}$, $n\ge1$,
applied to $1$ is the $q$-value of an eventually-$2$ negative continued
fraction and hence is a $q$-rational number.
Explicitly,
\[
 F_{q^{n_1}}\circ\cdots\circ F_{q^{n_k}}(1)
 =[[n_1+1,\ldots,n_k+1,2,2,\ldots]]_q,
\]
so the corresponding classical continued fraction is rational and at least
$1$. Thus the set of
$q$-rationals with $x\ge1$ is precisely the orbit of $1$ under the maps
$F_{q^n}$, $n\ge1$.
Namely, denote this orbit by
$\mathcal O_q:=
 \left\{F_{q^{n_1}}\circ\cdots\circ F_{q^{n_k}}(1):
 k\ge0,\ n_j\ge1\right\}.$
Then
\[
 \mathcal O_q=X_q^{\ge1}:=\{[x]_q:x\in\mathbb Q,\ x\ge1\},
 \qquad \mathcal O_q\subset\mathcal R_q.
\]

Assume now that $q$ is not a root of unity. Then its positive powers are
dense in the unit circle. Given a finite composition
$F_{u_1}\circ\cdots\circ F_{u_k}(1)$ with $|u_j|=1$, choose positive
integers $n_j(m)$ such that $q^{n_j(m)}\to u_j$ for every $j$. Since
M\"obius transformations depend continuously on their parameters and
arguments in the spherical topology,
\[
 F_{q^{n_1(m)}}\circ\cdots\circ F_{q^{n_k(m)}}(1)
 \longrightarrow F_{u_1}\circ\cdots\circ F_{u_k}(1).
\]
It follows that the closure of the $q$-rationals with $x\ge1$ is
$\overline{\mathcal R_q}$.
Thus
\[
 \mathcal O_q\subset\mathcal R_q\subset\overline{\mathcal O_q},
\]
and therefore
\[
 \overline{X_q^{\ge1}}
 =\overline{\mathcal O_q}
 =\overline{\mathcal R_q}.
\]
In particular, Lemma
\ref{unitreachableset} gives $\overline{X_q}=\Bbb C\Bbb P^1$ when
$q\in C_-$ and $q$ is not a root of unity. Notice that in this case
${\rm Re}\,q<1/2$, since the two points with ${\rm Re}\,q=1/2$ are the
sixth roots of unity $e^{\pm \pi i/3}$.

Suppose next that $q\in C_+$. Then $A>1$. Identity
\eqref{unitreachinvariant} and the initial inequality
$|1-a|=A>\sqrt{A^2-1}$ show that every positive $q$-rational satisfies the inequality
\begin{equation}\label{ineee}
 |z-a|>\sqrt{A^2-1}.
\end{equation}

More precisely, for $x\ge1$ put $z=[x]_q$.  If $S_x(q)\ne0$, then $z$ is the
affine coordinate $R_x(q)/S_x(q)$ and \eqref{ineee} holds.  If
$S_x(q)=0$, then $z=\infty$; in this case the homogeneous inequality
\eqref{unitoutsideineq} is immediate because $R_x(q)\ne0$.
The translation formula gives
\begin{equation}\label{unittranslationcenter}
 [x+n]_q-a=q^n([x]_q-a).
\end{equation}
Equation \eqref{unittranslationcenter} is understood projectively: the
translation map fixes $\infty$, while for finite values it rotates the centered
coordinate by the unit scalar $q^n$.
So the same statement holds for every rational $x$. We have 
\[
 A^2-1
 =\frac1{|1-q|^2}-1
 =\frac{q+q^{-1}-1}{|1-q|^2}.
\]
Multiplying \eqref{ineee} by $|1-q|$ gives
\eqref{unitoutsidevalue}, and then multiplying by $|S_x(q)|$ gives
\eqref{unitoutsideineq}. If $S_x(q)=0$, the latter inequality is
immediate.

Finally, when $q\in C_+$ is not a root of unity, the density argument
and Lemma \ref{unitreachableset} show that the closure of the
$q$-rationals $[x]_q$ with $x\ge1$ is
\[
 \left\{z\in\Bbb C\Bbb P^1: |z-a|\ge\sqrt{A^2-1}\right\}.
\]
(in particular, this set contains $\infty$). 
Equation \eqref{unittranslationcenter} and the inequality just proved
show that adjoining all rational translates does not enlarge this
closure.

This proves \eqref{unitclosure} and completes the proof of
Theorem \ref{unitmainthm}. 

\section{$q$-real numbers modulo $m$}\label{modtwosection}

For every integer $m\ge2$, coefficientwise reduction of $[x]_q$ modulo $m$ separates real numbers $x$, and moreover 
separates points of the Cantor line.  Below we first prove this injectivity statement.  We
then take $m=p$ prime, so that $[x]_q$ mod $p$ takes values in 
$\mathbb P^1(\mathbb F_p((q)))$, and study rational and quadratic irrational points.
For $p=2$ we show that the map $x\mapsto [x]_q$ is also surjective, hence bijective; moreover, it is a homeomorphism.

\subsection{Injectivity modulo $m$ and bijectivity modulo $2$}
Let
\[
 \R_{\rm C}:=(\R\setminus\Q)\sqcup\Q_+\sqcup\Q_-
 \sqcup\{-\infty,+\infty\}
\]
be the \emph{Cantor line}: for $r\in\Q$ its two copies satisfy
$r_-<r_+$, and the points are ordered in the evident way.  We endow $\Bbb R_{\rm C}$ with the order topology (making it homeomorphic to the Cantor set) 
and put
\[
 \R_{\rm C}^+:=[0_+,+\infty],\qquad
 \R_{\rm C}^1:=[1_+,+\infty].
\]
For $m\ge2$, let
$H_m:=(1-q)^{-1}\in(\mathbb Z/m\mathbb Z)[[q]]$, and denote
coefficientwise reduction by $\operatorname{red}_m$.  Define
\[
 \Phi_m:\mathbb R_{\rm C}\longrightarrow
 (\mathbb Z/m\mathbb Z)((q))\sqcup\{\infty\}
\]
by $\Phi_m(x)=\operatorname{red}_m([x]_q)$ for
$x\in(\mathbb R\setminus\mathbb Q)\sqcup\mathbb Q_+$,
$\Phi_m(r_-)=\operatorname{red}_m([r]_q^-)$, and
$\Phi_m(+\infty)=H_m$, $\Phi_m(-\infty)=\infty$.
Equip $(\mathbb Z/m\mathbb Z)((q))$ with its $q$-adic topology and the
adjoined point $\infty$ with the one-point compactification topology.

For $x\in\mathbb R_{\rm C}^1$ we use its split negative
continued-fraction code.  The point $+\infty$ has code $(\infty)$ and
$1_+$ has code $(2,2,\ldots)$.  If
$r=[[c_1,\ldots,c_N]]>1$, then
\[
 \operatorname{code}(r_-)=(c_1,\ldots,c_N,\infty),\qquad
 \operatorname{code}(r_+)=(c_1,\ldots,c_N+1,2,2,\ldots),
\]
while an irrational has its usual infinite code.  The $i$-th tail is
the point represented by the code beginning with its $i$-th symbol;
the tail beginning with $\infty$ is $+\infty$.

Reduction of the translation formula gives
\begin{equation}\label{modmtranslation}
 \Phi_m(x+n)=q^n\Phi_m(x)+[n]_q\qquad(n\in\mathbb Z).
\end{equation}
A common finite prefix forces agreement modulo a power of $q$ tending
to infinity with the prefix length.  At a terminating code, a nearby
next digit $a\to\infty$ gives agreement with $H_m$ modulo $q^{a-1}$.
Thus $\Phi_m$ is continuous at every finite point.  Finally, write
$x=-N+y$ with $0_+\le y<1_+$.  Then
$\Phi_m(y)\in q(\mathbb Z/m\mathbb Z)[[q]]$ and
\[
 \Phi_m(x)=q^{-N}\bigl(\Phi_m(y)-[N]_q\bigr),
\]
whose parenthesis has constant term $-1$.  Hence
$\operatorname{val}_q\Phi_m(x)=-N$, proving continuity at $-\infty$.

\begin{theorem}  The map $\Phi_m$ is injective for every $m\ge2$.  If
$p$ is prime, $\Phi_p$ is a closed embedding into
$\mathbb P^1(\mathbb F_p((q)))$; for $p=2$ it is a homeomorphism.
\end{theorem}

\begin{proof}
Let $x\in\mathbb R_{\rm C}^1$ have code
$(a_1,a_2,\ldots)$, finite or infinite, and let $x_i$ be its $i$-th
tail.  For a nonterminal tail, put
$F_i:=\Phi_m(x_i)\in1+q(\mathbb Z/m\mathbb Z)[[q]]$.  Then
\(
 F_i=[a_i]_q-q^{a_i-1}/F_{i+1},
\)
so $[q^j]F_i=1$ for $0\le j<a_i-1$ and
$[q^{a_i-1}]F_i=0$.  Hence
\[
 a_i=1+\min\{j\ge1:[q^j]F_i\ne1\},\qquad
 F_{i+1}=\frac{q^{a_i-1}}{[a_i]_q-F_i}.
\]
The value $H_m$ marks termination.  Iteration therefore recovers the
entire code from $\Phi_m(x)$, proving injectivity on
$\mathbb R_{\rm C}^1$ and, by \eqref{modmtranslation}, on
$\mathbb R_{\rm C}$.

For $m=2$, every series $F\in1+q\mathbb F_2[[q]]$ other than $H_2$ has
a first zero coefficient which follows a string of ones.  The displayed formulas therefore produce a
unique finite or infinite code, so
$\Phi_2:\mathbb R_{\rm C}^1\to1+q\mathbb F_2[[q]]$ is bijective.  Finally, by \eqref{modmtranslation}, 
for every $G\in\mathbb F_2((q))$, one can choose $n\gg0$ so that
$[n]_q+q^nG\in1+q\mathbb F_2[[q]]$. Translating back therefore gives a
unique preimage of $G$.  Since $-\infty$ maps to $\infty$, $\Phi_2$ is
bijective on the full spaces. Finally, compactness of $\mathbb R_{\rm C}$ and
$\mathbb P^1(\mathbb F_p((q)))$ implies that $\Phi_p$ is a closed embedding for every prime $p$, and 
$\Phi_2$ is a homeomorphism. 
\end{proof}

\subsection{$\Phi_p(x)$ for rational and quadratic irrational $x$}
Define the matrix
\[
 M_a^{(p)}(q):=
 \begin{pmatrix}[a]_q&-q^{a-1}\\1&0\end{pmatrix}
 \in\operatorname{Mat}_2(\mathbb F_p[q]).
\]
It represents the fractional-linear map
$z\mapsto[a]_q-\frac{q^{a-1}}{z}$.

\begin{proposition}\label{modtwoalgebraicity}
\begin{enumerate}[label=\textup{(\roman*)}]
 \item If $x$ is $\pm\infty$, $r_+,r_-$, or a real
 quadratic irrational, then $\Phi_p(x)$ is algebraic of degree at most
 $2$ over $\mathbb F_p(q)$.
 \item If $\Phi_p(x)\in\mathbb P^1(\mathbb F_p(q))$, then $x$ is $\pm\infty$, $r_+,r_-$, or a real quadratic irrational.\footnote{However, $x$ can be very complicated even for a fairly simple function $\Phi_p(x)$. For example, a computer calculation shows that 
if $\Phi_2(x)=(1+q)^{10}=1+q^2+q^8+q^{10}$ then the discriminant of the real quadratic field $\Bbb Q(x)$ has {104} decimal digits.}
\end{enumerate}
\end{proposition}

\begin{proof}
The points $\pm\infty,r_+,r_-$ have rational projective images.  If $x$
is a real quadratic irrational, its negative continued fraction is
eventually periodic.  The image of its periodic tail is therefore a
fixed point of a fractional-linear transformation over
$\mathbb F_p(q)$, so it has degree at most $2$; the finite preperiod does
not increase the degree.

Conversely, translate a finite $x$ into $\mathbb R_{\rm C}^1$.  Let
$x_i$ be its successive tails, $a_i$ their first digits, and
$F_i:=\Phi_p(x_i)$.  If $F_i=P_i/Q_i$ in lowest terms, normalized by
$P_i(0)=Q_i(0)=1$, then
\[
 E_i:=q^{-a_i+1}\bigl([a_i]_qQ_i-P_i\bigr)\in\mathbb F_p[q],
 \qquad F_{i+1}=\frac{Q_i}{E_i}.
\]
Here $E_i(0)=1$, $(Q_i,E_i)=1$, and
\(
 \deg E_i\le
 \max\{\deg Q_i,\deg P_i-a_i+1\}.
\)
Thus $\max(\deg P_i,\deg Q_i)$ never increases.  Only finitely many
normalized coprime pairs of bounded degree exist over $\mathbb F_p$.
Hence either a terminal tail is reached, or two tails coincide; in the
latter case the subsequent digits repeat.  Therefore $x$ is a split
rational point or a real quadratic irrational. 
\end{proof}

\begin{proposition}\label{modtworationality}
(i) Let 
\[
 x=[[b_1,\ldots,b_s,\overline{c_1,\ldots,c_\ell}]]
 \in\R_{\rm C}^1
\]
have an infinite eventually periodic code.  Write
\[
 M_{c_1}^{(p)}\cdots M_{c_\ell}^{(p)}
 =\begin{pmatrix}A&B\\C&D\end{pmatrix}.
\]
Then the image $Y=\Phi_p([[\overline{c_1,\ldots,c_\ell}]])$ 
of the purely periodic tail of $x$ is the unique root in
$1+q\mathbb F_p[[q]]$ of the equation
\begin{equation}\label{modtwofixedquadratic}
 CY^2+(D-A)Y-B=0.
\end{equation}
Consequently, $\Phi_p(x)$ is rational if and only if this polynomial
splits over $\mathbb F_p(q)$; otherwise it has degree exactly $2$.
For odd $p$, splitting is equivalent to
\[
 (D-A)^2+4BC\quad\text{being a square in }\mathbb F_p(q).
\]
For $p=2$, it is equivalent to
\begin{equation}\label{modtwoAScriterion}
 \frac{BC}{(A+D)^2}
 \in\{u^2+u:u\in\mathbb F_2(q)\}.
\end{equation}

(ii) Every terminating code $x=(b_1,\ldots,b_s,\infty)$ has rational
image $\Phi_p(x)\in \Bbb F_p(q)$.
\end{proposition}

\begin{proof}
(i) The fixed-point equation for the period matrix is
\eqref{modtwofixedquadratic}.  At $q=0$ it is $Y^2-Y=0$, and its
root $Y=1$ is simple, so Hensel's lemma gives a unique root in
$1+q\mathbb F_p[[q]]$.  The preperiod acts by an invertible
fractional-linear transformation over $\mathbb F_p(q)$, so it preserves
rationality.  For odd $p$, splitting is the discriminant condition.  In
characteristic $2$, equation \eqref{modtwofixedquadratic} is
\(
 CY^2+(A+D)Y+B=0.
\)
Since $A+D$ has constant term one, the substitution
$Z=CY/(A+D)$ gives
\(
 Z^2+Z=\frac{BC}{(A+D)^2},
\)
which proves \eqref{modtwoAScriterion}.

(ii) For a terminating code, the terminal tail has image
$H_p=(1-q)^{-1}\in\mathbb F_p(q)$.  Applying the finitely many
fractional-linear maps represented by
$M_{b_s}^{(p)},\ldots,M_{b_1}^{(p)}$ preserves rationality.
\end{proof}

\subsection{Quadratic series over $\mathbb F_2(q)$ and eventual periodicity}

Quadraticity over $\mathbb F_2(q)$ does not by itself imply eventual
periodicity.  For $F\in1+q\mathbb F_2[[q]]$, set $F_1=F$.  If $F_i=H_2$,
stop.  Otherwise define
\[
 a_i:=1+\min\{j\ge1:[q^j]F_i=0\},\qquad
 F_{i+1}:=\frac{q^{a_i-1}}{[a_i]_q-F_i}.
\]
Then $(a_i)$ is the negative continued-fraction code of
$\Phi_2^{-1}(F)$, and $F_i$ is the image of its $i$-th tail.

Suppose that $F$ has degree $2$ over
$\mathbb F_2(q)$.  Every $F_i$ then also has degree $2$, since it is an
invertible fractional-linear transform of $F$.  For each $i$, let
\(
 A_iT^2+B_iT+C_i\in\mathbb F_2[q][T]
\)
be the unique primitive irreducible polynomial satisfied by $F_i$.

\begin{theorem}
The code $(a_i)$ is eventually periodic if and only if the three
sequences $(\deg A_i)$, $(\deg B_i)$, and $(\deg C_i)$ are bounded.
\end{theorem}

\begin{proof}  Eventual periodicity is equivalent to equality
$F_i=F_j$ for some $i<j$: the forward implication is immediate, and the
reverse implication follows from the displayed deterministic
recursion.  Repetition plainly implies bounded degrees.  Conversely,
there are only finitely many primitive quadratic polynomials over
$\mathbb F_2[q]$ with bounded coefficient degrees, and each has at most
two roots.  Hence only finitely many $F_i$ occur, so two of them are
equal. 
\end{proof} 

\begin{example}
Let
\(
 F=1+q+q^2+q^4+q^8+\cdots,\qquad F^2+F=q.
\)
The polynomial $T^2+T+q$ is irreducible over $\mathbb F_2(q)$ by the
valuation at infinity.  The first zero coefficient of $F$ is in degree
$3$, so the first digit is $4$ and the next tail has minimal polynomial
\(
 (1+q+q^3)T^2+T+q^3.
\)
More generally, suppose a tail has minimal polynomial
$AT^2+T+C$, with $\deg A=\deg C=d$, and its next digit is $n+1$.
Substituting $F_i=[n+1]_q-q^n/F_{i+1}$ gives
\[
 q^{-n}\bigl(A[n+1]_q^2+[n+1]_q+C\bigr)T^2+T+Aq^n
\]
for the next tail.  The expression in parentheses is divisible by
$q^n$, and the middle coefficient is $1$, so this polynomial is
primitive.  Its leading and constant coefficients both have
degree $d+n$.  Starting from the displayed polynomial, these degrees
therefore increase strictly, so the criterion shows that the code is
not eventually periodic.  It cannot terminate either, since a
terminating code has rational image.  Thus $\Phi_2^{-1}(F)$ is neither
a split rational point nor a real quadratic irrational.
\end{example}

\subsection{Numbers $x$ with eventually even and eventually odd coefficients of $[x]_q$}
\begin{corollary}\label{modtwoparity}
Let $x$ have an eventually periodic code as in
Proposition~\ref{modtworationality}(i), and assume its splitting
condition.  Write the unit root of
\eqref{modtwofixedquadratic} in lowest terms as $Y=P/Q$, and put
\[
 M_a(q):=M_a^{(2)}(q)
 =\begin{pmatrix}[a]_q&q^{a-1}\\1&0\end{pmatrix},
\]
\[
 M_{b_1}\cdots M_{b_s}
 =\begin{pmatrix}A_0&B_0\\C_0&D_0\end{pmatrix},\qquad
 N:=A_0P+B_0Q,\quad \Delta:=C_0P+D_0Q.
\]
Then $\Phi_2(x)=N/\Delta$, and
\begin{align}
 \Phi_2(x)\in\mathbb F_2[q]
 &\iff \Delta\mid N,\label{modtwopolynomialcriterion}\\
 \Phi_2(x)\in H_2+\mathbb F_2[q]
 &\iff (1-q)\Delta\mid\Delta+(1-q)N.
 \label{modtwooddcriterion}
\end{align}
For a terminating code $(b_1,\ldots,b_s,\infty)$, the same formulas
hold with $P=1$ and $Q=1+q$. 
These tests are exhaustive.  Indeed, eventual evenness or oddness
makes $\Phi_2(x)$ rational; Proposition~\ref{modtwoalgebraicity}(ii)
then forces a terminating or eventually periodic code.
Formula \eqref{modmtranslation}, together with
$H_2=1+qH_2$, extends the criteria from $\mathbb R_{\rm C}^1$ to
$\mathbb R_{\rm C}^+$.

In particular, for $x\ge0$, write
$[x]_q=\sum_{j\ge0}c_jq^j$.  The coefficients $c_j$ are eventually
even exactly when \eqref{modtwopolynomialcriterion} holds, and are
eventually odd exactly when \eqref{modtwooddcriterion} holds.  For $x=r_-$ the same statement applies to $[x]_q^-$.
\end{corollary}

\begin{proof}
The preperiod matrix gives $\Phi_2(x)=N/\Delta$.  The first criterion is
immediate, while
\[
 H_2+\Phi_2(x)=\frac{\Delta+(1-q)N}{(1-q)\Delta}
\]
gives the second. The reduced series is a polynomial exactly
when the integral coefficients are eventually even, and it belongs to
$H_2+\mathbb F_2[q]$ exactly when they are eventually odd. 
\end{proof}

\subsection{Numbers with $[x]_q=1+q^n$ mod $2$}
As an example, let us compute the real numbers whose mod $2$
$q$-analogue is $1+q^n$.

\begin{proposition}\label{modtwooneplusqn}
Let $x_n:=\Phi_2^{-1}(1+q^n)$.  Then
$x_1=2_+,\qquad x_2=2_-$, and
for every $n\ge3$,
\begin{equation}\label{modtwoonepluscode}
 x_n=[[2,2,\overline{2^{(n-3)},4,2^{(n-3)},3}]],
\end{equation}
where $2^{(r)}$ denotes a string of $r$ copies of $2$. Thus
\begin{equation}\label{modtwooneplusvalue}
 x_n=
 \frac{n(n+1)-\sqrt{(n-2)(n-1)n(n+1)}}{2n-1},
\end{equation}
and it is a quadratic irrational. 
Thus the $x_n$ are precisely the points of $\mathbb R_{\rm C}^1$ for which $\Phi_2(x)$ has exactly one nonzero coefficient of positive degree. 
\end{proposition}

\begin{proof}
The first two cases follow from
$[2]_q=1+q$ and $[2]_q^-=1+q^2$.  Let $n\ge3$ and set
\[
 S_j:=1+q^j+\cdots+q^{n-1}\quad(1\le j\le n-1),\qquad S_n:=1,
\]
\[
 R_j:=[j+1]_q+q^{n-1}\quad(0\le j\le n-3),
 \qquad R_{n-2}:=S_1.
\]
For $\Psi_a(G):=q^{a-1}/([a]_q-G)$, direct calculation gives
\begin{align*}
 \Psi_2(1+q^n)&=\frac{S_n}{S_{n-1}},\\
 \Psi_2\left(\frac{S_j}{S_{j-1}}\right)
 &=\frac{S_{j-1}}{S_{j-2}} &&(3\le j\le n),\\
 \Psi_4\left(\frac{S_2}{S_1}\right)
 &=\frac{R_{n-2}}{R_{n-3}},\\
 \Psi_2\left(\frac{R_j}{R_{j-1}}\right)
 &=\frac{R_{j-1}}{R_{j-2}} &&(2\le j\le n-2),\\
 \Psi_3\left(\frac{R_1}{R_0}\right)
 &=\frac{S_{n-1}}{S_{n-2}}.
\end{align*}
All displayed outputs have constant term $1$, so the arrow labels are
the corresponding continued-fraction digits.  Thus the first two
digits are $2,2$, after which the period is
$2^{(n-3)},4,2^{(n-3)},3$, proving
\eqref{modtwoonepluscode}.

Let $L_a=\begin{pmatrix}a&-1\\1&0\end{pmatrix}$ and $k=n-3$.  Then
\[
 L_2^2(L_2^kL_4L_2^kL_3)L_2^{-2}
 =\begin{pmatrix}-2n-1&2n(n+1)\\1-2n&2n^2-1\end{pmatrix}.
\]
Hence
\(
 (2n-1)x_n^2-2n(n+1)x_n+2n(n+1)=0.
\)
The code gives $1<x_n<2$, so $x_n$ is the smaller root, namely
\eqref{modtwooneplusvalue}.  Finally, with $A=n(n-1)$, the radicand is
$A(A-2)=(A-1)^2-1$, which lies strictly between consecutive squares;
therefore $x_n$ is quadratic irrational. 
\end{proof} 

\section{$q$-complex numbers}  

\subsection{Definition and formulas for $[\tau]_q$}
Let $0<q<1$ and let $t:=-\log q>0$, so that $q=e^{-t}$. We now want to 
extend the notion of a $q$-number to complex numbers. 
To this end, we want to define a meromorphic function $f(t,\tau)=[\tau]_q$ in $\tau=x+iy\in \Bbb C_+$, i.e., a holomorphic map 
\(
f_t: \Bbb C_+\to \Bbb C\Bbb P^1,\ f_t(\tau):=f(t,\tau),
\)
which is continuous in $(t,\tau)$ with $f(0,\tau)=\tau$, modular equivariant as in Lemma \ref{modinv}, and whose boundary values in a suitable sense 
are $[x]_q$, $x\in \Bbb R$. Thus we must have
\(
f(t,\tau+1)=e^{-t}f(t,\tau)+1.
\)
Hence $g(t,\tau):=e^{t\tau}f(t,\tau)$ must satisfy the equation 
\(
g(t,\tau+1)=g(t,\tau)+e^{t(\tau+1)}.
\)
Such a function must have the form 
\(
g(t,\tau)=h(t,\tau)+\frac{e^{t\tau}}{1-e^{-t}},
\)
where $h(t,\tau)$ is periodic in $\tau$ with period $1$, i.e., 
\(
h(t,\tau)=H(t,\bold q^2)
\)
where $\bold q=e^{\pi i\tau}$
and $H(t,z)$ is meromorphic in $0<|z|<1$. Thus 
\begin{equation}\label{fhform}
f(t,\tau)=e^{-t\tau}h(t,\tau)+\frac{1}{1-e^{-t}},\ h(t,\tau)=e^{t\tau}\left(f(t,\tau)-\frac{1}{1-e^{-t}}\right).
\end{equation}
Modular equivariance also requires that
\(
f\left(t,-\frac{1}{\tau}\right)=-\frac{e^t}{f(t,\tau)}.
\)

Let $j(\tau)$ be the $j$-invariant of the elliptic curve $E_\tau:=\Bbb C/(\Bbb Z\oplus \tau\Bbb Z)$ and $J(\tau):=j(\tau)/1728$. It follows 
 that the Schwarzian derivative\footnote{Recall that $\Bbb S(f)=\frac{f'''}{f'}-\frac{3}{2}\left(\frac{f''}{f'}\right)^2$, where the derivatives here are with respect to $J$.} $\Bbb S(f)$ of $f(t,\tau)$ with respect to $J(\tau)$ 
is a meromorphic modular function\footnote{Below we will extensively use hypergeometric 
and modular functions. We refer the reader to \cite{WW} for basics 
on these functions.}
 (i.e., a function on $\Bbb C_+/PSL_2(\Bbb Z)$). 
In other words, $\Bbb S(f)$ is a meromorphic function of $J$:
\(
\Bbb S(f)(\tau)=Q(J(\tau))
\)
where $Q\in {\rm Mer}(\Bbb C)$.
Thus $f(t,\tau)=\frac{\psi_1(J(\tau))}{\psi_2(J(\tau))}$, where $\psi_1,\psi_2$ are a basis of solutions of the linear second order ODE 
\begin{equation}\label{pf}
\frac{d^2\psi}{dJ^2}+\frac{1}{2}Q(J)\psi=0.
\end{equation} 
{
Near a finite point $J_0$, choose a local coordinate $w$ on the upper
half-plane such that $J-J_0=w^e$.  After a fractional linear change of
the target, write $f=w^k u(w)$, where $k\ge1$ and $u(0)\ne0$.
The resulting Puiseux expansion gives
\[
 Q(J)=\frac{1-(k/e)^2}{2(J-J_0)^2}
       +O\bigl((J-J_0)^{-2+1/e}\bigr).
\]
Since $Q$ is single-valued in $J$, it has at most a double pole at
$J_0$.  Thus \eqref{pf} has a regular or regular-singular point there.
}

Recall that $[x]_q\to \frac{1}{1-q}$ as $x\to +\infty$, 
but $[x]_q\to -\infty$ as $x\to -\infty$. Thus we should not expect $[\tau]_q$ to have a limit as $\tau\to \infty$. But we may require that it has limit $\frac{1}{1-q}$ when ${\rm Re}\tau,{\rm Im}\tau\to +\infty$. Namely, we make 

\vskip .05in
{\bf Assumption 1.} We have
$\lim_{{\rm Re}\tau,{\rm Im}\tau\to+\infty}f(t,\tau)=(1-e^{-t})^{-1}$.
\vskip .05in

By the removable singularity theorem, Assumption~1 is equivalent to
$H(t,z)$ being holomorphic at $z=0$.  It follows that $Q$ is holomorphic
at $\infty$, hence rational, and has at least a second-order zero there.

Let us also make

\vskip .05in
{\bf Assumption 2.} The map $\tau\mapsto f(t,\tau)$ is a local diffeomorphism, i.e., $df_\tau(t,\tau)\ne 0$ for all 
$\tau\in \Bbb C_+$. 
\vskip .05in

This is natural since $f(0,\tau)=\tau$. Assumption 2 implies that the function $Q$ 
is regular on $\Bbb C$ except $J=0$ and $J=1$, and at $0$ and $1$ we have 
\(
Q\sim \frac{4}{9J^2},\ Q\sim \frac{3}{8(J-1)^2}.
\)
Hence we have 
\begin{equation}\label{pfeq}
\frac{1}{2}Q=\frac{2}{9J^2}+\frac{3}{16(J-1)^2}+\frac{36c-23}{144J(J-1)}=\frac{36J^2-41J+32+36cJ(J-1)}{144J^2(J-1)^2}, \ c\in \Bbb C. 
\end{equation} 
Thus for $c=0$ equation \eqref{pf} is the classical Picard-Fuchs equation, and when $c\ne 0$, it is a $q$-deformed Picard-Fuchs equation. Namely, it is easy to check that the relation between $c$ and $q$ (i.e., $t$) is 
$$
c=-s^2=\frac{t^2}{4\pi^2},\text{ where } 
s:=\frac{t}{2\pi i}.
$$ 
Indeed, compare the local exponents at $J=\infty$.  From
\eqref{pfeq},
\(
 \frac12Q(J)=\frac{1+c}{4J^2}+O(J^{-3}),
\)
so the indicial roots of \eqref{pf} are
\[
 \frac{1+\sqrt{-c}}2,
 \qquad
 \frac{1-\sqrt{-c}}2.
\]
Thus the ratio of the two local monodromy eigenvalues is
$\exp(\mathord{\pm}2\pi i\sqrt{-c})$.  Let $\delta(t)$ be the
difference of the two local exponents, chosen continuously along the
deformation with $\delta(0)=0$.  Since the cusp monodromy is conjugate
to multiplication by $q=e^{-t}=e^{-2\pi i s}$, one has
$\delta(t)=\mathord{\pm}s+n$ for a locally constant integer $n$.  At
$t=0$ the equation is the classical Picard--Fuchs equation, so
$c(0)=0$ and $n=0$.  Therefore $\delta(t)^2=s^2$, and hence
$c=-s^2$.

Equation \eqref{pf} has Fuchsian singularities at $0,\infty,1$ (corresponding to $\tau=\rho:=e^{2\pi i/3}$, $\tau=\infty$ and $\tau=i$) and projective local monodromies represented by the matrices $(T_tS_t)^{-1},T_t,S_t$ respectively, where 
$$
S_t=ie^{-t/2}\begin{pmatrix} 0 & -e^{t}\\ 1 & 0\end{pmatrix},\  T_t:=ie^{t/2}\begin{pmatrix} e^{-t} & 1\\ 0 & 1\end{pmatrix}, 
$$
so that $T_tS_t=\begin{pmatrix} -1 & 1\\ -1 & 0\end{pmatrix}$ and $S_t^2=(T_tS_t)^3=1$.
These projective transformations come from the Burau representation of the braid group $B_3$; see \cite{MOV}.

It is more convenient to go up to an $S_3$ Galois covering (of degree 6), $\xi:\Bbb C\Bbb P^1\to \Bbb C\Bbb P^1$ where the coordinate on the source $\Bbb C\Bbb P^1$ is the classical modular function $\lambda$ such that $E_\tau$ is given by the equation 
\(
y^2=x(x-1)(x-\lambda).
\)
This $\lambda$ is well defined up to the group $S_3$ generated by $\lambda\to 1-\lambda$ and $\lambda\to 1/\lambda$, and it is related to the $J$-invariant by the equation
$$
J(\tau)=\frac{4}{27}\frac{(1-\lambda+\lambda^2)^3}{\lambda^2(1-\lambda)^2}
$$
(defining the map $\xi$). Thus $\lambda$ is a modular function with respect to $\Gamma(2)$; 
it is given by 
$$
\lambda(\tau)=\frac{\theta_{10}(\tau)^4}{\theta_{00}(\tau)^4}=16\bold q+O(\bold q^2),
$$
where 
$$
\theta_{00}(\tau)=\sum_{n\in \Bbb Z}\bold q^{n^2},\ \theta_{10}(\tau)=\sum_{n\in \Bbb Z}\bold q^{(n+\frac{1}{2})^2}.
$$

In terms of $\lambda$, the Picard-Fuchs equation is
$$
\left(\partial_\lambda^2+\left(\frac{1}{\lambda}+\frac{1}{\lambda-1}\right)\partial_\lambda+\frac{1}{4\lambda(\lambda-1)}\right)F(\lambda)=0. 
$$
The inverse function of $\lambda(\tau)$ is given by 
$$
\tau=i\frac{F(\frac{1}{2},\frac{1}{2},1; 1-\lambda)}{F(\frac{1}{2},\frac{1}{2},1; \lambda)}
$$
where $F={}_2F_{1}$ is the Euler-Gauss hypergeometric function. 

Now consider $q\ne 1$. The deformed Picard-Fuchs equation \eqref{pfeq} in the $\lambda$-variable takes the form 
$$
\left(\partial_\lambda^2+\left(\frac{1}{\lambda}+\frac{1}{\lambda-1}\right)\partial_\lambda+\frac{1}{4\lambda(\lambda-1)}-\frac{s^2(\lambda^2-\lambda+1)}{\lambda^2(\lambda-1)^2}\right)F(\lambda)=0. 
$$
The points $\lambda=0,1,\infty$ all map to $\infty$ under the map $\xi$, with ramification indices $2,2,2$, so the parameters of the hypergeometric function should all 
get shifted by suitable multiples of $s$.
Indeed, the indicial exponents of the deformed equation are $\{-s,s\}$
at both $\lambda=0$ and $\lambda=1$, and
$\{-\frac12-s,-\frac12+s\}$ at $\lambda=\infty$.  For the Gauss
parameters
\(
 a=\frac12-s,\qquad b=\frac12-3s,\qquad c=1-2s,
\)
the exponent differences at $0,1,\infty$ are all $2s$.  Multiplication
of both solutions by the common gauge
$[\lambda(1-\lambda)]^{-s}$ changes the Gauss exponents precisely to
the three pairs above; the gauge cancels in their ratio.  Since a
second-order Fuchsian equation with three singular points is determined,
up to such a gauge, by its exponent data, this gives the ratio in
\eqref{tauq}.

More precisely, we should have 
\begin{equation}\label{tauq}
[\tau]_q=A_q\frac{F(\frac{1}{2}-s,\frac{1}{2}-3s,1-2s; 1-\lambda(\tau))}{F(\frac{1}{2}-s,\frac{1}{2}-3s,1-2s; \lambda(\tau))},
\end{equation}
where $A_q$ is a certain fractional linear transformation such that $A_{q=1}(z)=iz$ and the hypergeometric function 
$F$ is understood in the sense of analytic continuation. 
Namely, the function $\lambda(\tau)$ defines an isomorphism $\Bbb C_+/\Gamma(2)\cong \Bbb C\setminus \lbrace 0,1\rbrace$, i.e., is just the universal covering map $\lambda: \Bbb C_+\to \Bbb C\setminus \lbrace 0,1\rbrace$. So $F(a,b,c;\lambda(\tau))$ is just the function $F\circ \lambda:\Bbb C_+\to \Bbb C$ obtained by pulling back $F$ via $\lambda$. 
 
Recall that $\lambda(\tau)\to 0$ as ${\rm Im}\tau\to +\infty$, and 
$\lambda(\tau)\to 1$ as ${\rm Im}(-1/\tau)\to +\infty$ (as $\lambda(-1/\tau)=1-\lambda(\tau)$). 
Put $s_\varepsilon:=s+\varepsilon$, where $\varepsilon>0$.  Since
$s$ is purely imaginary and
\[
 (1-2s_\varepsilon)-\left(\tfrac12-s_\varepsilon\right)
 -\left(\tfrac12-3s_\varepsilon\right)=2s_\varepsilon,
\]
the real part of this difference is $2\varepsilon>0$, so Gauss'
formula applies at $z=1$.  We may
therefore define the parameter-regularized value
\begin{align*}
 F_s^{\mathrm{reg}}(1)
 &:={\lim}_{\varepsilon\to0^+}
 F\left(\tfrac12-s_\varepsilon,\tfrac12-3s_\varepsilon,
 1-2s_\varepsilon;1\right)\\
 &=\frac{\Gamma(1-2s)\Gamma(2s)}
 {\Gamma(\frac12-s)\Gamma(\frac12+s)}
 =i\frac{\cosh(\frac t2)}{\sinh t}
 =\frac{i}{q^{-1/2}-q^{1/2}}.
\end{align*}
For the original purely imaginary $s$, this is not the limit as
$z\to1$ at fixed $s$; it is a connection coefficient obtained by
meromorphic continuation in the parameter.  More precisely, if
\(
 a=\tfrac12-s,\qquad b=\tfrac12-3s,\qquad c=1-2s,
\)
then the connection formula at $\lambda=0$ reads
\begin{align*}
 F(a,b;c;1-\lambda)
 ={}&F_s^{\mathrm{reg}}(1)F(a,b;c;\lambda)\\
 &+B(s)\lambda^{2s}
 F(\tfrac12-s,\tfrac12+s;1+2s;\lambda),
\end{align*}
where
\[
 B(s)=\frac{\Gamma(1-2s)\Gamma(-2s)}
 {\Gamma(\frac12-s)\Gamma(\frac12-3s)}.
\]
Thus $F_s^{\mathrm{reg}}(1)$ is the coefficient of the holomorphic
local solution, rather than a boundary value of the hypergeometric
function at fixed $s$.  The cusp normalization requires the fractional
linear transformation $A_q$ to have its pole at this coefficient; the
symmetry $\lambda\mapsto1-\lambda$ requires its zero at the reciprocal
coefficient.  Hence
\[
 A_q\left(\frac{i}{q^{-1/2}-q^{1/2}}\right)=\infty,
 \qquad
 A_q\left(-i(q^{-1/2}-q^{1/2})\right)=0.
\]
Equivalently, these relations may first be obtained for
$\operatorname{Re}s>0$, where Gauss' formula is literal, and then
continued meromorphically to the purely imaginary value of $s$.
So we get 
$$
A_q(z)=a_q\frac{1-\frac{iz}{q^{-{1\over 2}}-q^{1\over 2}}}{z-\frac{i}{q^{-{1\over 2}}-q^{1\over 2}}}
$$
for some $a_q\in \Bbb C^\times$. So, using Euler transformation formulas for $F$, 
$$
[\tau]_q=a_q\frac{F(\frac{1}{2}-s,\frac{1}{2}-3s,1-2s; \lambda(\tau))-\frac{i}{q^{-{1\over 2}}-q^{1\over 2}}F(\frac{1}{2}-s,\frac{1}{2}-3s,1-2s; 1-\lambda(\tau))}{F(\frac{1}{2}-s,\frac{1}{2}-3s,1-2s; 1-\lambda(\tau))-\frac{i}{q^{-{1\over 2}}-q^{1\over 2}}F(\frac{1}{2}-s,\frac{1}{2}-3s,1-2s; \lambda(\tau))}=
$$
$$
a_q\left(\frac{1-\lambda(\tau)}{\lambda(\tau)}\right)^{2s}\frac{F(\frac{1}{2}-s,\frac{1}{2}+s,1+2s;1-\lambda(\tau))}{F(\frac{1}{2}-s,\frac{1}{2}+s,1+2s;\lambda(\tau))},
$$
i.e., $a_{q=1}=i$.  Also replacement of $\tau$ by $-\frac{1}{\tau}$ together with the formula 
$[-\frac{1}{\tau}]_q=-\frac{1}{q[\tau]_q}$ tells us that $a_q=-\frac{1}{qa_q}$. Thus we have 
$a_q=iq^{-\frac{1}{2}}=ie^{\pi is}$. So we get 
\begin{equation}\label{form1}
[\tau]_q=ie^{\pi is}\left(\frac{1-\lambda(\tau)}{\lambda(\tau)}\right)^{2s}\frac{F(\frac{1}{2}-s,\frac{1}{2}+s,1+2s;1-\lambda(\tau))}{F(\frac{1}{2}-s,\frac{1}{2}+s,1+2s;\lambda(\tau))}.
\end{equation}

In particular, for $\tau=i$ we have $\lambda(\tau)=\frac{1}{2}$, so we get 
\(
[i]_q=ie^{\pi is}=iq^{-\frac{1}{2}}.
\)
Not surprisingly, this agrees with the formula in \cite{O}, Subsection 2.3. 
Modular equivariance also implies that 
$$
q[\rho]_q+1=[1+\rho]_q=-\frac{q^{-1}}{[-\frac{1}{1+\rho}]_q}=-\frac{q^{-1}}{[\rho]_q},
$$
which yields 
\(
[\rho]_q=q^{-1}\rho, 
\)
also as in \cite{O}, Subsection 2.3. 
However, in general our definition of $[\tau]_q$ does not agree with that of 
\cite{O}, e.g. for a Gaussian integer $\tau$, the function $q\to [\tau]_q$ is typically not rational. 

Also, since 
$$
F(\tfrac{1}{2}-s,\tfrac{1}{2}+s,1+2s;1-z) = 
$$
$$
\frac{\Gamma(1+2s)\Gamma(2s)}{\Gamma(\frac{1}{2}+s)\Gamma(\frac{1}{2}+3s)}F(\tfrac{1}{2}-s,\tfrac{1}{2}+s,1-2s;z)  -\frac{i}{q^{-1/2}-q^{1/2}}
z^{2s}F(\tfrac{1}{2}+3s,\tfrac{1}{2}+s, 1+2s;z),
$$
we obtain, again using Euler transformations 

\begin{equation}\label{form2}
f(t,\tau)=[\tau]_q=iq^{-1/2}\frac{\Gamma(1+2s)\Gamma(2s)}{\Gamma(\frac{1}{2}+s)\Gamma(\frac{1}{2}+3s)}\left(\frac{1-\lambda(\tau)}{\lambda(\tau)}\right)^{2s}\frac{F(\tfrac{1}{2}-s,\tfrac{1}{2}+s,1-2s;\lambda(\tau))}{F(\tfrac{1}{2}-s,\tfrac{1}{2}+s,1+2s;\lambda(\tau))}
+\frac{1}{1-q}.
\end{equation}
So 
\begin{equation}\label{hformu}
h(t,\tau)=ie^{t/2}\frac{\Gamma(1+2s)\Gamma(2s)}{\Gamma(\frac{1}{2}+s)\Gamma(\frac{1}{2}+3s)}\left(\bold q\frac{1-\lambda_{\bold q}}{\lambda_{\bold q}}\right)^{2s}\frac{F(\tfrac{1}{2}-s,\tfrac{1}{2}+s,1-2s;\lambda_{\bold q})}{F(\tfrac{1}{2}-s,\tfrac{1}{2}+s,1+2s;\lambda_{\bold q})},
\end{equation} 
where $\lambda_{\bold q}=\lambda(\tau)$. 
This is manifestly an element of $\Bbb C[[\bold q]]$, as expected, since 
$$
\lambda_{\bold q}=16\bold q-128\bold q^2+...\in \bold q\Bbb C[[\bold q]].
$$ 
In fact, since $h$ is $1$-periodic in $\tau$, this is even an element of 
$\Bbb C[[\bold q^2]]$, although this is not obvious from the formula. 

The uniqueness used here may be summarized as follows.  Once $Q$ is
fixed, the ratio of two independent solutions of \eqref{pf} is unique
up to a constant fractional linear transformation.  The prescribed
projective local monodromies $T_t,S_t$, together with the two cusp normalizations
encoded in Assumption~1 and inversion, fix this fractional linear
freedom: the first two conditions determine the zero and pole of
$A_q$.  The equation $a_q^2=-q^{-1}$, together with the
deformation normalization $a_{q=1}=i$, fixes the remaining scalar.  Conversely, the numerator and denominator in
\eqref{form1} are independent solutions of the same equation.  Their
Wronskian never vanishes, and $\lambda:\mathbb C_+\to
\mathbb C\setminus\{0,1\}$ is unramified, so the resulting ratio is a
local diffeomorphism and satisfies Assumption~2.

Thus we obtain 

\begin{theorem} Let $f(t,\tau)$ be continuous in $(t,\tau)$ with $f(0,\tau)=\tau$.  For each $t>0$, put $q=e^{-t}$ and assume
that $f(t,\cdot)=[\cdot]_q$ is meromorphic in the upper half-plane,
satisfies
\begin{equation}\label{modeq}
[\tau+1]_q=q[\tau]_q+1,\qquad
[-\tfrac{1}{\tau}]_q=-\tfrac{1}{q[\tau]_q},
\end{equation}
and satisfies Assumptions~1 and~2.  Then $[\tau]_q$ is uniquely
determined and is given by either of the formulas \eqref{form1},
\eqref{form2}.  More precisely, these are the formulas for imaginary
$\tau$, where we use the principal branch of $F(a,b,c;z)$ on $(0,1)$,
defined by the power series
$\sum_{n\ge0}\frac{(a)_n(b)_n}{(c)_n\,n!}z^n$; the general case is
obtained by analytic continuation, as in \eqref{tauq}.
\end{theorem}

\begin{remark} Meromorphicity of hypergeometric functions with respect to parameters
implies that the function $[\tau]_q$ extends to a meromorphic function 
of $\tau\in \Bbb C_+$ and $s\in \Bbb C$. It  is a multivalued function of $q=e^{-2\pi is}$, however (i.e., only a function of $s=-\frac{\log q}{2\pi i}$), so it should better be denoted 
by $[\tau]_{(s)}$. The continuity assumption for $f$ with $f(0,\tau)=\tau$ selects the
branch $s=\frac{t}{2\pi i}$ continuous from $s=0$; the integer-shifted
branches correspond to different deformations. For example, for
$n\in \Bbb Z_{\ge 0}$
\eqref{form1} yields
\(
 F(-n,n+1,2n+2;z)
 =P_n^{2n+1,-2n-1}(1-2z),
\)
so
\begin{equation}\label{jacpol}
[\tau]_{(n+\frac{1}{2})}=(-1)^{n+1}
\left(\frac{1-\lambda(\tau)}{\lambda(\tau)}\right)^{2n+1}
\frac{P_n^{2n+1,-2n-1}(2\lambda(\tau)-1)}
{P_n^{2n+1,-2n-1}(1-2\lambda(\tau))},
\end{equation}
where $P_n^{2n+1,-2n-1}$ is the (renormalized) $n$-th Jacobi polynomial 
\[
P_{n}^{2n+1,\,-2n-1}(z)\;=\;
\sum_{k=0}^{n}
(-1)^{k}\,
\frac{(n+k)!}{(n-k)!\,k!}\,
\frac{(2n+1)!}{(2n+1+k)!}\,
\left(\frac{1-z}{2}\right)^{k}.
\]
E.g.,
\[
[\tau]_{(\frac{1}{2})}=-\frac{1-\lambda(\tau)}{\lambda(\tau)},\qquad
[\tau]_{(\frac{3}{2})}=\left(\frac{1-\lambda(\tau)}{\lambda(\tau)}\right)^3
\frac{1+\lambda(\tau)}{2-\lambda(\tau)},
\quad\text{etc.}
\]
{
Integer shifts of $s$ preserve $q=e^{-2\pi is}$ but can change the
cusp growth of the deformation.  Indeed, put
$f_s(\tau):=[\tau]_{(s)}$.  For $s=\frac{t}{2\pi i}$ with $t>0$,
\eqref{fhform} gives $f_s(iy)=O(1)$ as $y\to+\infty$, whereas
Gauss' evaluation at $1$ applied to \eqref{form1} gives
\[
 f_{s+1}(iy)\sim C_s e^{(2\pi-it)y},\qquad C_s\ne0.
\]
The cusp normalization therefore depends on the chosen branch of $s$.
}
\end{remark}

\begin{remark}
If $s\in\Bbb Q$ has denominator $m>1$, the projective monodromy group
$G\subseteq\operatorname{PGL}_2(\Bbb C)$ generated by the classes of
$S_t$ and $T_t$ is the triangle group $\Delta(2,3,m)$.
The degree over $\Bbb C(\lambda)$, however, is governed by
\(
 H:=\rho_s(\Gamma(2))\subseteq G.
\)
Here $\rho_s$ denotes the projective monodromy representation, and
$H$ is the normal closure of the class of $T_t^2$ in $G$.

When $H$ is finite, analytic continuation makes $f$ single-valued on
the finite covering of the $\lambda$-sphere corresponding to
$\ker(\rho_s|_{\Gamma(2)})$.  The local monodromies are finite, so the
regular-singular local expansions have no logarithmic part and $f$
extends meromorphically across the points over $0,1,\infty$ after
compactification.  Hence $f$ is algebraic over $\mathbb C(\lambda)$.
The following branch count therefore computes the algebraic degree,
not merely the size of an analytic monodromy orbit.

Although the order of a projective monodromy group does not determine the degree
of an arbitrary, possibly non-Galois, extension, in the present
situation it does.  Put $f=[\tau]_{(s)}$.  In the finite cases under
consideration, namely $m=2,3,4,5$, its nonconstancy follows directly
from modular equivariance. Indeed, if $f$ were constant, it would have a finite
value $z_0$, and
\(
 z_0=qz_0+1,\qquad z_0=-\frac1{qz_0}.
\)
Eliminating $z_0$ gives $q^2-q+1=0$, so $q$ would have order $6$, a
contradiction.  The branches of $f$ over the
$\lambda$-line are precisely
\(
 h\cdot f=\frac{a_hf+b_h}{c_hf+d_h},\qquad h\in H.
\)
They are pairwise distinct: if $h_1\cdot f=h_2\cdot f$ as germs, then
$h_2^{-1}h_1$ fixes the nonconstant meromorphic function $f$
identically, which is impossible for a nonidentity fractional linear
transformation.  Thus the branch orbit has cardinality $|H|$.
Moreover, every branch $h\cdot f$ already belongs to
$\Bbb C(\lambda,f)$.  Hence, when $H$ is finite, the minimal polynomial
of $f$ over $\Bbb C(\lambda)$ has $|H|$ distinct roots and splits in
$\Bbb C(\lambda,f)$; in particular, this extension is Galois with group
$H$ and has degree $|H|$.

For $m=2$ this group is
trivial, so $[\tau]_{(s)}$ is rational in $\lambda$, in agreement with
\eqref{jacpol}.  For odd $m$, $T_t$ is a power of $T_t^2$, hence
$H=G$.  Thus the generic algebraic degrees are $|A_4|=12$ for $m=3$
and $|A_5|=60$ for $m=5$.  For $m=4$, the normal closure of $T_t^2$
in $S_4$ is the Klein four group, so the generic degree is $4$.
Consequently $[\tau]_{(s)}$ is algebraic over $\Bbb C(\lambda)$ in
these cases, with generic degrees $1,12,4,60$ for
$m=2,3,4,5$, respectively.  Finally, if $m=6$ (i.e., $q$ is a sixth
root of $1$), then $G=\Delta(2,3,6)=\Bbb Z/6\ltimes\Bbb Z^2$, 
so $[\tau]_{(s)}$ expresses as a function of $\lambda$ 
in quadratures (i.e., via exponentials of integrals of algebraic functions).  

Note that unlike the real case, if $\tau\in \Bbb C_+$ is a quadratic irrational then 
$[\tau]_q$ is not, in general, a quadratic irrational over $\Bbb Q(q)$ - it is typically a transcendental function. However, in this case the elliptic curve $E_\tau$
has complex multiplication, thus $j(\tau)$ and $\lambda(\tau)$ are algebraic numbers (\cite{Si}). {
For such a quadratic irrational $\tau$, if $s\in\mathbb Q$ has
reduced denominator $m\in\{2,3,4,5\}$, then
\[
 [\tau]_{(s)}\in\mathbb P^1(\overline{\mathbb Q});
\]
in particular, every finite value is algebraic.  The case $s=0$ also
has this property, since $[\tau]_{(0)}=\tau$.

To justify the arithmetic normalization, put
\[
 F_s(z):={}_2F_1(\tfrac12-s,\tfrac12+s;1+2s;z).
\]
For $m=3,4,5$, this hypergeometric equation has finite projective
monodromy and rational local exponents.  Its determinant character
therefore has finite image.  The scalar kernel of projectivization
is consequently finite, so its linear monodromy is finite as well.
Consequently $F_s$ is algebraic over $\mathbb C(z)$.  Its Taylor
coefficients are rational, so it is algebraic over $\mathbb Q(z)$:
the coefficient equations for a polynomial relation of bounded degree
have rational entries and only finitely many linearly independent
rows.  Formula \eqref{form1}, with its algebraic prefactor and rational
powers, now gives algebraicity over
$\overline{\mathbb Q}(\lambda)$.

For $m=2$, positive half-integers are covered by \eqref{jacpol}.
At $s=-n-\tfrac12$, $n\ge0$, the parameter limit of $F_s$ has rational
Taylor coefficients, and for $j\ge2n+1$ they are
\[
 [z^j]F_s(z)=\frac{(-1)^n}{2(2n)!}
 \frac{(j+n)!(j-n-1)!}{j!(j-2n-1)!}.
\]
This is a polynomial in $j$ of degree $2n$, so $F_s$ is rational.
Thus \eqref{form1} gives a rational function of $\lambda$ for negative
half-integers as well.  Specializing at the algebraic value
$\lambda(\tau)$ proves the assertion in
$\mathbb P^1(\overline{\mathbb Q})$.

} 
\end{remark} 

\subsection{Asymptotics of $[\tau]_q$}
Since $h(t,\tau)$ is holomorphic in $\bold q$ and non-vanishing
at $\bold q=0$, it has no zeros or poles in the region $V(M)$ defined by ${\rm Im}\tau\ge M$ for some $M>0$. Let us fix such $M$. 

\begin{proposition}\label{minusin} (i) We have 
$$
[\tau]_q-\frac{1}{1-q} \sim \frac{iq^{-\frac{1}{2}}2^{\frac{4\log q}{\pi i}}\Gamma(1-\frac{\log q}{\pi i})\Gamma(-\frac{\log q}{\pi i})}{\Gamma(\frac{1}{2}-\frac{\log q}{2\pi i})\Gamma(\frac{1}{2}-\frac{3\log q}{2\pi i})}q^{\tau},\ {\rm Im}\tau\to +\infty
$$

(ii) Let $K:={\rm min}_{\tau\in V(M)}|h(t,\tau)|$ (so $K>0$).
If $\tau \in V(M)$ then 
\begin{equation}\label{firstin}
|[\tau]_q|\ge Kq^{{\rm Re}\tau}-\frac{1}{1-q}.
\end{equation}
Thus if ${\rm Re}\tau\le \frac{\log\frac{2}{K(1-q)}}{\log q}$ then 
\begin{equation}\label{secondin}
|[\tau]_q|\ge \tfrac{1}{2}Kq^{{\rm Re}\tau}.
\end{equation}
In particular, if ${\rm Re}\tau\to -\infty$ in $V(M)$ then 
$[\tau]_q\to \infty$. 

(iii) Let $L:={\rm max}_{\tau\in V(M)}|h(t,\tau)|$. If $\tau\in V(M)$ then 
$$
\left|[\tau]_q-\frac{1}{1-q}\right|\le Lq^{{\rm Re}\tau}.
$$
In particular, if ${\rm Re}\tau\to +\infty$ in $V(M)$ then 
$[\tau]_q\to \frac{1}{1-q}$. 
\end{proposition} 

\begin{remark} The minimum and maximum exist since $h$ is $1$-periodic and extends to a finite nonzero value at the cusp, so one may work on the compactified strip $0\le\operatorname{Re}\tau\le1$, $\operatorname{Im}\tau\ge M$.
\end{remark}

\begin{proof}
(i) This follows from the equality
$$
\lim_{\bold q\to 0}h(t,\tfrac{\log \bold q}{\pi i})=ie^{\pi is}\frac{2^{-8s}\Gamma(1+2s)\Gamma(2s)}{\Gamma(\frac{1}{2}+s)\Gamma(\frac{1}{2}+3s)},
$$
which is obvious from \eqref{hformu}.

(ii), (iii) follow from the definition of $K,L$ and \eqref{fhform}.
\end{proof} 

\subsection{The poles of $[\tau]_q$} 
Since $f(t,\tau+1)=qf(t,\tau)+1$, poles of $f(t,\tau)$ come in cosets of the form $\tau_0+\Bbb Z$. 

\begin{proposition} \label{poles}
The function $f(t,-)$ has infinitely many poles modulo translations by $\Bbb Z$, and all of them are simple.
\end{proposition} 

\begin{proof} Denote the standard fundamental domain of $PSL_2(\Bbb Z)$ on $\Bbb C_+$ given by 
$$
|\tau|\ge 1, -\frac{1}{2}\le {\rm Re}\tau\le \frac{1}{2}
$$ 
by $D_1$ and 
for ${\rm g}\in PSL_2(\Bbb Z)$ let $D_{\rm g}:={\rm g}D_1$. We have 
\(
f(t,{\rm g}^{-1}\tau)=\rho_q({\rm g})^{-1}f(t,\tau)
\)
where $\rho_q$ is the projective Burau representation determined by $S_t,T_t$. So if $\tau\in D_{\rm g}$ is a pole of $f$ then 
\(f(t,{\rm g}^{-1}\tau)=[x]_q\) with ${\rm g}x=\infty$. So we should look 
for points $\sigma\in D_1$ where $f(t,\sigma)=[x]_q$ 
for various $x\in \Bbb Q\cup\infty$, and then the poles of $f$ in $D_{\rm g}$ are $\tau={\rm g}\sigma$ where ${\rm g}\in PSL_2(\Bbb Z)$ maps $x$ to $\infty$. 

Explicitly, by \eqref{form1} the poles of $f(t,\tau)$ are zeros 
of the function $F(\frac{1}{2}-s,\frac{1}{2}+s,1+2s;\lambda(\tau))$, which are images under elements ${\rm g}_x\in PSL_2(\Bbb Z)$ mapping $x$ to $\infty$ of 
solutions $\sigma$ of the equation 
\begin{equation}\label{eqpole}
ie^{t/2}q^\sigma\frac{\Gamma(1+2s)\Gamma(2s)}{\Gamma(\frac{1}{2}+s)\Gamma(\frac{1}{2}+3s)}\left(\bold q\frac{1-\lambda_{\bold q}}{\lambda_{\bold q}}\right)^{2s}\frac{F(\tfrac{1}{2}-s,\tfrac{1}{2}+s,1-2s;\lambda_{\bold q})}{F(\tfrac{1}{2}-s,\tfrac{1}{2}+s,1+2s;\lambda_{\bold q})}=[x]_q-\frac{1}{1-q},\ x\in \Bbb Q
\end{equation}
in the standard fundamental domain $E$ of $\Gamma(2)$ defined by 
$$
|\tau\pm \tfrac{1}{2}|\ge \tfrac{1}{2},\ -1\le {\rm Re}(\tau)\le 1.
$$ 
Note that
\[
\left|\frac{\Gamma(\frac{1}{2}+s)\Gamma(\frac{1}{2}+3s)}
{\Gamma(1+2s)\Gamma(2s)}\right|
=\frac{|\sin 2\pi s|}{\sqrt{\cos \pi s\cos 3\pi s}}
=\frac{q^{-1/2}-q^{1/2}}{\sqrt{q+q^{-1}-1}}.
\]
Thus the left hand side of \eqref{eqpole} takes any value 
of  the form $re^{i\theta}$, $0\le \theta<2\pi$, with 
$$
\frac{q\sqrt{q+q^{-1}-1}}{1-q}<r<\frac{q^{-1} \sqrt{q+q^{-1}-1}}{1-q}
$$
arbitrarily close to the cusp at infinity in $E$. 

Indeed, let $h_\infty\ne0$ be the
cusp value of $h(t,\sigma)$.  Uniformly on vertical translates of a
compact disk in the strip $-1<\operatorname{Re}\sigma<1$, the left
side of \eqref{eqpole} has the form
\[
 q^\sigma h(t,\sigma)=h_\infty q^\sigma(1+o(1))
 \qquad(\operatorname{Im}\sigma\to\infty).
\]
For every $w$ in the open annulus
$|h_\infty|q<|w|<|h_\infty|q^{-1}$, the equation
$h_\infty q^\sigma=w$ has infinitely many solutions $\sigma_k$ with
$-1<\operatorname{Re}\sigma_k<1$ and
$\operatorname{Im}\sigma_k\to\infty$.  Its derivative at each such
solution is $(\log q)w\ne0$.  On a fixed sufficiently small circle
about $\sigma_k$, the error term tends uniformly to zero; Rouch\'e's
theorem therefore gives a zero of
$q^\sigma h(t,\sigma)-w$ inside that circle for every sufficiently
large $k$.  The modulus of $h_\infty$ is
$\sqrt{q+q^{-1}-1}/(1-q)$, which gives exactly the two displayed radii.

Hence if  
$$
\frac{1-q^{-1}\sqrt{q+q^{-1}-1}}{1-q}<[x]_q<\frac{1-q\sqrt{q+q^{-1}-1}}{1-q}
$$
then there is an infinite string of solutions $\sigma_{n,x}$ convergent to this cusp.
Since the lower bound here is less than $[0]_q^-=1-q^{-1}$ (and goes to $-1$ as $q\to 1$), while the upper bound is greater than $[0]_q=0$ (and goes to $1$ as $q\to 1$), we do get infinitely many solutions, e.g. ones corresponding to small $x$. Asymptotically as $q\to 1$, all $x\in (-1,1)$ 
will contribute. 

Finally, at a pole $\tau_0$ use $1/f$ as the local affine coordinate on
$\mathbb{CP}^1$ near $\infty$.  Assumption~2 says that
$d(1/f)_{\tau_0}\ne0$, so $1/f$ has a simple zero and $f$ has a simple
pole.
\end{proof} 

\begin{remark} By modular symmetry, the zeros of $f$ are related to its poles by the map $\tau\mapsto -1/\tau$, and all of them are simple. 
\end{remark}

\subsection{Boundary values of $[\tau]_q$ at rational points.} 
Let us now discuss the boundary values of $f(t,\tau)$. 
Let $x\in \Bbb Q$, and suppose that a sequence $\tau_n\in \Bbb C_+$ 
converges to $x$. 

\begin{definition} We say that $\tau_n$ 
converges to $x$ from the left (respectively, right) if for any (or, equivalently, some) ${\rm g}\in PSL_2(\Bbb Z)$ such that ${\rm g}\infty=x$, we have ${\rm Re}({\rm g}^{-1}\tau_n)\to +\infty$ (respectively, ${\rm Re}({\rm g}^{-1}\tau_n)\to -\infty$).

We say that $\tau_n$ is $M$-bounded if ${\rm g}^{-1}\tau_n\in V(M)$ for all $n$. 
\end{definition} 

The left/right part of the definition is unchanged if
$PSL_2(\mathbb Z)$ is replaced by $PSL_2(\mathbb R)$.  For
$M$-boundedness, however, the numerical value of $M$ depends on the
choice of a real matrix: two real matrices sending $\infty$ to $x$
differ on the right by an upper triangular matrix, and this rescales
the imaginary part by a positive constant.  Thus, with
$PSL_2(\mathbb R)$, only the property of being $M$-bounded for some
$M>0$ is choice-independent; the definition with a fixed numerical
$M$ uses $PSL_2(\mathbb Z)$.

\begin{lemma} 
$\tau_n\to x$ from the left (right) iff 
$\frac{{\rm Re}(\tau_n-x)}{|\tau_n-x|^2}\to -\infty$
(respectively, $\frac{{\rm Re}(\tau_n-x)}{|\tau_n-x|^2}\to +\infty$).  
It is $M$-bounded iff $\frac{{\rm Im}(\tau_n-x)}{|\tau_n-x|^2}\ge Mb^2$, where 
$b$ is the denominator of $x$ in the lowest terms representation. 
\end{lemma}

\begin{proof}
Write $x=a/b$ in lowest terms and choose
$g=\begin{pmatrix}a&c\\ b&d\end{pmatrix}$ with $ad-bc=1$.  If
$\tau=x+w$, then
\[
 g^{-1}\tau=\frac{-d\tau+c}{b\tau-a}
 =-\frac db-\frac{1}{b^2w}.
\]
Consequently,
\[
 \operatorname{Re}(g^{-1}\tau)
 =-\frac db-\frac{\operatorname{Re}w}{b^2|w|^2},
 \qquad
 \operatorname{Im}(g^{-1}\tau)
 =\frac{\operatorname{Im}w}{b^2|w|^2}.
\]
These identities give the three asserted equivalences.  They are
independent of the chosen $g$: two matrices sending $\infty$ to $x$
differ on the right by an integral translation, which changes only
$\operatorname{Re}(g^{-1}\tau)$ by a bounded integer and leaves its
imaginary part unchanged.
\end{proof}

The modular equivariance property \eqref{modeq}, Assumption 1 and Proposition \ref{minusin}(ii) imply 

\begin{proposition}\label{ratbou} Let $x\in \Bbb Q$ and suppose $\tau_n\to x$ is $M$-bounded. 

(i) If $\tau_n\to x$ from the right then $[\tau_n]_q\to [x]_q$.

(ii) If $\tau_n\to x$ from the left then 
$[\tau_n]_q\to [x]_q^-$.
\end{proposition}

\begin{proof}
Choose $g\in PSL_2(\mathbb Z)$ with $g\infty=x$ and put
$\sigma_n=g^{-1}\tau_n$.  Modular equivariance gives
$[\tau_n]_q=\rho_q(g)[\sigma_n]_q$.

If the approach is from the right, then
$\operatorname{Re}\sigma_n\to-\infty$ while
$\operatorname{Im}\sigma_n\ge M$.  Proposition \ref{minusin}(ii)
therefore gives $[\sigma_n]_q\to\infty$ in $\mathbb{CP}^1$, and hence
\(
 [\tau_n]_q\longrightarrow\rho_q(g)\infty=[x]_q,
\)
the last equality being the modular-equivariant definition of the
$q$-rational number.

If the approach is from the left, Proposition \ref{minusin}(iii)
gives $[\sigma_n]_q\to1/(1-q)$.  To identify its image, take real
$u\to+\infty$.  Then $g u\to x$ from the left and
$[u]_q\to1/(1-q)$, so modular equivariance and Proposition
\ref{ratcon} imply
\[
 \rho_q(g)\frac1{1-q}
 =\lim_{u\to+\infty}[g u]_q=[x]_q^-.
\]
This proves both assertions.
\end{proof}

\subsection{Boundary values of $[\tau]_q$ at irrational points.}

The behavior of $[\tau]_q$ near irrational real numbers is more complicated,
due to its rapid oscillation when approaching the real line (as seen, for
example, from Proposition \ref{poles}). To quantify the convergence in
Proposition \ref{ratbou}(i), we first fix the matrix carrying $\infty$ to a
rational boundary point.

\begin{lemma}\label{estt}
Let $x=[[c_1,\ldots,c_N]]=a/b>1$ be written in lowest terms, with $b>0$,
and let
\[
 g=\begin{pmatrix}a&c\\ b&d\end{pmatrix}\in SL_2(\Bbb Z),
 \qquad 0\le d<b.
\]
Put
\[
 B:=b_N(q),\qquad
 \beta:=\frac{q^{c_N-1}b_{N-1}(q)}{b_N(q)}.
\]
Then $B\ge1$ and $0\le\beta<1$. If $\tau=g\sigma$ and
$Z=[\sigma]_q$ is finite, then, whenever the denominator on the right
is nonzero,
\begin{equation}\label{boundaryexact}
 |[\tau]_q-[x]_q|
 =\frac{q^{C_N}}{B^2}
 \begin{cases}
 |Z|^{-1},&N=1,\\[2pt]
 |qZ+1-\beta|^{-1},&N\ge2.
 \end{cases}
\end{equation}
\end{lemma}

\begin{proof}
Write $L_c(z)=c-1/z$ and $g_0=L_{c_1}\cdots L_{c_N}$. The ordinary
continuant identities give
\[
 g_0=\begin{pmatrix}
 a_N(1)&-a_{N-1}(1)\\ b_N(1)&-b_{N-1}(1)
 \end{pmatrix}.
\]
If $N=1$, then $b_N(1)=1$ and $b_{N-1}(1)=0$, so $g=g_0$.
If $N\ge2$, then $0<b_{N-1}(1)<b_N(1)$, so the prescribed normalization
is $g=g_0T$, where $T(z)=z+1$.

A matrix representing $\rho_q(g_0)$ is
\[
 M_N=\begin{pmatrix}
 a_N(q)&-q^{c_N-1}a_{N-1}(q)\\
 b_N(q)&-q^{c_N-1}b_{N-1}(q)
 \end{pmatrix},
 \qquad \det M_N=q^{C_N}.
\]
Consequently, for finite $W\ne\beta$,
\[
 M_NW-\frac{a_N(q)}{b_N(q)}
 =-\frac{q^{C_N}}{B^2(W-\beta)}.
\]
Modular equivariance gives $[\tau]_q=M_NZ$ if $N=1$, and
$[\tau]_q=M_N(qZ+1)$ if $N\ge2$. This proves
\eqref{boundaryexact}. Finally, $B$ and
$B-q^{c_N-1}b_{N-1}(q)$ have nonnegative coefficients and constant term
$1$, by the continuant recurrence. Thus $B\ge1$ and $0\le\beta<1$.
\end{proof}

Fix $0<\alpha<1$ and put $\varepsilon_0:=1-q^{1-\alpha}$, so that
$0<\varepsilon_0<1-q$. Retain the constants $M,K$ from Proposition
\ref{minusin}. For the normalized matrix in Lemma \ref{estt}, we have
\begin{equation}\label{boundarysigma}
 \sigma=g^{-1}\tau=-\frac db-\frac{1}{b^2(\tau-x)},
 \qquad
 \begin{aligned}
 \operatorname{Re}\sigma
 &=-\frac db-\frac{\operatorname{Re}(\tau-x)}{b^2|\tau-x|^2},\\
 \operatorname{Im}\sigma
 &=\frac{\operatorname{Im}(\tau-x)}{b^2|\tau-x|^2}.
 \end{aligned}
\end{equation}

\begin{proposition}\label{estt1}
With the notation of Lemma \ref{estt}, suppose that
\begin{equation}\label{circ}
 \operatorname{Re}\frac{\tau-x}{|\tau-x|^2}>
 b^2\frac{\log(K\varepsilon_0/2)}{\log q}+b(b-d),
 \qquad
 \operatorname{Im}\frac{\tau-x}{|\tau-x|^2}>Mb^2.
\end{equation}
Then
\begin{equation}\label{boundarybound}
 |[\tau]_q-[x]_q|<\varepsilon_0 q^{C_N}.
\end{equation}
\end{proposition}

\begin{proof}
By \eqref{boundarysigma}, the second inequality in \eqref{circ} gives
$\operatorname{Im}\sigma>M$, while the first gives
\[
 \operatorname{Re}\sigma<
 -1-\frac{\log(K\varepsilon_0/2)}{\log q}
 =\frac{\log(2/(Kq\varepsilon_0))}{\log q}.
\]
As $0<q<1$, it follows that
$Kq^{\operatorname{Re}\sigma}>2/(q\varepsilon_0)$. The stronger
cusp estimate \eqref{firstin} therefore yields
\begin{equation}\label{boundaryZbound}
 |Z|>\frac{2}{q\varepsilon_0}-\frac{1}{1-q}.
\end{equation}
If $N\ge2$, then $0\le\beta<1$ and hence
\[
 |qZ+1-\beta|
 \ge q|Z|-(1-\beta)
 >\frac{2}{\varepsilon_0}-\frac{q}{1-q}-1
 =\frac{2}{\varepsilon_0}-\frac{1}{1-q}
 >\frac{1}{\varepsilon_0}.
\]
If $N=1$, then \eqref{boundaryZbound} and
$q\varepsilon_0<1-q$ give
$|Z|>1/(q\varepsilon_0)>1/\varepsilon_0$.
In either case \eqref{boundarybound} now follows from
\eqref{boundaryexact} and $B\ge1$.
\end{proof}

Now let $V_x=V_x(M,\alpha)$ be the set defined by the two inequalities
in \eqref{circ}. Each boundary is a circle through $x$ (or a line in
the degenerate case), and the two boundaries meet orthogonally at $x$.
Depending on the sign of the first constant in \eqref{circ}, the first
inequality selects the interior or the exterior of its circle; no sign
assumption is needed in the argument. Using translations by $\Bbb Z$,
we may extend this definition to all rational $x$.

The following theorem is an analog of Proposition \ref{ratbou}(i) to
the case of irrational $x$.

\begin{theorem}
Let $x\in\Bbb R\setminus\Bbb Q$, and let $\tau_n\in\Bbb C_+$,
$n\ge1$. Suppose that for each $n$ there exists $x<y_n\in\Bbb Q$
such that $\tau_n\in V_{y_n}$, and $y_n\to x$. Then
$\tau_n\to x$ and $[\tau_n]_q\to[x]_q$.
\end{theorem}

\begin{proof}
By translation $\Bbb Z$-equivariance we may assume that $x>1$.
Since $y_n\to x\notin\Bbb Q$, their reduced denominators $b_n$ tend
to infinity. The second inequality defining $V_{y_n}$ implies
\[
 |\tau_n-y_n|<\frac{1}{Mb_n^2},
\]
so $\tau_n\to x$.

Write $y_n=[[c_{n,1},\ldots,c_{n,N_n}]]$ and
$C_{N_n}=\sum_{j=1}^{N_n}(c_{n,j}-1)$. For each fixed bound on
$C_{N_n}$, there are only finitely many possible digit lists, since
every digit is at least $2$. Thus $C_{N_n}\to\infty$, as otherwise
a subsequence of $y_n$ would be constant and its limit rational.
Proposition \ref{estt1} therefore gives
$[\tau_n]_q-[y_n]_q\to0$. Continuity at irrational points, from
Proposition \ref{ratcon}, gives $[y_n]_q\to[x]_q$, proving the result.
\end{proof}

In a similar way, using Proposition \ref{minusin}(iii), one can prove
an analog of Proposition \ref{ratbou}(ii) for irrational $x$, but we
will not do it here.

 \end{document}